\documentclass[a4paper,12pt,english]{amsart}

\usepackage{amsmath,upref,amssymb}
\usepackage{pdfsync}
\usepackage{tensor}

\swapnumbers

\input xypic
\xyoption{all}
\xyoption{arc}
\xyoption{2cell}
\CompileMatrices

\RequirePackage{calrsfs}
\DeclareSymbolFont{rsfscript}{OMS}{rsfs}{m}{b}
\DeclareSymbolFontAlphabet{\mathrsfs}{rsfscript}
\renewcommand{\mathcal}{\mathrsfs}

\newcommand{\nc}{\newcommand}

\nc{\on}{\operatorname}

\nc{\DMO}{\DeclareMathOperator}
\numberwithin{equation}{subsection}

\theoremstyle{plain}
\newtheorem*{thm*}{Theorem}
\newtheorem*{lem*}{Lemma}
\newtheorem*{prop*}{Proposition}
\newtheorem*{cor*}{Corollary}

\theoremstyle{definition}
\newtheorem*{defn*}{Definition} 
\newtheorem*{conj*}{Conjecture}
\newtheorem*{quest*}{Question}
\newtheorem*{exmp*}{Example}
\newtheorem*{ass*}{Assumption}

\theoremstyle{remark}
\newtheorem*{rem*}{Remarks}
\newtheorem*{note*}{Note}
\newtheorem*{case*}{Case}

\makeatletter
\renewcommand{\p@enumi}{\thesubsection}
\renewcommand{\p@enumii}{\thesubsection\theenumi}
\makeatother

\newenvironment{num}{\renewcommand{\theenumi}{(\alph{enumi})}
                      
                                          \begin{enumerate} }
                    {\end{enumerate} }
 \newenvironment{conds}{\renewcommand{\theenumi}{(\roman{enumi})}
                       
                        \begin{enumerate} }
                     {\end{enumerate} }

\nc{\be}{\begin{equation}}
\nc{\ee}{\end{equation}}

\nc{\bee}{\begin{equation*}}
\nc{\eee}{\end{equation*}}

\nc{\bs}{\begin{split}}
\nc{\es}{\end{split}}

\nc{\bc}{\begin{cases}}
\nc{\ec}{\end{cases}}

\nc{\bml}{\begin{multline}}
\nc{\eml}{\end{multline}}

\nc{\bmll}{\begin{multline*}}
\nc{\emll}{\end{multline*}}

\nc{\lb}[1]{\label{#1}\mar{#1}}

\nc{\belb}[1]{\mar{#1}\begin{equation}\label{#1}}

\newcommand{\mpair}[1]{\pair{\,#1\,}}

\newcommand{\mset}[1]{\set{\,#1\,}}

\newcommand{\pair}[1]{\langle #1\rangle}

\newcommand{\set}[1]{\{#1\}}

\nc{\mc}{{\mathcal}}
\nc{\mf}{{\mathfrak}}
\nc{\CA}{{\mathcal A}}
\nc{\CB}{{\mathcal B}}
\nc{\CC}{{\mathcal C}}
\nc{\CD}{{\mathcal D}}
\nc{\CE}{{\mathcal E}}
\nc{\CF}{{\mathcal F}}
\nc{\CG}{{\mathcal G}}
\nc{\CH}{{\mathcal H}}
\nc{\CI}{{\mathcal I }}
\nc{\CJ}{{\mathcal J }}
\nc{\CK}{{\mathcal K }}
\nc{\CL}{{\mathcal L}}
\nc{\CM}{{\mathcal M}}
\nc{\CN}{{\mathcal N}}
\nc{\CO}{{\mathcal O}}
\nc{\CP}{{\mathcal P}}
\nc{\CQ}{{\mathcal Q}}
\nc{\CR}{{\mathcal R}}
\nc{\CS}{{\mathcal S}}
\nc{\CT}{{\mathcal T}}
\nc{\CU}{{\mathcal U}}
\nc{\CV}{{\mathcal V}}
\nc{\CW}{{\mathcal W}}
\nc{\CX}{{\mathcal X}}
\nc{\CY}{{\mathcal Y}}
\nc{\CZ}{{\mathcal Z}}
\nc{\bbZ}{{\mathbb Z}}
\nc{\bbC}{{\mathbb C}}
\nc{\bbR}{{\mathbb R}}
\nc{\bbP}{{\mathbb P}}
\nc{\bbF}{{\mathbb F}}
\nc{\bbN}{{\mathbb N}}
\nc{\bbD}{{\mathbb D}}
\nc{\bbO}{{\mathbb O}}
\nc{\bbQ}{{\mathbb Q}}

\def\a{\alpha}
\def\b{\beta}
\def\g{\gamma}
\def\G{\Gamma}

\def\D{\Delta}
\def\e{\epsilon}

\def\l{\lambda}

\def\L{\Lambda}

\def\o{\omega}

\def\th{\theta}
\def\Th{\Theta}

\def\to{\rightarrow}

\def\lsupp#1#2{\kern\scriptspace\vphantom{#2}^{#1}\kern-\scriptspace#2}

\def\lsubb#1#2{\kern\scriptspace\vphantom{#2}_{#1}\kern-\scriptspace#2}

\def\lsub#1#2{\tensor*[_{#1}]{#2}{}}

\def\lrsub#1#2#3{\tensor*[_{#1}]{#2}{_{#3}}}

\newcommand{\ssect}{\subsection}

\newcommand{\seq}{{\,\subseteq\,}}

\newcommand{\sneq}{{\,\subsetneq\,}}

\newcommand{\sreq}{{\,\supseteq\,}}

\newcommand{\sm}{{\,\setminus\,}}

\newcommand{\eset}{{\emptyset}}

\nc{\wt}{\widetilde}

\nc{\wh}{\widehat}

\nc{\ol}{\overline}

\def\dotcup{\hskip1mm\dot{\cup}\hskip1mm}

\newcommand{\op}{^{{\rm op}}}

\newcommand{\mar}{\marginpar}

\newcommand{\join}{\bigvee}

\newcommand{\meet}{\bigwedge}

\DMO{\Supp}{{\mathrm{Supp}}}

\DeclareMathOperator{\Aut}{{\mathrm{Aut}}}

\DMO{\Rad}{{\mathrm{Rad}}}

\DMO{\Ann}{{\mathrm{Ann}}}

\nc{\Id}{\mathrm{Id}}

\DMO {\Add}{{\mathrm Add}}

\DMO {\add}{{\mathrm add}}

\DMO {\Img}{{\mathrm Im}}

\DMO {\coim}{{\mathrm Coim}}

\DMO {\coker}{{\mathrm Coker}}

\DMO {\colim}{\varinjlim}

\DMO {\plim}{\varprojlim}

\DMO {\mEnd}{{\mathrm End}}

\DMO {\mend}{{\mathrm end}}

\DMO {\Proj}{{\mathrm Proj}}

\DMO {\Ext}{{\mathrm Ext}}

\DMO {\ext}{{\mathrm ext}}

\DMO {\tor}{{\mathrm tor}}

\DMO {\Tor}{{\mathrm Tor}}

\DMO {\Hom}{{\mathrm Hom}}

\DMO {\HOM}{{\mathrm HOM}}

\DMO {\Modfg}{{\mathrm -Modfg}}

\DMO {\modulfg}{{\mathrm -modfg}}

\DMO {\modul}{{\mathrm -mod}}

\DMO {\Mod}{{\mathrm -Mod}}

\DMO {\pdim}{{\mathrm proj.dim.}}

\DMO {\PD}{{{\mathrm Proj.Dim.}}}

\DMO {\gldim}{{\mathrm gr.gl.dim.\ }}

\DMO {\grad}{{\mathrm rad }}

\DMO {\Sheaves}{{\mathrm Sh}}

\DMO {\Flab}{{\mathrm Fl}}

\DMO {\Poinc}{{\mathrm Poinc}}

\DMO {\Groth}{{{{\mathrm K}}_0}}

\DMO {\Mat}{{\mathrm Mat}}

\DMO \Sl{{\mathrm sl}}

\DMO \SL{{\mathrm SL}}

\DMO \Gl{{\mathrm gl}}

\DMO \GL{{\mathrm GL}}

\DMO \lcm{\mathrm{lcm}}

\DMO \rank{{\mathrm rank}}

\DMO \diag{{\mathrm diag}}

\nc{\bib}{\bibitem}

\nc{\pa}{\partial}

\DMO {\Addp}{{\mathrm Add}'}

\DMO {\addp}{{\mathrm add}'}

\DMO {\hatGr}{{\widehat{{\mathrm K}}_0}}

\DMO {\NExt}{{\mathrm NExt}}

\DMO {\nExt}{{\mathrm next}}

\DMO {\DExt}{{\mathrm DExt}}

\DMO {\dext}{{\mathrm dext}}

\newcommand{\Int}{{\mathbb Z}}
\newcommand{\Nat}{{\mathbb N}}
\newcommand{\real}{{\mathbb R}}

\newcommand{\simquo}{\negthinspace\negthinspace\sim\thinspace}

\DMO{\ob}{ob}
\DMO{\mor}{mor}
\DMO{\tr}{tr}
\DMO{\spec}{spec}

\newcommand{\ab}{\mathfrak{A}}

\newcommand{\gp}{\hspace{.12em}}

\DMO{\cov}{cov}
\DMO{\Sym}{Sym}
\DMO{\Dih}{Dih}
\DMO{\Spec}{Spec}
\DMO{\domn}{dom}
\DMO{\cod}{cod}
\DMO{\ord}{ord}
\newcommand{\rec}{\text{\rm rec}}

\newcommand{\nat}{\natural}
\newcommand{\cp}{\complement}

\newcommand{\gpdpreord}{\mathbf{Gpd\text{\bf -}PreOrd}}

\newcommand{\fprc}{\mathbf{FPrd}}

\newcommand{\rlec}{\mathbf{RdE}}

\newcommand{\setc}{\mathbf{Set}}

\newcommand{\catc}{\mathbf{Cat}}

\newcommand{\grpdc}{\mathbf{Gpd}}

\newcommand{\bringc}{\mathbf{BRng}}
\newcommand{\balgc}{\mathbf{BAlg}}

\newcommand{\abcatc}{\mathbf{AbCat}}

\newcommand{\prootc}{\mathbf{Prd}}
\newcommand{\prootcp}{\mathbf{Prd'}}

\newcommand{\rootc}{\mathbf{rd}}
\newcommand{\rtd}{\mathbf{Rd}}
\newcommand{\posetc}{\mathbf{Ord}}
\newcommand{\csl}{\mathbf{CSL}}
\newcommand{\preordc}{\mathbf{PreOrd}}

\newcommand{\setproot}{\setc\text{\bf -}\prootc}
\newcommand{\grpdset}{\grpdc\text{\bf -}\setc_{\pm}}

\begin{document}

\title{Groupoids,  root systems   and weak order I}

\author{Matthew Dyer}
\address{Department of Mathematics 
\\ 255 Hurley Building\\ University of Notre Dame \\
Notre Dame, Indiana 46556, U.S.A.}
\email{dyer.1@nd.edu}

\begin{abstract} This is  the first of a series of  papers concerned with 
certain structures called rootoids and protorootoids, the definition of which is   abstracted from  formal 
properties of   Coxeter groups 
with their root systems   and weak orders. A protorootoid is
 defined to be a groupoid (i.e. a small category with every morphism 
 invertible)  equipped with a  representation in the 
category of Boolean rings and a corresponding  $1$-cocycle.
Weak preorders of protorootoids are preorders on  the sets of  their morphisms with fixed codomain, defined by the natural ordering  of the corresponding values of the  cocycle in the  Boolean ring in which they lie. A  
rootoid is  a protorootoid satisfying axioms which imply  that its weak  preorders are partial orders embeddable as order ideals of complete ortholattices. 
Rootoids and protorootoids may 
  be    studied alternatively  as groupoids  with abstract root systems.
The most novel results of these papers involve    categories
  of rootoids with  favorable 
properties including   existence of small limits and  of functor rootoids, which provide interesting new structures even for 
 finite Weyl groups such as the  symmetric
 groups.
    This  first introductory paper  gives  basic definitions and  terminology, and explains     the correspondence between  cocycles and  abstract  root systems. 
        Examples discussed here  include rootoids  attached to Coxeter groups, to real  simplicial hyperplane 
   arrangements and to orthogonal groups. Rudimentary  properties of  subclasses of rootoids called principal rootoids and complete rootoids are established. Principal  (respectively, complete principal) rootoids naturally generalize arbitrary (respectively finite) Coxeter systems.   
 \end{abstract}

\maketitle
%
%
%
\section*{Introduction}\label{s0}
\subsection*{Introduction to this work}  Coxeter groups are an important class of discrete groups 
which occur naturally in many areas of mathematics. 
A Coxeter group $W$  has  a   presentation as a group generated by  a set $S$ of  involutory  simple generators subject to  braid relations, which specify  the order of the product of each pair of simple generators. The presentation is  determined by   a square matrix, the Coxeter matrix, with  its entries in $\Nat_{\geq 1}\cup\set{\infty}$ equal to the orders of these products. In applications related to Lie theory,  $W$  often occurs in conjunction with root systems, which are combinatorial  
structures on which $W$  acts.   
Classically,   root systems $\Phi$ are subsets of real vector spaces which encode  a description of an  action of  $W$ 
 as a finite or  discrete real reflection group.
   More abstract 
 notions of root  systems   are also useful, for example
 in the theory of buildings.  Examples of 
  Coxeter groups which are  obviously important  are the finite and affine Weyl groups, and especially the symmetric groups $S_{n}$, which arise 
  as  the Weyl groups of the special linear groups
   and Lie algebras. 
    
 A root system $\Phi$ of a Coxeter group $W$  may be viewed 
 abstractly as   a $W$-set (the elements of which are called 
 roots) equipped with a commuting free action by the group 
 $\set{\pm 1}$ and a distinguished set $\Phi_{+}$ of
  $\set{\pm 1}$-orbit representatives, called the set of positive 
  roots.   There is a bijection $\a\mapsto s_{\a}$ from the set of positive roots
   $\Phi_{+} $ to the set $T=\mset{wsw^{-1}\mid w\in W,s\in S}$ of reflections.
   For any $w\in W$, define 
   $\Phi_{w}:=\Phi_{+}\cap w(-\Phi_{+})$ and $N(w):=\mset{s_{\a}\mid \a\in \Phi_{w}}$.   Let $\wp(T)$ be the Boolean ring of subsets of $T$. As a function $N\colon W\to \wp(T)$, $N$ is called the reflection cocycle of $W$ (see  \cite{DyRef}). It is well known that $N(w)=\mset{t\in T\mid l(tw)<l(w)}$ where $l(w)=\min(\mset{n\in \Nat\mid w=s_{1}\cdots s_{n}, s_{i}\in S})$ is the standard length of $w$.   
   
        In the case of $S_{n}$, its simple generators are the adjacent transpositions $(i,i+1)$ for $ i=1,\ldots, n-1$, the   reflections are the transpositions
  $(i,j)$ for $ i\neq j$ and  $N(w):=\mset{(i,j)\mid i<j, w^{-1}(i)>w^{-1}(j)}$ identifies with the set of inversions of $w^{-1}$. The root system $\Phi$ may be described as follows. Let $S_{n}$ act in the natural way on $\real^{n}$ permuting coordinates, and let $\mset{e_{1},\ldots, e_{n}}$ be the standard ordered orthonormal basis. Then $\Phi:=\mset{e_{i}-e_{j}\mid i\neq j}$ with natural action by $W\times \set{\pm1}$ and with $\Phi_{+}:=\mset{e_{i}-e_{j}\mid i< j}$. One has $\Phi_{w}=\mset{e_{i}-e_{j}\mid i<j, w^{-1}(i)>w^{-1}(j)}$ and $s_{e_{i}-e_{j}}=(i,j)\in T$.

     The weak right order  of $W$ is the 
 partial order $\leq$ on the set $W$ defined by the condition
  $x\leq y$ if $\Phi_{x}\seq \Phi_{y}$ (equivalently, $N(x)\seq N(y)$ or $l(y)=l(x)+l(x^{-1}y))$.
 Weak right  order is important in the basic combinatorics of $W$. For example,   reduced expressions (that is expressions, $w=s_{1}\cdots s_{n}$ with each $s_{i}\in S$ and $n=l(w)$)
  correspond bijectively  to
  maximal chains in the  order ideal of weak  right order generated by $w$. Weak right  order  is a complete meet semilattice, which is a complete lattice if and only if  $W$ is finite (see \cite{BjBr}).

 To explain the framework adopted in  these papers, we discuss one of the  main results in the context of Coxeter groups.  Recall that 
  a groupoid is a small category $G$ in which every morphism is invertible.  Let $G$ be the groupoid with one object $\bullet$, with the Coxeter group  $W$ as  automorphism group of $\bullet$. Let  $H$ be a non-empty  groupoid, assumed connected (i.e. with at least one  morphism between any two of its objects) for simplicity. The  functor category $G^{H}$ has  as  objects the  functors $H\to G$, and as morphisms the natural transformations  of such functors. It is itself   a groupoid. It has  a subgroupoid $K=G^{H}_{\square}$  containing  all objects of $G^{H}$, but only those morphisms $\nu\colon  F\to F'$ (where $F,F'\colon H\to G$ are functors)  such that for each morphism  $h\colon a\to b$ in $H$, \begin{equation*} \nu_{b}(\Phi_{F(h)})=\Phi_{F'(h)}\end{equation*}  where $\nu_{b}\colon F(b)\to F'(b)$ is the component of $\nu$ at $b$  (note that $\nu_{b}, F(h), F'(h)\in \Hom_{G}(\bullet,\bullet)=W$).
The groupoids studied in \cite{BrHowNorm} are covering quotients of a  very restricted  class of components of groupoids  $G^{H}_{\square}$ (with their objects indexed by subsets of $S$ instead of subsets of $W$) in  cases in which $H$ is connected and simply connected i.e. has a unique morphism between any two of its objects. 
   
  In general, the groupoids $K$  have  properties closely 
  analogous  to those  of Coxeter groups mentioned above. For example, $K$ 
    has complete semilattice weak right orders, it has a canonical  braid presentation specified  by families of Coxeter matrices (which here give 
    the lengths of the joins of pairs of simple generators) with    additional 
    combinatorial data, it has an abstract root system, and  for  a connected groupoid $L$, one can form analogously  $K^{L}_{\square}$ 
     with similar properties, and so on recursively.   Analogous facts  hold for the groupoids of \cite{BrHowNorm}, where existence of  canonical 
      presentations was the main result and the other properties  just mentioned were not established,  for 
  the Coxeter groupoids (with root systems) of \cite{HY}, and for other   natural classes of groupoids. The (components of) groupoids in 
  \cite{BrHowNorm} are not in general Coxeter groupoids, though they turn out to be so if $W$ is finite.
    
These papers  study a class of structures,  which we call rootoids, which provides a natural framework within which  the results  mentioned in the previous paragraph, and many  others, can be   stated precisely  and proved. 
The notion of rootoid  is obtained by  abstracting as an axiom   the   fact 
        that the weak right order of a Coxeter group is a  complete meet    semilattice. (Note that the weak right order of a Coxeter group 
     is not translation invariant, and that the theory of rootoids is  quite 
     distinct from the theory of ordered groups and groupoids). 
     The notion of rootoid is far more general than that of Coxeter group, but       special subclasses of 
     rootoids have properties very similar to certain basic properties of Coxeter groups. The results stated above are proved by showing that rootoids
 are preserved by certain  natural categorical constructions (for example, the category of rootoids has all  small limits) which also preserve the conditions defining these subclasses.

  Before discussing  rootoids, the  closely related notion of  signed groupoid-sets, of 
     which $(W,\Phi)$ is the archetypal example, will be described.  
The definition  involves    a category of signed sets, which  has as  objects sets $S$ with a specified free action by $\set{\pm 1}$ and a  
  specified set $S_{+}$ of $\set{\pm 1}$-orbit 
  representatives, and as   morphisms the
   $\set{\pm 1}$-equivariant functions. A 
     signed groupoid-set    is a pair $R=(G,\Psi)$ of a groupoid $G$ and a representation  $\Psi$ of $G$ in  the category of signed sets. For $g\colon a\to b$ in $G$, define $\Psi_{g}:=\Psi(b)_{+}\cap \Psi(g)(-\Psi(a)_{+})$.
    The set $\lsub{b}{G}$ of  morphisms of $G$ with codomain $b$ is preordered by  $g\leq h$ if $\Psi_{g}\seq \Psi_{h}$. This gives one weak right preorder of $R$ for each object $b$ of $G$.
    Then  $R$ is said to be rootoidal if each weak right  preorder is a complete meet semilattice satisfying a   technical condition  called JOP  which implies that the weak right order  embeds as an  order ideal   in a complete ortholattice.     One says that $R$ is principal if $G$ has a set of generators (called simple generators)  such that the minimal length of each morphism $g$ as a product of these generators is equal to $\vert \Psi_{g}\vert $. 
    
     The above construction of  the groupoid  
    $K=G^{H}_{\square}$ can be adapted to the signed groupoid-set $R=(G,\Psi)$.  An abstract root system of $K$ is constructed as follows. First, there is a  signed groupoid-set  $(K,\L)$ where $\L$ is obtained  by pullback of the functor  $\Psi$ 
    under the evaluation homomorphism $\e_{b}\colon K=G^{H}_{\square}\to G$ for some chosen object $b$ of $H$. 
    One forms a new signed groupoid-set  $(K,\L^{\rec})$
    where the real compression $\L^{\rec}$ 
    is a sub-quotient  representation of $\L$. It is  obtained, roughly speaking,  by discarding any imaginary roots of $\L$ (defined as those $\a\in \L(a)$ such that $\a$ and $\L(g)(\a)$ have the same sign (positive or negative)
    for all morphisms $g$ of $G$ with domain $a$)  and identifying any two  remaining roots $\a,\b\in \L(a)$ which are equivalent in the equivalence relation associated to  dominance preorder i.e.  such that  for all $g$ with domain $a$,  $\L(g)(\a)$ and $\L(g)(\b)$ have the same sign. 
    
    The  results about  $K$ referred to above    in the case of Coxeter groups now follow from three facts. First, a Coxeter group gives rise to a principal,  rootoidal, 
    signed groupoid-set $(W,\Phi)$. Secondly, if  $R=(G,\Psi)$ above  is any  principal, rootoidal, signed groupoid-set, then $(K,\L^{\rec})$
    is also  a  principal, rootoidal signed groupoid-set. Finally, the underlying groupoid of  any  principal, rootoidal  signed groupoid-set  has a braid presentation.     

   Rootoidal signed groupoid-sets provide one  framework in which our main results may be formulated, and are  particularly natural in certain cases (e.g related to Coxeter groups) where they can be realized concretely in real vector spaces. However, they seem less 
 convenient in more abstract examples   and for  general categorical arguments. For these reasons,    many results and arguments will be expressed   in terms of  analogues of the reflection cocycle $N$ of $W$, especially since the  concept of cocycle   is already well established in  general contexts  and is very convenient for calculation. 
    This approach leads to the     notions of protorootoid and rootoid (explained later in the introduction) 
    which  will be used   for our general   development.  
    
      The previous discussion  illustrates the three primary goals of these papers:    to give examples of rootoids (of various classes), to describe general properties of rootoids (of various classes) and to study categories of rootoids (of various classes) and especially to establish their closure under certain natural categorical constructions. The third goal involves the most novel results, but the others are necessary to make these new results useful,  and these aims will be pursued together in these papers. 
         This first paper  gives  basic definitions, terminology and 
         examples. The second  paper will give   more examples of rootoids 
         and morphisms of rootoids, and discusses
  (without proof, and partly informally)  some   ideas, results and 
  questions  from subsequent papers of the series. 
These  two papers together give the most rudimentary part of the theory, an  
introduction
to the main ideas and results of subsequent papers  and  indications of   ingredients of some of the main
 proofs. The rest of the theory, and the   proofs, will appear in the later  papers of the series.

    The original motivations for this work were in relation to the set of 
    initial sections  of reflection orders of a Coxeter group $W$, which 
 has  significant applications in the combinatorics of Kazhdan-Lusztig 
 polynomials and  Bruhat order, and conjectural relevance to 
 associated representation theory. Ordered by inclusion, the set of initial sections is a poset in 
 which the weak order   of $W$ embeds naturally as an order ideal.  
 Many of the main ideas  of this work originated in the study in   
    \cite{DyWeak} of a longstanding  conjecture which is  equivalent to the 
    statement that the set of initial sections is a complete ortholattice.   
      Other papers which have suggested useful ideas include    
 \cite{Muhl}, \cite{DyEmb1}, \cite{DyEmb2},
       \cite{BjEZ},  \cite{CH}, \cite{HV},  \cite{HY}, and
  especially      \cite{BrHowNorm}.  
      Though the theory described in  this paper and its sequels has been largely 
   motivated by study of  Coxeter groups, the most novel aspects of  the work  involve   the interaction of 
    several partly independent  formalisms  of a quite general nature  and it  seems unlikely that rootoids  and protorootoids provide the  natural setting for all these formalisms. 
  Limitations of space and energy  preclude  detailed discussion of all the   known     generalizations and    variations, but many  are  indicated in the remarks.

         Finally, we mention that
 formulations and proofs of many of  the main results in these 
 papers require
 only    knowledge of rather limited basic aspects of   algebra, 
 combinatorics of ordered sets and category 
 theory as can be found in, for example, \cite{Lang},
  \cite{Davpr} and \cite{MacL}.    More  background knowledge,
 and especially familiarity with Coxeter groups,   is 
  assumed for the examples and applications.  
    As general references on Coxeter groups, 
    see  \cite{Bour}, \cite{Hum} 
    and  \cite{BjBr}. 
     
 \subsection*{Introduction to this paper}      
     The concerns of this first paper are as follows.  
       Section \ref{s1} gives  background and fixes frequently used   general
       notation. Section \ref{s2} defines  protorootoids,  which provide a  
       convenient general context for this work.         A protorootoid is a   
       triple $\CR:=(G,\L,N)$ of a groupoid $G$, a
representation $\L$ of $G$  in the category of Boolean rings (i.e. a 
functor $\L\colon \G\to\bringc$) and a 
  cocycle $N\in Z^{1}(G,\L)$ for the  representation of $G$  underlying
   $\L$ in $\Int\modul$.  Concretely, $N$ is a function which  associates 
   to each morphism $g$ of $G$, with codomain $a$, say,  an element 
   $N(g)$ of the Boolean ring $\L(a)$, satisfying the cocycle condition 
   $N(gg')=N(g)+gN(g')$ where we abbreviate $gN(g'):=(\L(g))(N(g'))$.
  The motivating examples
are  the protorootoids  $\CC_{(W,S)}:=(W,\L,N)$ attached to  Coxeter 
systems $(W,S)$. In this,
 $\L$ denotes   the conjugacy representation of $W$ on the Boolean 
 ring 
$\wp(T)$ of subsets of  the reflections $T$, and 
$N\colon W\to \wp(T)$ is the reflection 
cocycle of $(W,S)$. 

Associated to  the  
protorootoid $\CR$ is a family of weak right preorders  indexed by 
 the objects of  $G$; the set  $\lsub{a}{G}$ of morphisms of $G$ with 
 fixed codomain $a$ is 
preordered by $x\leq y$ if $N(x)\leq N(y)$ in the natural partial order of 
the Boolean ring $\L(a)$ (which is defined by $r\leq r'$ if $r=rr'$). There is a natural category of protorootoids 
 in which, in particular,  morphisms induce preorder preserving maps of 
 associated  weak right preorders. Every protorootoid $(G,\L,N)$ has an underlying groupoid-preorder, forgetting $\L$ and $N$ but remembering the groupoid $G$ and its weak right preorders; subsequent papers will show that
  an alternative approach to much of the theory of protorootoids and rootoids can be developed in terms of groupoid-preorders.
 
  Important subclasses of protorootoids are    defined 
 in Section \ref{s3}. These include the faithful ones (for which weak right 
 preorders are partial orders),  complete faithful ones (the weak right 
 orders of which are  
 complete lattices), the faithful,  interval finite ones (for which intervals 
 in weak right order are finite, and for which the set of atoms of weak 
 right orders is  a groupoid generating set)  and the principal ones (for 
 which the 
 underlying groupoid  has  a set of simple generators 
 with respect to which the length of any morphism $g\colon b\to a$ is equal to the rank in $\L(a)$ of $N(g)$ i.e the length of a maximal chain from $0$ to $N(g)$).   The class of preprincipal    protorootoids is defined      by requiring, in   
    an interval finite, faithful protorootoid,   a natural  generalization of a length-compatibility condition from   \cite{Muhl} (expressed as a cocycle compatibility condition in \cite{DyEmb1}, \cite{DyEmb2}) which requires  that    $l_{S}(w's')=l_{S}(w')\pm l_{S}(s')$  for all $w\in W'$, $s'\in S'$ for special  embeddings  $W'\seq  W$ of  Coxeter groups, where $S$  and  $S'$ are the  sets of   simple generators of $W$ and  $W'$ respectively.

   The abridgement of a protorootoid $\CR=(G,\L,N)$ is defined by
   $\CR^{a}:=(G,\L',N')$ where $\L'$ is the $G$-subrepresentation of
    $\L$
    (in the category of Boolean rings) generated by the values of
     the 
    cocycle $N$, and $N'$ is the evident restriction  of $N$. 
 It gives a minimal version of $\CR$ preserving $G$ 
    and its weak  right   preorders.  Abridgement is an analogue for protorootoids of real compression, but has better functorial properties. A technically important fact proved here is that a protorootoid is preprincipal if and only if its abridgement is principal. 
     
     In Section \ref{s4}, a   rootoid is defined to be  a faithful
       protorootoid in which the weak 
 orders  are complete meet semilattices 
  which satisfy a  condition called JOP
 (join orthogonality property).  To explain JOP, say that two morphisms 
 $x,y$ with common codomain $a$ are orthogonal if 
 $N(x) N(y)=0$ in $\L(a)$ (or $x^{-1}\leq x^{-1 }y$, in terms of 
 weak  right preorder). Then JOP requires that if each of a non-empty 
 family of morphisms $(x_{i})$ is orthogonal to $y$ and  $(x_{i})$ has a 
 join (least upper bound) $x$, then $x$ is orthogonal to $y$. 
 
  Section \ref{s4} also contains  the definition of the category of rootoids.
Consider two rootoids and a morphism of protorootoids 
between them such that each induced 
map $\theta$ of weak orders  preserves all meets (greatest lower 
bounds) and joins of 
non-empty sets 
which exist in its domain. Each  $\th$ then has a partially 
defined (categorical) left adjoint $\th^{\perp}$, defined on the
 order ideal generated by the image of $\th$. The morphism is 
 defined to 
be a morphism of rootoids if it satisfies the following adjunction 
orthogonality property (AOP): for all $\th$, one has that $\th(x)$ and $y$ (in the order 
ideal) are 
orthogonal if and only if  $x$ and $\th^{\perp}(y)$ are orthogonal
(in this, ``if'' holds automatically). 
A   subcategory of the category of rootoids which is important in relation to functor rootoids  is that of rootoid local embeddings, in which morphisms are restricted so  that the weak orders of the domain embed as join-closed meet subsemilattices of the weak orders of the codomain.

     The  principal rootoids have many basic properties in common with 
  Coxeter groups and their study is one main focus of these papers
  (study of complete rootoids is another).  
   However, in categorical contexts it is technically  more convenient to 
   work with preprincipal rootoids.  The JOP 
   ensures that
     several categorical constructions with protorootoids  preserve rootoids  and 
     that the morphisms involved   induce semilattice   homomorphisms 
     between corresponding weak right orders. The AOP   ensures further that  many such  constructions 
     preserve preprincipal rootoids; composing such constructions with 
     abridgement gives constructions preserving principal rootoids. 
         Though we do not emphasize it in the introductory papers, many results of a categorical nature about principal rootoids hold for a larger class (of  regular, saturated, pseudoprincipal rootoids) in which the principal rootoids are distinguished as the interval finite members.
         See Section \ref{s3} for the definitions, which are motivated partly  by conjectures on initial sections of reflection orders of Coxeter groups.
                
     Any signed groupoid-set $(G,\Psi)$ gives rise to a protorootoid $(G,\wp_{G}(\Psi/\set{\pm 1}),M)$ where the $1$-cocycle $M\in Z^{1}(G,\wp_{G}(\Psi/\set{\pm 1})) $ classifies $\Psi$ as a principal bundle with fiber $\set{\pm 1}$ over a suitably defined  quotient representation $\Psi/\set{\pm 1}$ of $\Psi$,
and $\wp_{G}$ is a power-set functor  from $G$-sets to $G$-Boolean rings.  
The triple $(G,\Psi/\set{\pm 1},M)$ is called a set protorootoid.
In Section \ref{sa5},   categories of set protorootoids and 
    signed groupoid-sets  are  defined.
It is shown that  the construction above induces an equivalence between these  two categories, and the category of set    protorootoids is trivially equivalent to  a  subcategory of the    category of protorootoids.    These facts
 will be  supplemented by a  result in subsequent papers
 showing that 
 protorootoids have representations as set protorootoids (in a 
 similar  manner as Boolean rings have representations as rings of 
 sets, by Stone's theorem). Together (see \ref{ssa5.8}), these results make it possible to restate  many of the results of these papers  in terms of  signed groupoid-sets and then to obtain stronger  analogues of some of them  for   special classes of signed groupoid-sets which have  realizations in real vector spaces (there are several natural  notions of realization, of   different strength and generality). 

 Section \ref{sa6}  recalls some background on Coxeter groups, and discusses     basic examples of rootoids. The reader may wish to read it  in parallel with earlier sections. The
 protorootoid $\CC_{(W,S)}$ of a Coxeter system  is shown   to be a 
 principal rootoid, complete if and only if  $W$ is finite. Two proofs are given; the more self-contained one involves 
  a  useful    semilocal criterion (SLC)  for an interval finite, faithful protorootoid to be a rootoid.
  Protorootoids are also  defined in Section  \ref{sa6} from finite, real, central hyperplane 
  arrangements and shown using \cite{BjEZ} to be rootoids if and only if 
   the arrangement is simplicial, in which 
  case the rootoid is complete and principal.  Similarly, Coxeter 
  groupoids and Weyl groupoids with root systems in the 
  sense of \cite{CH} and \cite{HY} give rise to principal rootoids 
  (complete if and only if  finite, in the 
  case of a connected Coxeter groupoid) and  protorootoids naturally  attached to  
  simplicial oriented geometries (see \cite{BjEZ}) are rootoids  if and 
  only if  the oriented geometry is 
  simplicial (in which case the rootoid is complete and principal). Discussion of these and most other important  examples is deferred to  subsequent papers, though  for diversity, rootoids  with  the additive group of real numbers or   a compact real orthogonal group as underlying groupoid are also described  in Section \ref{sa6}.    

Finally in this paper,   a brief   Section \ref{sa7}  shows that each weak order of  a complete rootoid has  a longest element   (analogous   to that  of a finite  Coxeter group) making it a complete ortholattice. These longest elements are shown to  give rise to an involutory automorphism of the rootoid (the analogue of the diagram automorphism given by conjugation by the longest element of a finite Coxeter group). 
  
  \section{Background, notation  and conventions}
 \label{s1}
           \subsection{Ordered sets}\label{ss1.1} The following 
            terminology and notation  for posets,  
            lattices and semilattices will be 
           used  throughout this series of  papers. 
     
  A \emph{poset} (i.e. partially ordered set)  $(\CL,\leq)$
   will usually be denoted just as 
  $\CL$  (note that  the empty poset $\eset$ is permitted).
 An \emph{order ideal} of a poset $(\L,\leq)$ is a subset $\G\seq \L$ 
 such that  if   
  $x,y\in \L$ satisfy $x\leq y\in \G$, then $x\in \G$.
   \emph{Order coideals} $\G$ are 
  defined dually: if $x\geq y\in \G$, then $x\in \G$. Closed  
  intervals in $\L$ are denoted  as $[x,y]=[x,y]_{\L}:=
  \mset{z\in \L\mid x\leq z\leq y}$. An 
  element $x$ is an \emph{atom} of $\L$ if $\L$ has a
   minimum element, say $m$, and 
  $\vert [m,x]\vert =2$ (generally,  the cardinality of a set $A$ is
   denoted  by
   $\vert A\vert$).   \emph{Coatoms} are defined dually. 
     
    The category of posets and order preserving maps  is denoted
     $\posetc$.  
    It is a full, reflective subcategory  of the category $\preordc$ of 
    preordered sets and preorder preserving maps (by definition, a \emph{reflective} (resp., \emph{coreflective}) subcategory of a category  is one such that the corresponding inclusion functor  has a left (resp., right) adjoint).      Recall that 
 a    \emph{preordered set} $(L,\leq)$ is a set $L$ equipped with a 
 reflexive and 
    transitive relation $\leq$.  The left adjoint $\preordc\to \posetc$ sends 
        the  preordered set $(L,\leq)$ to its associated poset, which is 
        defined as follows. The associated equivalence relation 
    $\sim$ of  $(L,\leq) $ is the relation on $L$ given by $a\sim b$ if
     $a\leq b$ and
     $b\leq a$. Denote the equivalence class of $a\in L$ as $[a]_{L}$ or
      $[a]$ for short.
    Then the associated  poset $(L/\simquo, \preceq)$
     of $(L,\leq)$ has  $L/\simquo:=\mset{[a]\mid a\in L}$ and
    $[a]\preceq [b]$ if and only if  $a\leq b$.  
    
    A preordered set (or partial order) is \emph{directed} if all of its finite subsets have an upper bound (in particular, it is non-empty). A subset of a preorder is directed if it is directed in the induced preorder. A subset $X$ of a preordered set  (or poset) $(L,\leq)$ is \emph{cofinal} if for any $y\in L$, there  exists $x\in X$ with $y\leq x$.

  In any partially ordered set,  the \emph{meet}
   or greatest lower bound of  a subset $\G$ (resp., the \emph{join} or 
   least 
   upper bound of the subset $\G$) is  denoted by $\meet \G$
    (resp., $\join \G$)   if it exists. Notions of meets and joins are  defined  similarly for preordered sets (but   are only determined up to equivalence when they exist).
   Write $\meet\set{\g_{1},\ldots, \g_{n}}=\g_{1}\wedge
    \cdots \wedge \g_{n}$ and 
   $\join\set{\g_{1},\ldots, \g_{n}}=\g_{1}\vee\cdots\vee \g_{n}$ for joins 
   and
    meets of finite sets. Notation such as  $\meet_{i}x_{i},\wedge x_{i}$ 
    and $\join_{i}x_{i}, \vee x_{i}$ will be 
    used for meets and joins of indexed families $\set{x_{i}}_{i\in I}$.
  A \emph{complete lattice} $\CL$ is a non-empty partially ordered set 
  in which every   subset has a meet and join; in particular, $\CL$ has a 
  minimum element $0_{\CL}$ and a maximum element $1_{\CL}$. A 
  \emph{complete ortholattice} is a complete lattice $\CL$ equipped 
  with a given order reversing map $x\mapsto x^{\cp}\colon \CL\to \CL$, called \emph{orthocomplementation}, 
  such that for all  $x\in \CL$,  $(x^{\cp})^{\cp}=x$,
   $x\wedge x^{\cp }=0_{\CL}$ and $x\vee x^{\cp}=1_{\CL}$.  
  
  A \emph{complete meet semilattice} is a possibly empty, partially 
  ordered set $\CL$ in which any non-empty subset $Z$ has a meet 
   $\meet Z$.  This implies that any 
  subset $Z$ of $\CL$ which is bounded above in $\CL$ has a  join
    $\join Z$ in $\CL$ (namely, the meet in $\CL$ of the set of upper bounds of $Z$ in $\CL$), and that $\CL$ has a minimum element if it is non-empty.
     A \emph{complete join-closed  
  meet subsemilattice} $Z$ of  a complete meet  semilattice $\CL$ is 
  defined to be  a subset $Z$ of $\CL$ which is closed under
   formation of arbitrary meets in $\CL$ of  non-empty subsets
    of $Z$, and such that for any  non-empty  subset of $Z$ with
     a join in $\CL$, that join is an element of $Z$. Such a subset $Z$ is
      itself a complete meet semilattice, with meets (and those joins 
      which exist) of its non-empty subsets 
 coinciding with the meets (and joins) of those subsets  in $\CL$.  
 Note that it is  not required that the minimum element of a non-empty
 complete join-closed meet subsemilattice of $\CL$  must coincide with
  the minimum element of $\CL$. Dually,   \emph{complete join semilattices} and their  \emph{complete
meet-closed join subsemilattices} are defined.

 \subsection{Boolean rings}\label{ss1.2} Some of the  following  
 terminology is slightly  non-standard (cf. \cite{Davpr}), 
 but it is convenient for our purposes here. An element of a monoid
  or ring is \emph{idempotent} if $x^{2}=x$.
A \emph{Boolean ring} is a (possibly non-unital)  ring in which every 
element is idempotent. A Boolean ring $B$ is commutative and 
satisfies $x+x=0$ for all $x\in B$.  Also, $B$ has a partial order $\leq$
 defined by $x\leq y$ if $xy=x$, for $x,y\in B$.  It will usually be more
  convenient to  denote this partial order as $\seq$, the multiplication
   as $(x,y)\mapsto x\cap y$ and the addition as $(x,y)\mapsto x+y$.
    The join $x\vee y$ of two elements $x,y$ of $B$ always exists; it is
     given by $x\vee y=x+y+xy$ and  will be denoted as 
     $x\cup y$. The meet of $x$ and $y$ exists and  is their product 
     $x\cap y$.
The additive identity  $0$ of $B$ may   be denoted  
as  $\eset$.  A unital Boolean ring $B$ is called a
 \emph{Boolean algebra}; in any such Boolean algebra, the 
 \emph{complement} $x^{\cp}$ of $x\in B$ is defined by
$x^{\cp}:=1_{B}+x$. This 
notation is in accord with the standard notation 
for the Boolean algebra  $\wp (A)$  arising as the power set of a set 
$A$, where  $\seq$ is inclusion, $\cap$ is intersection, $+$ is 
symmetric difference $x+y=(x\cup y)\sm (x\cap y)$, $x^{\cp}=A\sm x$, 
and $\cup$ is union. It is partially justified in general by the existence 
(by Stone's theorem) of an isomorphism of any Boolean ring with a 
subring of some Boolean algebra $\wp (A)$.
The field $\bbF_{2}$ of two elements is an important example of a Boolean algebra.
\begin{rem*} The use of $\cup, \cap, \seq$ for Boolean rings enables us to reserve
$\vee, \wedge,\leq$ mostly for weak (right) orders.\end{rem*}
 \subsection{}\label{ss1.3} The category $\bringc$ of Boolean  rings has 
 Boolean rings as objects, and its  morphisms are   ring   
homomorphisms, with usual composition.   The category $\balgc$ of 
Boolean algebras is the  subcategory of $\bringc $ with Boolean 
algebras as objects, and ring homomorphisms which  preserve  identity 
elements as morphisms.

There is  a contravariant power set functor
 $\wp \colon  \setc \to \bringc$ such that for any set $A$, $\wp (A)$ is 
 the power set of $A$, regarded as Boolean ring as above and for 
 $f\colon A\to A'$, $\wp (f)\colon \wp (A')\to \wp (A)$ is the map
$x\mapsto f^{-1}(x):=\mset{a\in A\mid f(a)\in x}$ for $x\seq A'$, which is a 
homomorphism of Boolean rings.  The functor  $\wp$
 factors through the inclusion $\balgc\to\bringc$.

 \subsection{Categories, morphisms, stars}\label{ss1.4} For any small 
 category  $G$, the sets of objects and morphisms of $G$
  are denoted as
  $\ob(G)$ and $\mor(G)$ respectively.  The domain (resp., codomain) of a morphism $\a$ in $G$ is denoted $\domn(\a)$ (resp., $\cod(\a)$). For any $a,b\in \ob(G)$,
   define    $\lrsub{a}{G}{b}:=\Hom_{G} (b,a)$, 
the \emph{left star} $\lsub{a}{G}:=\dotcup_{b\in \ob(G)}\lrsub{a}{G}
{b}$ of $G$ at $a$  and the \emph{right star} $G_{b}:=
\dotcup_{a\in \ob(G)}\lrsub{a}{G}{b}$ of $G$ at $b$. In these equations 
and elsewhere, the symbol ``$\dotcup$'' emphasizes that a union is  
one of pairwise disjoint sets. In a Boolean ring, 
$ z=\dotcup_{i=1}^{n} z_{i}$ would mean similarly that $z$
 is the join of $z_{1},\ldots, z_{n}$ where $z_{i}\cap z_{j}=\eset$ 
 (i.e. $z_{i}z_{j}=0$) for $i\neq j$.  To simplify terminology, left stars 
 will be   simply called  stars when they are referred to 
 other than notationally. 
 
 For any subset $X$ of $\mor(G)$ and for   all $a,b\in\ob(G)$, set
 $\lsub{a}{X}:=X\cap \lsub{a}{G}$, $X_{b}:=X\cap G_{b}$ and 
 $\lrsub{a}{X}{b}:=X\cap \lrsub{a}{G}{b}$. If a product $g_{1}\ldots g_{n}$ of morphisms $g_{i}$ of $G$ appears in a formula (such as $g=g_{1}\ldots g_{n}$)
 it is tacitly required that the $g_{i}$ are such that the composite is defined.  Occasionally, the notation  $\exists g_{1}\cdots g_{n}$ is used to indicate    that a composite $g_{1}\dots g_{n}$ is defined.
 
 For two categories $G,C$ with $G$ small, let $C^{G}$ denote the category of functors $G\to C$, with natural transformations between such functors as morphisms.

  \subsection{Slice category}\label{ss1.5} Given a category $G$ and 
  object $a\in \ob(G)$, there is a \emph{slice category} $G/a$ with 
  objects the morphisms $f\colon b\to a$ in $G$. A morphism 
  $F\colon f\to f'$ in $G/a$, where $f'\colon b'\to a$ in $G$, is a 
  morphism  $F\colon b\to b'$ in $G$ such that $f=f'F$, and 
  composition is induced by composition in $G$. There is a  functor 
  $G/a\to G$ given on objects by $f\mapsto b$ and  on morphisms
   by mapping  $F\colon f\to f'$ in $G/a$ to $F\colon b\to b'$ in $G$, with
    notation as above.  Dually, define the \emph{coslice category} $a\backslash G$. 
 
  \subsection{Opposites}\label{ss1.6} The \emph{opposite}  
  of a category (resp., preorder, poset, ring etc)  $G$ is denoted as  $G^{\op}$. 
 \subsection{Representations}\label{ss1.7} A \emph{representation} of a
  (small) category $G$ in a category $X$ is a functor $F\colon G\to X$. Two representations are equivalent if they are naturally isomorphic as functors. 
   A concrete category  is a  category $X$ equipped with a faithful 
   functor $U\colon X\to \setc $ where $\setc $ is the category of sets
    and functions; examples include $\bringc $  and $\balgc$ 
    (with $U$ the underlying set functor) and $\setc $ (with $U$ 
    the identity functor).  For certain well-known 
      concrete categories,
such as those above, we will often not distinguish notationally  between  an object or morphism of $X$, and its image under $U$, as is customary.

For a representation $F$ of a category $G$ in a concrete 
category $X$,  $U(F(x))$ may be abbreviated as $\lsub{x}{F}$ 
for $x\in \ob(G)$, and
$(U(F(g)))(a)\in \lsub{y}{F}$ may be abbreviated as $ga$ or $g(a)$ for $g\colon x\to y$ 
in $\mor(G)$ and $a\in\lsub{x}{F}$. It will often be tacitly assumed, replacing $F$ by an equivalent functor 
if necessary
that the sets $\lsub{x}{F}$  for $x\in \ob(G)$ are pairwise disjoint. 
By a \emph{subrepresentation} of $F$, we shall mean a representation $F'$ of $G$ in $X$ such that there is a natural transformation $\mu\colon F'\to F$ such that 
all functions $U(\mu_{x})\colon UF'(x)\to UF(x)$ for $x\in \ob(G)$ are inclusion maps. Define a  \emph{trivial representation}  to be a representation that is equivalent to a constant functor.

 \subsection{Groupoids}\label{ss1.8}  A groupoid  $G$ is a 
 small category in which every 
morphism is invertible.  The contravariant self-equivalence 
$\iota_{G}\colon G
\to G$ of the groupoid $G$ induced by inversion  will be denoted as 
$x\mapsto x^{*}$ for short. Thus, $x^{*}:=x$ if $x\in \ob(G)$ and 
$x^{*}:=x^{-1}$ 
if $x\in \mor(G)$. The automorphism groups 
$\lrsub{a}{G}{a}=\Aut_{G}(a)$ of 
objects $a$ of $G$ are called the \emph{vertex groups} of $G$. 
The identity morphism of $G$ at $a$ will be denoted $\lrsub{}{1}{a}$ or 
$\lsub{a}{1}$.
The morphisms $g$ of $G$ which are in some (necessarily 
unique)  vertex group (i.e. such that $\exists gg$) are called \emph{self-composable 
morphisms}. A groupoid is empty 
if its sets of objects and morphisms are empty, and trivial if it has 
one object 
and one morphism. Groups will be regarded in the usual way or 
as groupoids with one object, as convenience dictates.

Given a representation $\L\colon G\to \setc $ of a groupoid $G$ in the
 category of sets, there is a corresponding  representation
 $\wp _{G}(\L)\colon G\to \bringc$ of $G$ in the category of 
Boolean rings,  defined by
$\wp _{G}(\L):=\wp  \L  \iota_{G}$,.  This has  a subrepresentation
$\wp' _{G}(\L)$ such that for each $a\in \ob(G)$,
 $\lsub{a}{(\wp' _{G}(\L))}$ is the subring of 
 $\lsub{a}{(\wp _{G}(\L))}$ consisting of all finite subsets. Note that  $\wp_{G}$ may be regarded as a contravariant functor  $\wp_{G}\colon \setc^{G}\to \bringc^{G}$: if $\nu\colon \L\to \G$ is a natural transformation between functors  $\L,\G\colon G\to \setc$,
 then $\wp_{G}(\nu)\colon \wp_{G}(\G)\to \wp_{G}(\L)$
 is the natural transformation with component at $a\in \ob(G)$ given by $(\wp _{G}(\nu))_{a}=\wp( \nu_{a})\colon \wp (\G(a))\to \wp(\L(a))$.

 \subsection{Connectedness}\label{ss1.9}  A non-empty groupoid $G$ 
  is said to be {\em connected} (resp., {\em simply connected}) if 
    there is at 
least one (resp., at most one) morphism between any two of its objects.
By convention, the empty groupoid is simply connected but not 
connected (note that some  sources regard  the empty groupoid as  connected). A 
connected, simply connected groupoid  $G$ is   determined up to 
isomorphism by the cardinality  $\vert \ob(G)\vert$ of the set  of its 
objects,
 which can be any non-zero cardinal.
 
  \subsection{The category of groupoids}\label{ss1.10} 
  Let $\catc$  denote the category of small categories  with functors as morphisms, with usual composition of functors.  For categories  $G,H$,  $\Hom_{\catc}(G,H)$  is the set of objects of a category, in which a morphism
   $\nu\colon F\to F'$  is a  natural transformation
    $\nu\colon F\to F'$ of functors $G\to H$. For a morphism $\nu'\colon F'\to F''$, the composite $\nu'\nu$ has components
$(\nu'\nu)_{a}=\nu'_{a}\nu_{a}$  for all $a\in \ob(G)$.    
 For our purposes, there is little need for, and we do not adopt, $2$-categorical language, but basic facts (such as the interchange law) about composition of functors and natural transformation will be used without comment.
  
  The category
   $\grpdc$ of groupoids  is the full subcategory of  $\catc$ which
      has groupoids as objects. 
  For groupoids $G,H$,  the category  $\Hom_{\grpdc}(G,H)$ is
   a groupoid; the inverse $\nu^{*}\colon F'\to F$ of a natural transformation $\nu\colon F\to F'$, where $F,F'\colon G\to H$, has components
   $(\nu^{*})_{a}=\nu^{*}_{a}:=(\nu_{a})^{*}$.
  Both $\catc$ and $\grpdc$ are complete and cocomplete i.e. have all 
  limits and colimits of functors from small categories. 
  See \cite{Hig} and \cite{Br} as general references on groupoids, 
  though our terminology and notation differ from theirs.

 \subsection{Components}\label{ss1.11} A \emph{subgroupoid} of  a 
 groupoid $G$ is a subcategory $H$ of $G$ which contains
the inverse in $G$ of any morphism of $H$.  A non-empty subgroupoid
 $H$ of $G$ is called a \emph{component} of $G$ if it is maximal 
 connected subgroupoid  of $G$ (i.e. if $\mor(H)$  is  maximal under  inclusion amongst  
 morphism sets of   connected subgroupoids of $G$). A component is 
 full as a subcategory.   The set of morphisms (resp., objects) of $G$ is 
 the disjoint union of the  sets of morphisms (resp., objects) of its 
 components For an object or morphism $x$ of $G$, $G[x]$ denotes the 
 component of $G$ containing $x$.
 The study of any groupoid largely reduces to that of its components.

 \subsection{Coverings} \label{ss1.12} 
 A morphism  $\pi\colon H\rightarrow G$ in $\grpdc$  is called  a 
{\em covering morphism} or \emph{covering} of $G$ if  the induced maps
 $\lsub{a}{\pi}\colon \lsub{a}{H}\xrightarrow{\cong} \lsub{\pi(a)}{G}$
  given by $h\mapsto \a(h)$  are  bijections for all $a\in \ob(H)$. For 
  such a covering,  each component of $H$ is mapped into a 
  component of $G$; the restriction to a morphism between two  
  components is surjective on both objects and morphisms. For 
  example, the natural embedding  of a connected component 
 $H$ of $G$ as a subgroupoid of $G$ is a covering morphism. 
 For a fixed groupoid $G$,  the category of \emph{covering transformations} or \emph{coverings} of $G$ is the 
 subcategory of  the slice category $\grpdc/G$ with objects the  
 coverings of $G$, and  only those morphisms  which are isomorphisms.  Morphisms in the category of coverings of $G$ are 
 called \emph{covering  transformations}. If a  covering $\pi$ as above is surjective on objects, we call $G$ a \emph{covering quotient} of  $H$ and $\pi$ a \emph{covering quotient morphism}.

  By definition, a {\em universal covering groupoid} of a  groupoid $G$ 
  is a groupoid $H$  equipped with a  covering morphism
   $\pi\colon H\to G$,
  called the \emph{universal covering morphism}, such that that $H$ is 
  simply connected and the  natural map from components of $H$ to 
  components of $G$ induced by $\pi$ is bijective.  The universal 
   covering morphism exists and  is determined up to isomorphism as an object  of 
   the category of covering morphisms of $G$.   It is a covering quotient morphism.

  The construction of the universal covering $\pi\colon H\to G$ in 
  general readily reduces  to the case $G$ is  connected, in which case, 
  $\pi$  may be identified with the 
 natural morphism $\pi\colon G/a\to G$ for any $a\in \ob(G)$.  
 See \cite{Hig}  for basic facts about coverings.

 \subsection{Groupoid $1$-cocycles}\label{ss1.13}
  Let $\Psi\colon G\to \Int\modul$ be a fixed representation of a groupoid $G$
   in  the category of abelian groups.   A \emph{$1$-cocycle} $N$ of $G$ for $\Psi$
is a family of maps $\lsub{a}{N}\colon \lsub{a}{G}\to \lsub{a}{\Psi}$ for 
$a\in \ob(G)$ satisfying the \emph{cocycle condition}
\begin{equation} \label{eq1.13.1}
\lsub{a}{N}(gh)=\lsub{a}{N}(g)+\L(g)( \lsub{b}{N}(h))\end{equation}
for all $a,b\in \ob(G)$, $g\in \lrsub{a}{G}{b}$ and $h\in\lsub{b}{G}$.
Unless confusion is likely,  $N$ is regarded as a function
$\mor(G)\to \cup_{a\in \ob(G)}\lsub{a}{\Psi}$, and 
 $\lsub{a}{N}(g)$  is denoted  simply as $N(g)$ or sometimes
  even $N_{g}$. Abbreviate $(\Psi(g))(x)$ as $gx$
    for $g\in G_{a}$ and $x\in \Psi(a)$.
 Then \eqref{eq1.13.1} states that  for $g,h\in \mor(G)$, 
  \begin{equation} \label{eq1.13.2} N(gh)=N(g)+gN(h).\end{equation}
 For example, 
 given  a family $(x_{a})_{a\in \ob(G)}$ with
  $x_{a}\in\lsub{a}{\Psi}$, there is a cocycle $N$ defined by
   $N(g)=x_{a}-g(x_{b})$ for  $g\in\lrsub{a}{G}{b}$ and all 
   $a,b\in \ob(G)$; a cocycle of this type is called a \emph{coboundary}.
   The cocycles and coboundaries under their natural  pointwise 
   addition form additive abelian groups 
   $Z^{1}(G,\Psi)\sreq B^{1}(G,\Psi)$ and the quotient abelian group 
   \begin{equation}\label{eq1.13.3} H^{1}(G,\Psi)
   :=Z^{1}(G,\Psi)/B^{1}(G,\Psi)\end{equation} is called  the \emph{first 
   cohomology group}.

  The cocycle condition  readily implies that for any cocycle $N$ of 
  $G$ and any identity morphism $1_{a}$ of $G$, one has $N(1_{a})=0$ 
  and for any morphism $g$ of $G$, $N(g^{*})=g^{*}N(g)$.
  
  Given a cocycle $N\in Z^{1}(G,\Psi)$ and a groupoid homomorphism $\a\colon H\to G$, there is a pullback representation $\Psi \a\colon H\to \Int\modul$ of $H$
  and a cocycle $N'\in Z^{1}(H,\Psi\a)$ defined by $N'(h)=N(\a(h))$.
  By abuse of notation, $N'$ will often be denoted as $N'=N\a$.
  Similarly, given a natural transformation $\nu\colon \Psi\to \Psi'$ between functors $\Psi,\Psi'\colon G\to \Int\modul$ and $N\in Z^{1}(G,\Psi)$,  denote by $\nu N:=N''$ the cocycle $N''\in Z^{1}(G,\Psi')$ defined by $\lsub{a}{N}''(g)=\lrsub{}{\nu}{a}(\,\lsub{a}{N}(g))$ for all $a\in \ob(G)$.

Associated to the $G$-cocycle $N$ on $\Psi$, there is a representation
 $\L\colon G\to \Int\modul$ given by setting $\L(x)$ equal to   $\Psi(x)$ 
 for any object $x$ of  $G$, and $(\L(g))(x)=N(g)+g x$ for $g\in {G}_{a}$ 
 and $x\in \lsub{a}{\Psi}$. This will be called the 
 $N$-twisted $G$-action  or \emph{dot action}, and denoted 
\begin{equation}\label{eq1.13.4} (g,x)\mapsto( \L (g))(x)=g\cdot x:= N(g)+g x. \end{equation}

 \begin{rem*} 
  A representation $\Psi\colon G\to \bringc$  gives, by applying the 
 forgetful functor $\bringc\to  \Int\modul$, an underlying representation
  $\Psi_{\Int}\colon G\to \Int\modul$.  Similarly for morphisms (natural transformations) between such representations.   By a cocycle $N$ for a representation $\Psi\colon G\to \bringc$, we 
  shall mean a $1$-cocycle $N$ for  the underlying  representation $\Psi_{\Int}\colon G\to \Int\modul$.
  Analogous conventions   are used for 
  representations $\Psi\colon G\to \balgc$.

 It is particularly important for our applications here  that the values
 $N(g)$, $g\in\lsub{a}{G}$ of the cocycle $N$ are naturally partially 
 ordered (since they  lie in the Boolean ring $\Psi(a)$).  There are  other    notions  of $1$-cocycle  which  can be used in formulating  similar theories (Remark \ref{ssa5.5}(2)--(3) suggests one of them).   \end{rem*} 

\subsection{}\label{ss1.14}  The following facts are well-known, and 
their simple  proofs are omitted.    
   \begin{lem*} Let $G$ be a   connected, simply connected groupoid.
   \begin{num}\item Any representation of $G$  is equivalent to a trivial 
   representation (i.e. a constant functor).
   \item If $\Psi\colon G\to \Int\modul$ is a representation of $G$, then 
   $H^{1}(G,\Psi)=0$ i.e. any $G$-cocycle for  $\Psi$ is a coboundary.
   \end{num}\end{lem*}
   
\section{Protorootoids}\label{s2}
\subsection{Protorootoids}\label{ss2.1} Protorootoids are important in 
these papers mainly as a framework in which to define and study 
rootoids. However, it is technically convenient to indicate in the 
exposition  many  properties of rootoids which hold even for 
protorootoids, and to give definitions applicable to both  protorootoids 
and  rootoids when possible.

 \begin{defn*} A protorootoid  is  a  triple  $(G,\L,N)$ such that $G$ is
  a groupoid, $\L\colon G\to \bringc $ is a representation of $G$ in the 
  category $\bringc $ and $N\in Z^{1}(G,\L)$ is a $G$-cocycle for
   $\L$.  
\end{defn*}

A  protorootoid $(G,\L,N)$ is  said to be unitary if $\L$ factors as a 
   composite functor $G\to \balgc\to \bringc $ where $\balgc\to \bringc$ 
   is the  inclusion functor.

\begin{rem*} (1) Note that given a groupoid $G$ and functor 
$\L\colon G\to \bringc$, 
$(G,\L, N)$ is  a protorootoid for any $N\in Z^{1}(G,\L)$).  In 
particular, even   for fixed $(G,\L)$, protorootoids $(G,\L,N)$ may exist 
in  abundance (since  $Z^{1}(G,\L)\sreq B^{1}(G,\L)$).

(2)  Despite our use of $\eset$ or $0$ to denote the 
additive identity of $\lsub{a}{\L}$ for each $a\in \ob(G)$, it is tacitly 
assumed 
that the sets $\lsub{a}{\L}$ are pairwise disjoint (cf. \ref{ss1.7}).\end
{rem*}

 \subsection{Weak order}\label{ss2.2} Let $\CR=(G,\L,N)$ be a 
 protorootoid. For $a\in \ob(G)$, define the  subposet $\lsub{a}{\CL}:=
 \mset{N(g)\mid g\in \lsub{a}{G}}$ of $\lsub{a}{\L}$.
The poset  $\lsub{a}{\CL}$ is called  the     \emph{weak order} 
(of $\CR$, or $G$) at   $a$. The weak order at $a$ has a minimum element, 
denoted as $0$ or $\eset$. Those  
 joins and meets which exist in $\lsub{a}{\CL}$ will be denoted using $\vee$ or 
 $\join$, and  $\wedge$ or $\meet$, respectively.

 Define the coproduct  $\CL:=\coprod_{a\in \ob(G)}\,\lsub{a}{\CL}$ in 
 $\posetc$
 and denote its partial order still as $\seq$. In view of our 
 assumption that the 
 sets $\lsub{a}{\L}$ are pairwise disjoint,  $\CL$  is identified with the disjoint 
 union  $\CL:=\dotcup_{a\in \ob(G)}\, \lsub{a}{\CL}$ as set,    with the 
 partial order  which restricts to the weak order on $\lsub{a}{\CL}$ and 
 such that no element of
 $\lsub{a}{\CL}$ is comparable with any element of $\lsub{b}{\CL}$ for 
 any $a\neq b$. 
   The  poset $(\CL,\seq)$ will be called the \emph{big weak order} of 
   $\CR$. Note 
   that it is   stable under the dot action of $G$ in the following sense: if 
   $A=N(g)\in \lsub{a}{\CL}$ where $g\in \lsub{a}{G}$ and 
   $h\in \lsub{b}{G}_{a}$, then
   \begin{equation} \label{eq2.2.1}  h\cdot A=h\cdot N(g)=N(h)+hN(g)=N(hg)\in \lsub{b}{\CL}.\end{equation}
   However, the dot $G$-action is not order-preserving (in a similar sense) on $\CL$ in 
   general.
   On the other hand, the action of $G$ via $\L$ is by (necessarily
    order-preserving)  isomorphisms between the Boolean rings
     $\lsub{a}{\L}$,  which do not  in general preserve their 
     subsets $\lsub{a}{\CL}$. 
   
Comparability of two elements $N(x),N(z)$ in 
$\lsub{a}{\CL}$ admits several 
equivalent reformulations. In fact, 
let $x,y,z\in \mor(G)$ with    $z=xy$. Then $N(z)=N(x)+xN(y)$
and $N(z^{*})=N(y^{*})+y^{*}N(x^{*})$.
Hence  \begin{multline}\label{eq2.2.2}
N(x)\seq N(z)\iff N(x)\cap xN(y)=\eset \iff N(x^{*})\cap N(y)=\eset\\
\iff  N(y^{*})\cap y^{*}N(x^{*})=\eset\iff N(y^{*})\seq  N(z^{*}).
\end{multline}

 \subsection{Weak right preorder}\label{ss2.3} There is    
 a   preorder  $\lsub{a}{\leq  }$   
  on $\lsub{a}{G}$,
   called  the  \emph{weak right preorder} at $a$ (or  on $\lsub{a}{G}$), 
 given by $x\lsub{a}{\leq  } y$ if $N(x)\seq N(y)$.
 By definition,  there is a preorder preserving map
 \begin{equation}\label{eq2.3.1}
  x\mapsto N(x)\colon \lsub{a}{G}
\to  \lsub{a}{\CL}.\end{equation} 
 This map identifies with the component at 
 $(\lsub{a}{G},\lsub{a}{\leq  })$
 of the unit of the adjunction arising  from  existence  of a  
 left adjoint of  $\posetc\to\preordc$ (see \ref{ss1.1}).

 Define the \emph{big right weak 
preorder}  as the coproduct $\coprod_{a\in \ob{(G)}}(\lsub{a}{G},\lsub{a}
{\leq  })$ in $\preordc$. Its underlying set is taken to be 
 $\dotcup_{a\in \ob(G)}\lsub{a}{G}=\mor(G)$, and the partial order 
 on it  is denoted $\leq$. Two 
 morphisms with distinct codomains are incomparable in 
 $\leq$, and the 
 restriction of $\leq$ to $\lsub{a}{G}$ is $\lsub{a}{\leq  }$.    
 Similarly,  define the \emph{weak left preorder} $\leq_{a}$ at $a$ 
 as the preorder   of $G_{a}$  defined by $x\leq _{a}y$ if and only if  
 $x^{*}\lsub{a}{\leq  } y^{*}$, etc.

 \subsection{Properties of   weak right preorders}\label{ss2.4}
 Some  useful properties of weak right preorders of protorootoids
are listed in the next proposition.
 \begin{prop*} Let $\CR=(G,\L,N)$ be a protorootoid.
Let  $a,b,c,d\in \ob(G)$, $v\in \lrsub{a}{G}{d}$, 
$x\in\lrsub{a}{G}{b}$, $y\in \lrsub{b}{G}{c}$  and 
$w\in \lsub{b}{G}$. Then 
\begin{num}
\item $\tensor{1}{_a}\lsub{a}{\leq   }x$.
\item If $x\lsub{a}{\leq  } xy$ then $y^{*}\lsub{c}{\leq  } y^{*}x^{*}$.
\item If $x\lsub{a}{\leq  } xy$ and  $x\lsub{a}{\leq  } xw$, then
$xy\lsub{a}{\leq} xw$  if and only if $y\lsub{b}{\leq  } w$.
\item If $v^{*}\lsub{d}{\leq  }v^{* }x$, $v\lsub{a}{\leq  } xy$ 
and $y^{*}\lsub{c}{\leq  } y^{*}w$, then  
$v^{*}\lsub{d}{\leq  } v^{*}xw$.
\item If $y\lsub{b}{\leq} w$ and $w\lsub{b}{\leq}y$ then $xy\lsub{a}{\leq}xw$.
\end{num}
\end{prop*} 
\begin{proof} Part (a) is trivial and (b) follows from \eqref{eq2.2.2}.
 Part (c) amounts to the  fact that if 
 $N(x)\cap xN(y)=\eset =N(x)\cap xN(w)$, then 
$N(x)\dotcup xN(y)\seq N(x)\dotcup xN(w)$ if and only if  
$N(y)\seq N(w)$.
For (d), the first two conditions give $N(x)\cap N(v)=\eset$ 
(by \eqref{eq2.2.2}) and
 \begin{equation*} N(v)\seq N(xy)= N(x)+ xN(y)\seq N(x)\cup xN(y),\end{equation*} so 
 $N(v)\seq xN(y)$. The third condition gives  by \eqref{eq2.2.2} 
 again that
 $N(y)\cap N(w)=\eset$, hence $xN(y)\cap xN(w)=\eset$. 
 Therefore, $N(v) 
 \cap x N(w)=\eset$. This implies that 
 \begin{equation*} N(v)\cap N(xw)=N(v)\cap (N(x)+xN
 (w))=(N(v)\cap N(x)) +(N(v)\cap xN(w))=\eset\end{equation*} which gives the
  conclusion of 
 (d) by \eqref{eq2.2.2}. The simple proof of (e) is omitted. \end{proof}
 \subsection{Faithful protorootoids}\label{ss2.5}
The next Lemma follows immediately from the 
definitions and the cocycle property. 
\begin{lem*} The following conditions $\text{\rm{(i)--(iv)}}$ on a
 protorootoid $(G,\L,N)$ are equivalent: 
 \begin{conds}
 \item  For all $a\in \ob(G)$ and $g\in \lsub{a}{G}$, 
 one has $N(g)=0$ only if $g=1_{a}$.
\item  For all $a\in \ob(G)$ and $g\in \lsub{a}{G}$, one has 
$g\lsub{a}{\leq  }\tensor{1}{_{a}}$ only if $g=1_{a}$.
\item   For all $a\in \ob(G)$ and any $g,h\in\lsub{a}{G}$  
with $N(g)=N(h)$, one has $g=h$.
\item  The weak right preorders of $\CR$ are all partial orders.
\end{conds}

\end{lem*}

If these conditions hold, then  $\CR$ is said to be a 
 \emph{faithful}  protorootoid. In that case,  there is  an order 
 isomorphism  $x\mapsto N(x)\colon \lsub{a}{G}\xrightarrow{\cong}
 \lsub{a}{\CL}$ and  the weak right preorder
  $\lsub{a}{\leq  }$ on
 $\lsub{a}{G}$ will be called    the \emph{weak right order} 
 (of $\CR$, or $G$, at $a$).  

   \subsection{Compatible expressions}\label{ss2.6}
     If $G$ is a category, an expression $e$ in $G$ is 
     defined to be a diagram
  \begin{equation}\label{eq2.6.1}
  a_{0}\xleftarrow{g_{1}}a_{1}\xleftarrow{}
  \cdots \xleftarrow{}a_{n-1}\xleftarrow{g_{n}}a_{n}
  \end{equation}
  of objects $a_{i}$ and morphisms $g_{i}$ of $G$, where 
  $n\in \Nat$. This expression may be denoted more compactly as 
  $e=\lrsub{a_{0}}{[g_{1},\ldots, g_{n}]}{{a_{n}}}$. 
  Its value $g$ is defined as
  $g:=g_{1}\ldots g_{n}\in \mor(G)$ if $n>0$ and as $1_{a_{0}}$ if $n=0$,
  and its length is defined to be  $n$.
    By abuse of terminology and notation, we may simply say that 
  $g_{1}\cdots g_{n}$ is an expression (with value $g$) or that
   $g=g_{1}\cdots g_{n}$ is an expression.  
  
 Now assume that $G$ is the underlying groupoid of a protorootoid 
 $\CR=(G,\L,N)$. For the above expression $e$, the cocycle condition 
 implies by induction on $n$ that
  \begin{equation}\label{eq2.6.2} N(g)=N(g_{1})+g_{1}N(g_{2})+\ldots +
  g_{1}\cdots g_{n-1}N(g_{n}).\end{equation} The expression $e$ is said to be
    \emph{compatible} if  $g_{1}\cdots g_{i-1}N(g_{i})\cap g_{1}\cdots 
    g_{j-1}N(g_{j})=\eset$ for all $i\neq j$, or equivalently,  if 
  \begin{equation}\label{eq2.6.3} N(g)=N(g_{1})\dotcup g_{1}N(g_{2})\dotcup  
  \ldots \dotcup  g_{1}\cdots g_{n-1}N(g_{n}).\end{equation}

One easily sees by induction on $n$ (see \eqref{eq2.2.2} for 
$n=2$) that
 $e$ is compatible if and only if   
 \begin{equation} \label{eq2.6.4}\eset\seq N(g_{1})\seq N(g_{1}g_{2})\seq\ldots 
 \seq  N(g_{1}\cdots g_{n})\end{equation} in $\lsub{a_{0}}{\CL}$ 
 or equivalently if and only if 
 \begin{equation}\label{eq2.6.5} 1_{a_{0}}\leq g_{1}\leq g_{1}g_{2}\leq \ldots 
 \leq g_{1}\cdots g_{n}\end{equation} in $\lsub{a_{0}}{G}$. 

 The following  \emph{substitution property}  of  compatible 
 expressions partly reduces the study of compatible expressions to the study of those of length two. Its simple proof is omitted.
\begin{lem*} Let  $e$ be an expression as above and let 
 $0=i_{0}\leq i_{1}\leq \ldots \leq i_{p}=n$ be integers. For
  $j=1,\ldots, p$, let  $e_{j}$ denote the expression
 $\lrsub{a_{i_{j-1}}}{[g_{i_{j-1}+1},\ldots, g_{i_{j}}]}{a_{i_{j}}}$ 
 and $h_{j}$ denote the value of $e_{j}$. Then $e$ is  
 compatible if and only if $e_{1},\ldots, e_{p}$ and
  $\lrsub{a_{0}}{[h_{1},\ldots, h_{p}]}{a_{n}}$ are all compatible 
  expressions.\end{lem*}

 \begin{rem*}  Equipped with suitable face and degeneracy 
 operations, by composition and insertion of identity morphisms,
  the expressions in $G$ naturally determine  a simplicial set 
  (see \cite{MacL}) known as
   the nerve of $G$, which is involved in the construction of the 
   classifying space 
   $BG$ of the category $G$. The compatible expressions 
   constitute a simplicial
    subset of the nerve.  \end{rem*}  
 \subsection{Orthogonality of morphisms}\label{ss2.7}
 Let $\CR=(G,\L,N)$ be a protorootoid. 
  Two morphisms $g,h$ of $G$ 
 are said to be \emph{orthogonal} if they have a common 
 codomain $a$
  and  $N(g)\cap N(h)=\eset$ in $\lsub{a}{\L}$.
  
   The following trivial lemma records the way in which  each of the three concepts of 
 weak preorders, compatibility of expressions and 
 orthogonality of
  morphisms can be expressed in terms of each of the others.
  The proof is omitted.

  \begin{lem*} Let $x,y\in \mor(G)$ with $\exists xy$. Then the following conditions  are equivalent:
  \begin{conds}\item $x^{*}$ and $y$ are orthogonal morphisms.
  \item $xy$ is a  compatible expression.
  \item $x\leq xy$ in big weak preorder.
   \end{conds}\end{lem*}

 \subsection{Protomeshes}\label{ss2.8} 
  Call a pair $(R,L)$ consisting 
 of a Boolean ring $R$ and a subset 
$L$ of $R$ a  \emph{protomesh}.  Regard $L$ as a poset 
with order induced by that of $R$.  A morphism 
$\theta\colon (R,L)\to (R',L')$ of  protomeshes
 is defined to be a ring 
homomorphism $\theta\colon R\to R'$ such that 
$\theta(A)\in L'$ for all $A\in L$. This defines a category of 
protomeshes, with its composition given by
 composition  of ring 
homomorphisms. 
The  morphism $\th$ of protomeshes induces a morphism
 $L\to L'$ in $\posetc$. For a protomesh $(R,L)$ and 
 $\G\in R$,  define a protomesh $(R,\G+L)$ where 
 $\G+L:=\mset{\G+\D\mid \D\in L}$. 
 
 The  protomesh $(R,L)$ is said to be \emph{standard} if $0\in L$.
 If $(R,L)$ is any protomesh, then $(R,\G+L)$ is a  standard protomesh 
 for each $\G\in L$.

 \subsection{Relation between weak orders}\label{ss2.9} The following 
 proposition describes  relationships between the  weak orders
  of a protorootoid.
 \begin{prop*} Let $\CR=(G,\L,N)$ be a protorootoid with big weak
  order $\CL$. \begin{num}\item Let $a, b\in \ob(G)$ and 
  $g\in \lrsub{b}{G}{a}$. Then $\L(g)\colon \lsub{a}{\L}\to \lsub{b}{\L}$ 
    is an isomorphism of  protomeshes    $(\lsub{a}{\L},\lsub{a}{\CL})
 \xrightarrow{\cong}(\lsub{b}{\L},\G+\lsub{b}{\CL})$ where 
 $\G:=N(g)\in \lsub{b}{\CL}$. \item For  fixed $b\in \ob(G)$,   the 
 protomesh
 $(\lsub{b}{\L},\lsub{b}{\CL})$  completely determines the family of
  isomorphism types of the    protomeshes
   $(\lsub{a}{\L},\lsub{a}{\CL})$ for   $a\in \ob(G[b])$.     \end{num}
 \end{prop*}
 \begin{proof} 
 For all $x\in \lsub{a}{G}$,
 $(\L(g))(N_{x})=N(g)+N(gx)=\G+N(gx)$. Since  the map 
 $x\mapsto gx\colon \lsub{a}{G}\to \lsub{b}{G}$ is bijective,
 $\lsub{b}{\CL}=\mset{N(gx)\mid x\in \lsub{a}{G}}$ and (a) follows.
Then (b)  holds since   the isomorphism types of     protomeshes
   $(\lsub{a}{\L},\lsub{a}{\CL})$ for   $a\in \ob(G[b])$ coincide with the 
   isomorphism types  of protomeshes 
   $ (\lsub{b}{\L},\G+\lsub{b}{\CL})$ for
 $\G\in \lsub{b}{\CL}$, by (a).
  \end{proof}

  \subsection{Categories of protorootoids}\label{ss2.10}
    By a \emph{morphism  of protorootoids}, we shall mean a morphism in the category $\prootc$ defined below, unless otherwise specified.  \begin{defn*} The category $\prootc$ is the category with  protorootoids as objects  and in which a  morphism
 $(G,\L,N)\to (G',\L',N')$ is  a pair $(\a,\mu)$ such that
 \begin{conds}\item $\a\colon G\to G'$ is a groupoid homomorphism i.e. a functor.
 \item $\mu\colon \L\to \L'\a$ is a natural transformation (between functors $G\to \bringc $).
 \item  $\mu N=N'\a$ i.e. for all $a\in \ob(G)$ and $g\in \lsub{a}{G}$, $\mu_{a}(N_{g})=N'_{\a(g)}$ in $\L'(\a(a))$. \end{conds}
 For another morphism  $(\a',\mu')\colon (G',\L',N')\to (G'',\L'',N'')$,
 the composite morphism $(\a',\mu')(\a,\mu)\colon (G,\L,N)\to (G'',\L'',N'')$   is defined to be  the  pair $(\a'\a,(\mu'\a)\mu)$ where the composite natural transformation $(\mu'\a)\mu\colon \L\to \L''\a'\a$ has component at $a\in \ob(G)$ given by $((\mu'\a)\mu)_{a}:=\mu'_{\a(a)}\mu_{a}$.
   \end{defn*}
   
 The full subcategory of $\prootc$ consisting of faithful 
protorootoids is denoted $\fprc$.
  For a fixed groupoid $G$, the category $G\text{\rm -}\prootc$ of  \emph{$G$-protorootoids} is defined as   the subcategory of  $\prootc$ with only those objects  of the form 
 $(G,\L,N)$  and with only those morphisms of the form $(\Id_{G},\mu)$. 
 The category $\prootc_{1}$ of \emph{unitary protorootoids} is the  (not full) subcategory of $\prootc$ with unitary protorootoids as objects and morphisms $(\a,\nu)$ in $\prootc$ between unitary protorootoids such that each component of $\nu$ is a morphism  
 in $\balgc$ (i.e. the components  are \emph{unital} ring homomorphisms).  

\begin{lem*}  For any  morphism $(\alpha,\mu)\colon (G,\L,N)\to(G',\L',N')$ of protorootoids, the map induced by $\a$ on morphisms  restricts for each $a\in \ob(G)$ to   a weak preorder  preserving map $\lsub{a}{\a}\colon (\lsub{a}{G},\lsub{a}{\leq  })\to (\lsub{a'}{G}',\lsub{a'}{\leq  })$ where $a':=\a(a)$.
Further,  $\mu_{a}\colon (\lsub{a}{\L},\lsub{a}{\CL})\to (\lsub{\a(a)}{\L}',\lsub{\a(a)}{\CL}')$ is a morphism of protomeshes,
where $\CL'$ is the big weak order of $(G',\L',N')$.  \end{lem*}
\begin{proof} Let $g_{1},g_{2}\in \lsub{a}{G}$ with $g_{1}
\lsub{a}{\leq  } g_{2}$ i.e. $N(g_{1})\seq N(g_{2})$.
Since $\mu_{a}\colon \lsub{a}{\L}\to \lsub{a'}{\L}'$ is a homomorphism of Boolean rings, it is order preserving, and  the definitions give \begin{equation*} N'(\a(g_{1}))=\mu_{a}(N(g_{1}))\seq \mu_{a}( N(g_{2}))=N'(\a(g_{2}))\end{equation*} i.e. $\a(g_{1})
\lsub{a'}{\leq  } \a(g_{2})$. Both assertions follow.\end{proof}
  \subsection{Inverse image}\label{ss2.11}
Let $\CR=(G,\L,N)$ be a protorootoid 
and  $i\colon H\to G$ be  a 
groupoid morphism. Define the \emph{inverse image protorootoid}  
$i^{\nat}\CR:=(H,\L i, Ni)$. 
There is a morphism $i^{\flat}=(i,\Id_{\L i}) \colon  i^{\nat}\CR \to \CR$ 
in $\prootc$.

Note that $i^{\nat}$ becomes a functor $i^{\nat}\colon 
G\text{\rm -}\prootc\to H\text{\rm -}\prootc$ if one defines 
$i^{\nat}(\Id_{G},\nu)=(\Id_{H},\nu i)$
for any morphism $(\Id_{G},\nu)$ of $G$-protorootoids.

Also, $i^{\flat}$ has the following universal property:
given a  protorootoid $\CR'=(H,\L',N')$ and a morphism
$f=(i,\mu)\colon \CR'\to \CR$ in $\prootc$, there is a unique morphism 
$g\colon\CR'\to i^{\nat}(\CR)$ in $H\text{\rm -}\prootc $ such that 
$f=i^{\flat}g$ in $\prootc$ (namely, $g=(\Id_{H},\mu)$). 

 \subsection{Restriction}\label{ss2.12} If 
$i\colon H\to G$ is the inclusion morphism of a 
subgroupoid $H$  into $G$, 
then $i^{\nat }\CR$ is called the   \emph{restriction}  of $\CR$ to $H$ 
and is denoted sometimes as $\CR_{H}:=i^{\nat }\CR$.  

\subsection{Coverings} \label{ss2.13}   A morphism 
$f=(\a,\nu)\colon \CR'\to \CR$ in $\prootc$ is said to be  a \emph{covering
 morphism} or \emph{covering} if  $\a$ is a covering morphism of
  groupoids and $\nu$ is a natural isomorphism. Equivalently, $f$ is a 
  covering if 
  $\a$ is a covering of groupoids and the natural morphism
  $\CR'\to \a^{\nat}(\CR)$ is an isomorphism in $\prootc$.
  In that case,  $\CR'$ is called a covering protorootoid  or
   \emph{covering} of $\CR$. If the groupoid morphism $\a$ is a covering quotient morphism, then  $\CR$ is called a 
   \emph{covering quotient} of $\CR'$ and $f$ is called a \emph{covering quotient morphism}.  
  A \emph{universal covering} of a protorootoid $\CR$ is  a covering
  $(\a,\nu)\colon \CR'\to \CR$ such that $\a$ is a universal covering in 
  $\grpdc$ of 
  the underlying groupoid of $\CR$. Such a universal covering exists
   and it  is unique up to isomorphism as an object of $\prootc/\CR$.

\subsection{}\label{ss2.14}   Another  category $\prootcp$ of
 protorootoids, which will only be considered occasionally,  is defined
  as in \ref{ss2.10}(i)--(iii) but  taking 
  $\mu\colon  \L'\a\to \L$ in (ii) and replacing (iii) by the following: 
$N=\mu N'\a$ i.e.  for $a\in \ob(G)$ and   $g\in \lsub{a}{G}$,  one has 
 $N_{g}=\mu_{a}(N'_{\a(g)})$.
 For another morphism  $(\a',\mu')\colon (G',\L',N')\to (G'',\L'',N'')$,
 the composite $(\a',\mu')(\a,\mu)\colon (G,\L,N)\to (G'',\L'',N'')$  in
  $\prootcp$  is  the  pair $(\a'\a,\mu(\mu'\a))$ where
    $\mu(\mu'\a)\colon  \L''\a'\a\to \L$ has component at 
    $a\in \ob(G)$ given 
    by $(\mu(\mu'\a))_{a}:=\mu_{a}\mu'_{\a(a)}$.  
    
    Define $G\text{\rm -}\prootcp$ from $\prootcp$ in a similar way as 
    $G\text{\rm -}\prootc$ is defined from $\prootc$.

\subsection{Groupoid-preorders} \label{ss2.15}
The  category $\gpdpreord$ of groupoid-preorders is defined as follows.  It has as 
objects groupoids  $G$
such that for each $a\in \ob(G)$, there is a given 
preorder  $\lsub{a}{\leq  }$ 
on $\lsub{a}{G}$, called the \emph{weak right preorder} of $G$ at $a$. 
A morphism $G\to H$ in $\gpdpreord$ is a groupoid 
homomorphism $\th\colon H\to G$ such that for each $a\in \ob(H)$, the 
restriction $\lsub{a}{\th}$ of $\th$ to a function 
$\lsub{a}{G}\to \lsub{\th(a)}{H}$ 
is a morphism in $\preordc$. Composition in $\gpdpreord$ is given by
 composition of underlying groupoid morphisms.     The full subcategory of 
$\gpdpreord$ consisting of groupoids for 
which all the weak right preorders are 
partial orders is called the category of groupoid-orders.

Formally, a groupoid-preorder is  a pair $(G,\leq)$
 consisting of a groupoid $G$ and a preorder $\leq$, which is  called the \emph{big weak right 
  preorder}, on $\mor(G)=\dotcup_{a\in \ob(G)}\lsub{a}{G}$ such that 
  $\leq$ restricts to the weak right preorder of $\lsub{a}{G}$, and
   elements in   different stars $\lsub{a}{G}$ of $G$ are incomparable.

There is a natural  forgetful functor 
$\mathfrak{P}\colon \prootc\to  \gpdpreord$
which on objects takes a protorootoid $\CR=(G,\L,N)$ to the 
groupoid $G$ 
endowed with the collection of  weak right preorders of 
$\CR$, and which takes a 
morphism of protorootoids to the underlying morphism of 
groupoids (which is 
a morphism in $\gpdpreord$ by Lemma \ref{ss2.10}). 
The full subcategory of $\gpdpreord$  with objects the groupoid-preorders which are isomorphic to
$\mathfrak{P}(\CR)$ for some protorootoid $\CR$ is 
denoted $\gpdpreord_{P}$ and called the category of 
\emph{protorootoidal groupoid-preorders}. In this definition,
``isomorphic'' could be replaced by ``equal'' 
since if $i\colon (G,\leq)\to \mathfrak{P}(\CR)$ is an isomorphism in $\gpdpreord$, then (regarding  $i$ just  as a morphism of groupoids)  $\mathfrak{P}(i^{\nat}(\CR))=(G,\leq)$.

%
 \subsection{Order isomorphism}\label{ss2.16} A 
 \emph{preorder isomorphism} from a  protorootoid  $\CR$ to a
  protorootoid $\CT$ is by definition  an isomorphism 
  $\mathfrak{P}(\CR)\to \mathfrak{P}(\CT)$ in $\gpdpreord$ i.e. 
 an isomorphism $\theta\colon G\to H$ from the underlying 
 groupoid of $\CR$ to that of $\CT$ such that the induced 
maps $\lsub{a}{\th}\colon \lsub{a}{G}\to \lsub{\theta(a)}(H)$ for $a\in \ob(G)$ are all preorder 
isomorphisms in the corresponding right weak preorders.  Similarly,
 one defines \emph{order isomorphisms} of faithful protorootoids. 

\begin{rem*} (1) Many, though not all, properties of 
(and definitions concerning)
protorootoids $\CR$  may be expressed completely
in terms of $\mathfrak{P}(\CR)$. For example, faithfulness of a 
protorootoid is such a property, and the simplicial set of compatible expressions depends up to isomorphism  only on the preorder  isomorphism type. 

(2) Subsequent papers will give characterizations of protorootoidal groupoid-preorders and show how an analogue of  
 part of the theory of protorootoids and rootoids may be developed in the context of  $\gpdpreord_{P}$.  \end{rem*}
  \section{Principal protorootoids}\label{s3}
 \subsection{Groupoid generators}\label{ss3.1} Let $G$ be a 
 groupoid and  $S\seq \mor(G)$. The subgroupoid $H$ of $G$ generated by $S$ is defined to the subgroupoid of $G$ containing all identity morphisms of $G$ and all  morphisms $g$ of $G$  which are  expressible
   as a product
 $s_{1}\cdots s_{n}$ with $s_{i}\in S\cup S^{*}$.  One says more briefly that $S$ \emph{generates} $H$.
   If $S$ generates $G$, a corresponding 
 \emph{length function} $l_{S}\colon \mor(G)\to \Nat$ is  defined by
 \begin{equation}\label{ss3.1.1} l_{S}(g):=\min(\mset{n\in \Nat\mid g=s_{1}\cdots s_{n}, s_{i}\in S\cup S^{*}})\end{equation} if $g$ is not an identity morphism, and $l_{S}(g):=0$ if $g$ is an identity morphism.

  \subsection{Rank in Boolean rings}\label{ss3.2} A finite Boolean ring $B$ is a Boolean algebra since the join of all
 (finitely many) of its elements is a maximal element of $B$ and hence an identity element of $B$. 
 Recall that a finite Boolean algebra $B$ is isomorphic to the Boolean algebra of  subsets of a finite set of uniquely determined cardinality $\rank(B)$ (equal to the number of atoms of $B$).
 
 Now  let $B$ denote an arbitrary Boolean ring.  For $x\in B$, the principal ideal  generated by $x$ is \begin{equation*} xB=x\cap B=\mset{x\cap y\mid y\in B}=\mset{x'\in B\mid x'\leq x},\end{equation*}  which, regarded as a subring of $B$, is itself a Boolean ring. If $xB$ is finite, say that $x$ is of finite \emph{rank} $\rank(x):=\rank(xB)$. The atoms of $B$ are its elements of rank $1$.
 
 Though the following is well known, a proof is given for completeness.
 \begin{lem*} Let $B$ be a Boolean ring, and $U$ be the set of atoms of $B$.  Let $\wp'(U)$ be the (possibly non-unital) subring of $\wp(U)$ with the finite subsets of $U$ as its elements.
 \begin{num}\item 
If $x,y\in B$ are of finite rank, then so are $x\cup y$ and $x\cap y$, and $\rank(x)+\rank(y)=\rank(x\cap y)+\rank(x\cup y)$.
 \item Let $B'$ be the subring (also an ideal and  order ideal) of $B$ consisting of elements of finite rank. Then the map $\th\colon x\mapsto \mset{y\leq x\mid \rank(y)=1}$ defines an isomorphism of Boolean rings
 $\th \colon B'\to \wp'(U)$.
 \item If $x,y\in B'$, then $\th(x\cup y)+\th(x\cap y)=\th(x)+\th(y)$. \end{num}
 \end{lem*}
\begin{proof} If  $x,y\in B$ satisfy $x\cap y=\eset$,
then $x$ and $y$ are orthogonal idempotents so
$(x+y)B\cong xB\times yB$ as ring. If $x$ and $y$ are also of finite  rank, this implies that $\rank(x+y)=\rank(x)+\rank(y)$, so $x+y$ is of finite rank, and $\th(x+y)=\theta(x)\dotcup \th(y)$.
 Now let $x,y\in B$ be arbitrary elements of  finite rank.
 Since $x\cap y$ and $x+(x\cap y)$ are orthogonal  idempotents
  (of finite rank since they are in $[\eset,x]_{B}$),
  one has \begin{equation*} \rank(x+(x\cap y))+\rank(x\cap y)=\rank(x)\end{equation*}
  and $\th(x+(x\cap y))\dotcup \th(x\cap y)=\th(x)$.
  But $y$ and $x+(x\cap y)$ are also orthogonal, with sum $x\cup y$, so  \begin{equation*} \rank(x+(x\cap y))+\rank(y)=\rank(x\cup y)\end{equation*}
  and $\th(x+(x\cap y))\dotcup \th( y)=\th(x\cup y)$. These formulae easily imply  (a), (c) and that $\theta(x\cup y)=\theta(x)\cup 
  \th(y)$. It is also clear by the definition of $\th$ that $\th(x\cap y)=\th(x)\cap \th(y)$. Using  $x\cup y=(x+y)\dotcup (x\cap y)$ and an analogous fact in $\wp'(U)$, it follows that
  \begin{multline*}(\th(x)+\th(y))\dotcup (\th(x)\cap \th(y))=
  \th(x)\cup \th(y)\\=\th(x\cup y)=\th(x+y)\dotcup \th(x\cap y)=
  \th(x+y)\dotcup (\th(x)\cap \th(y))\end{multline*} from which
  $\th(x+y)=\th(x)+\th(y)$. Hence $\th\colon B'\to\wp'(U)$ is a ring homomorphism.
  One readily checks that an inverse function  $\th^{-1}\colon \wp'(U)\to B'$ is given by $X\mapsto \cup_{x\in X} x$ (the join in $B$ of $X$)  for finite $X\seq U$.    \end{proof} 
 
  \subsection{Terminology for protorootoids}\label{ss3.3}
The following definition collects the   basic terminology used in these papers for protorootoids.  Complete and  principal protorootoids are the two most important classes; others are  technically useful in relation to them  or  in formulating  results in their natural generality.  \begin{defn*} Let $\CR= (G,\L,N)$ be a protorootoid with  big weak order $\CL$. \begin{num}
  \item $\CR$ is said to be \emph{connected} (resp., \emph{simply connected}) if its underlying groupoid  $G$ is connected (resp., simply connected).
   \item $\CR$ is \emph{complemented} if it is unitary    and  for each $a\in\ob(G)$ and $A\in \lsub{a}{\CL}$, one has $A^{\cp}:=\lrsub{}{1}{\L(a)}+A\in \lsub{a}{\CL}$.
\item   $\CR$  is  \emph{complete} if for each $a\in \lsub{a}{G}$, $\lsub{a}{\CL}$ is a complete lattice.
\item  $\CR$ is 
\emph{interval finite} if  for each $a\in \ob(G)$ and each morphism $g\in \lsub{a}{G}$, the interval
$[\eset,N(g)]_{\lsub{a}{\CL}}:=\mset{A\in \lsub{a}{\CL}\mid  A\seq N(g)}$ in  $\lsub{a}{\CL}$ is finite. 
\item $\CR$ is \emph{cocycle finite} if  for each $a\in \ob(G)$ and each morphism $g\in \lsub{a}{G}$, the element $N(g)$ is of finite rank in $\lsub{a}{\L}$.  In that case,  define a function $l_{N}\colon \mor(G)\to \Nat$ by  $l_{N}(g):=\rank(N(g))$.
\item An element $s\in \mor(G)$, say $s\in \lrsub{a}{G}{b}$, is  an  \emph{atomic morphism}  of $\CR$ (or $G$) if $N(s)$ is an  atom of  the  weak order  $\lsub{a}{\CL}$. Let $A_{\CR}$ denote the set of atomic morphisms of $\CR$.
\item An element $s\in \mor(G)$, say $s\in \lrsub{a}{G}{b}$, is  a \emph{simple morphism}  of $\CR$ (or of $G$) if $N(s)$ is an  atom of  the  Boolean ring  $\lsub{a}{\L}$. Let $S_{\CR}$ denote the set of simple morphisms of $\CR$.
\item $\CR$ is  \emph{atomically generated} (resp., \emph{simply generated}) if $A_{\CR}$ (resp., $S_{\CR}$) generates $G$.
\item $\CR$ is  \emph{principal} if it is cocycle finite, simply generated  and $l_{S}=l_{N}\colon \mor(G)\to \Nat$ where $S:= S_{\CR}$.
\item  $\CR$ is  \emph{preprincipal} if it is faithful and interval finite, and for all $a\in \ob(G)$, $g\in \lsub{a}{G}$ and $s\in \lsub{a}{A}$ (where $A:=A_{\CR}$), either $N(g)\cap N(s)=\eset$ or $N(s)\seq N(g)$ (i.e. either $s^{*}g$ or $sh$ is a compatible expression, where $h:=s^{*}g$). 
\item $\CR$ is \emph{abridged} if for each $a\in \ob(G)$, $\lsub{a}{\L}$ is generated as Boolean ring by $\lsub{a}{\CL}$.
\item $\CR$ is \emph{saturated} if for every $a\in \ob(G)$ and every $g\in\lsub{a}{G}$,
every maximal totally ordered subset of $[\eset,N(g)]_{\,\lsub{a}{\CL}}$ is also a maximal totally ordered subset of $[\eset,N(g)]_{\lsub{a}{\L}}$.
\item $\CR$ is \emph{pseudoprincipal} if for every $a\in \ob(G)$ and $h,g\in \lsub{a}{G}$ with  $N(h)\neq \eset$, there exists $x\in \lsub{a}{G}$ with $\eset\neq N(x)\seq N(h)$ and either $N(x)\seq N(g)$ or $N(x)\cap N(g)=\eset$.
\item  $\CR$ is \emph{regular} if for every $a\in \ob(G)$ and every
non-empty directed subset $X$ of $\lsub{a}{\CL}$ with a join $x$ in $\lsub{a}{\CL}$, $x$ is also the join of $X$ as a subset of $\lsub{a}{\L}$.
\end{num}
\end{defn*}

 The set $S$ of simple morphisms of a simply generated protorootoid   $\CR=(G,\L,N)$ is  
  called the set of \emph{simple generators} of $\CR$ (or less precisely, of $G$). Note that  $S_{\CR}$ and $A_{\CR}$ may  be empty for an arbitrary protorootoid $\CR$. 
  
  The property of being connected (resp., simply connected,  complete, interval finite, atomically generated, preprincipal, pseudoprincipal) depends only on the preorder isomorphism type of the protorootoid

 \subsection{}\label{ss3.4} Basic properties of atomic and simple morphisms are listed below.\begin{lem*} Let $\CR=(G,\L,N)$ be a protorootoid. 
Set $A:=A_{\CR}$ and $S:=S_{\CR}$. Then:
\begin{num}
\item$A=A^{*}$ and  $S=S^{*}$. 
\item $S\seq A$.
\item  $A$ (and therefore $S$)  contains no identity morphism of $G$.
\end{num}
 \end{lem*}\begin{proof}
To show $S=S^{*}$ in (a), observe that for any 
 $g\in \lrsub{b}{G}{a}$, $\L(g)$ maps the set of atoms of 
 $\lsub{a}{\L}$ to the set of atoms of $\lsub{b}{\L}$. 
 If $s\in \lrsub{a}{S}{b}$, then $N(s)$ is an atom of $ \lsub{a}{\L}$ 
 and so $N(s^{*})=s^{*}N(s)$ is an atom of
 $\lsub{b}{\L}$ i.e. $s^{*}\in \lsub{b}{S}$. This shows that 
 $S^{*}\seq S$. Hence $S=S^{**}\seq S^{*}\seq S$.  To prove that 
 $A=A^{*}$, it will suffice to show that if $s\in \mor(G)$ is not in 
 $A$, then $s^{*}\not\in A$ also.   Note that either $N(s)=\eset$ or there exist 
  $x,y\in \mor(G)$ with $s=xy$ and  $\eset \sneq N(x) \sneq N(s)$.
    In the first case $N(s^{*})=s^{*}N(s)=\eset$ and in the second 
    case, $s^{*}=y^{*}x^{*}$ with $\eset \sneq N(y^{*})\sneq N(s^{*})$ (see  \eqref{eq2.2.2}). In either case, $s^{*}\not \in A$. This proves (a). Parts (b)--(c) are immediate consequences of the definitions. \end{proof}
\subsection{}\label{ss3.5}   The following is a very useful property of interval finite protorootoids. 
 \begin{lem*} If  the protorootoid $\CR=(G,\L,N)$ is interval finite and faithful, then $G$ is atomically generated.\end{lem*}
  \begin{proof}
    Define a function $L\colon \mor(G)\to \Nat$ by 
    $L(g)=\vert [\eset, N(g)]_{\lsub{a}{\CL}}\vert $ for $g\in \lsub{a}{G}$. We show that $g\in \mor(G)$ is in the subgroupoid $G'$ of $G$ generated by $A:=A_{\CR}$ by induction on $L(g)$.
 If $L(g)=0$, then $N(g)=\eset$ and $g=1_{a}\in G'$ since $\CR$ is faithful. Suppose $L(g)>0$. There is $s\in \lsub{a}{A}$ with $N(s)\seq N(g)$. Set $g':=s^{*}g\in \lsub{b}{G}$. Since $N(s)\seq N(g)$, it follows that $N(s^{*})\cap N(g')=\eset$. Using  Proposition \ref{ss2.4}, one checks that the map
 $x\mapsto x':=sx$ defines a bijection 
 \begin{equation*} \mset{x\in \lsub{b}{G}\mid \lrsub{}{1}{b}\lsub{b}{\leq  } x \lsub{b}{\leq  }g'} \xrightarrow{\cong}
 \mset{x'\in\lsub{a}{G}\mid s\lsub{a}{\leq  } x' \lsub{a}{\leq  }g}. \end{equation*}
 Since $\lrsub{}{1}{a}\lsub{a}{\leq}g$ but $s\lsub{a}{\not\leq}\lrsub{}{1}{a}$, it follows that $L(g')<L(g)$. By induction, $g'\in\mor(G')$ so $g=sg'\in \mor(G')$ as required. 
  \end{proof}
  \begin{rem*}  If $\CR$ is interval finite and faithful and $H$ is a subgroupoid of $G$, then the restriction $\CR_{H}$ is also interval finite and faithful, so the atomic morphisms of $\CR_{H}$ form  a  set of generators 
 of $H$.
  \end{rem*}
  
    \subsection{} \label{ss3.6}
 Part (c) of the lemma below eliminates some redundancies from the definition of  principal protorootoids.     \begin{lem*} Let $\CR=(G,\L,N)$ be a protorootoid, $A:=A_{\CR}$ and $S:=S_{\CR}$.
       \begin{num}
       \item  If $\CR$ is cocycle finite, it is interval finite.
        \item If $\CR$ is simply generated, it is cocycle finite and atomically generated, and for all $g\in \mor(G)$, $l_{N}(g)\leq l_{S}(g)$.
     \item $\CR$ is principal if and only if  it is simply generated and for all $g\in \mor(G)$, $l_{S}(g)\leq l_{N}(g)$.
     \item If $\CR$ is interval finite, it is regular.
            \end{num}\end{lem*}
\begin{proof} Part (a) holds since for $g\in \mor(g)$,
\begin{equation}\label{ss3.6.1}
[\eset,N_{g}]_{\lsub{a}{\CL}}\seq [\eset,N_{g}]_{\lsub{a}{\L}}.
\end{equation} For (b), assume that $\CR$ is simply generated. By 
Lemma \ref{ss3.4}(b), $\CR$ is atomically generated. 
Let $g\in \mor(G)$, say $g=s_{1}\cdots s_{n}$ where $s_{i}\in S$ and $n=l_{S}(g)$. By  the cocycle condition, \begin{equation*}
N_{g}=\sum_{i=1}^{n}s_{1}\ldots s_{i-1}(N_{s_{i}}).
\end{equation*} In this, $\rank(s_{1}\cdots s_{i-1}(N_{s_{i}}))=
\rank(N_{s_{i}})=1$, so by Lemma  \ref{ss3.2}(a), 
$\rank(N_{g})\leq \sum_{i=1}^{n}1=n=l_{S}(g)$. This completes the proof of (b), and (c) follows immediately from (b) and the definition of 
principal protorootoids. Part (d) holds since if a non-empty subset $X$
of $\lsub{a}{\CL}$ has a join $x$ in $\lsub{a}{\CL}$, then $X$ is finite and  $x$ is the maximum element of $X$, so $x$ is also the join of $X$ in $\lsub{a}{\L}$.    
\end{proof}

\ssect{} Principal protorootoids may also  be  characterized as follows. \label{ss3.7}
\begin{lem*} Let  $\CR=(G,\L,N)$  be a  protorootoid,  $A:=A_{\CR}$ and  $S:=S_{\CR}$.
Then $\CR$ is principal if and only if  it is faithful and  atomically generated and $A\seq S$.
 In that case, $A=S$.\end{lem*}
\begin{proof} Suppose that $\CR$ is faithful, atomically generated and that $A\seq S$. Then $A=S$ by  Lemma \ref{ss3.4}(b), so 
$\CR$ is simply generated.
     Let  $g\in \lrsub{b}{G}{a}$ and $s\in \lsub{a}{S}$ for some
    $a\in\ob(G)$. Then $N(gs)=N(g)+gN(s)$. Since $N(s)$ is an atom of $\lsub{a}{\L}$, $gN(s)$ is an atom of $\lsub{b}{\L}$. By Lemma \ref{ss3.2}, it follows that $l_{N}(gs)=l_{N}(g)+1$ if $gN(s)\cap N(g)=\eset$, while otherwise, $gN(s)\seq N(g)$ and $l_{N}(gs)=l_{N}(g)-1$.
    In particular, $l_{N}(gs)\in\set{l_{N}(g)\pm 1}$.
     This   implies by induction on $l_{S}(g)$ that $l_{S}(g)\equiv l_{N}(g)
 \pmod 2$, and that $l_{S}(gs)\in\set{l_{S}(g)\pm 1}$ for $g\in \mor(G)$ and 
  $s\in S$ if $\exists gs$. Since $l_{S}(x)=l_{S}(x^{*})$ for all $x\in \mor(G)$, it follows that  for  $s\in S$, $g\in G$ with $\exists sg$, one has $l(sg)\in\set{l(g)\pm 1}$. 

To show $\CR$ is principal, it remains to show that for all $g\in \mor(G)$, $l_{S}(g)=l_{N}(g)$. This is  proved by induction on  $n:=l_{N}(g)$.
If $n=0$, then $N({g})=\eset$, $g$ is an identity morphism since  
$\CR$ is faithful, and so $l_{S}(g)=0=n$.   Next, suppose inductively that $l_{S}(g)=l_{N}(g)$ for all $g$ with $l_{N}(g)<n$, where $n>0$. Let $g\in \lsub{a}{G}$ with $l_{N}(g)=n$. 
 Since $n>0$, $N(g)\neq \eset$.
Since $\CR$ is interval finite, the interval 
$[0,N(g)]_{\lsub{a}{\CL}}$ has an atom i.e. there is some 
$r\in \lsub{a}{A}$  with $N({r})\seq N({g})$. 
Set $s=r^{*}\in S$. The cocycle condition implies $N({sg})=s(N({g})+N({r}))$  and so $l_{N}(sg)=l_{N}(g)-1=n-1$.  By induction, $l_{S}(sg)=n-1$. 
By Lemma \ref{ss3.6}(b),  \begin{equation*}
n=l_{N}(g)\leq l_{S}(g)\in\set{l_{S}(sg)\pm 1}=\set{(n-1)\pm 1}
\end{equation*} and thus  $l_{S}(g)=n=l_{N}(g)$ as required.

Conversely, suppose  that $\CR$ is principal.
If $g\in \lsub{a}{G}$ with $N(g)=\eset$, then 
$l_{N}(g)=0=l_{S}(g)$ so $g=1_{a}$. Hence $\CR$ is faithful.
Since $\CR$ is simply generated, it  is atomically generated by Lemma \ref{ss3.6}(b).  Let $s\in \lsub{a}{A}$ with $l_{N}(s)=l_{S}(s)=n$. Note $n>0$ by Lemma \ref{ss3.4}(c). Write $s=s_{1}\cdots s_{n}$ with $s_{i}\in S$.
Then $l_{N}(s_{1}\cdots s_{i})=l_{S}(s_{1},\cdots s_{i})=i$ for $i=0,\ldots, n$ which implies that
$\eset\sneq N(s_{1})\sneq N(s_{1}s_{2})\sneq \ldots \sneq   N(s)$.
In particular, since $N(s)$ is an atom of $\lsub{a}{\CL}$, it follows that
$n=1$ and 
 $s=s_{1}\in S$. This shows that $A\seq S$. One has $S= A$ by Lemma \ref{ss3.4}(b), completing the proof.
 \end{proof}

\ssect{} \label{ss3.8} Let $\set{\pm 1}$ be 
the group of units of the ring $\Int$ and  regarded it  as a groupoid with one object.
Let $G$ be a groupoid and $S$ be a set of generators of $G$. A sign character of $(G,S)$ is defined to be  a groupoid homomorphism 
$\e\colon G\to \set{\pm 1}$ such that $\e(s)=-1$ for all $s\in S$.
If it exists, it is unique, since it is given on morphisms by $\e(g)=(-1)^{l_{S}(g)}$.

\begin{cor*} Let $\CR=(G,\L,N)$ be a principal protorootoid and
$S:=S_{\CR}$. Then $(G,S)$ admits a sign character.\end{cor*}
\begin{proof} The preceding proof shows that if $g\in \mor(G)$ and $s\in S$ and  $l_{S}(gs)=l_{S}(g)\pm 1$ if $\exists gs$. Also,  $l_{S}(s)=1$ by Lemma \ref{ss3.4}(c). Thus, one may take   $\e(g):=(-1)^{l_{S}(g)}$.
\end{proof}

 \subsection{Abridgement}\label{ss3.9} Let $\CR=(G,\L,N)$ be a protorootoid. Define
 $\L'$ as the $G$-subrepresentation (in $\bringc$) of $\L$ generated by the 
 elements $\lsub{a}{N}(g) \in \lsub{a}{\L}$ for $a\in \ob(G)$ and  
 $g\in \lsub{a}{G}$.  As functor, $\L'$ is determined on objects by 
 defining, for $a\in \ob(G)$,  $\L'(a)$ as the subring of  $\lsub{a}{\L}$ generated by all 
 elements $\lsub{a}{N}(g) \in \lsub{a}{\L}$  with  
 $g\in \lsub{a}{G}$. For a morphism $g\in \lrsub{a}{G}{b}$, $\L(g)$  
 maps $\L'(b)$ into $\L'(a)$ (since for $h\in\lsub{b}{G}$,
$gN(h)=N(gh)+N(g)$) and so $\L'(g)$ may be defined as 
  the restriction of $\L(g)$ to a morphism 
  $\L'(g)\colon \L'(b)\to \L'(a)$  in $\bringc $. 
  Obviously, $N$ restricts to a $G$-cocycle  $N'$ for  $\L'$, 
  giving a protorootoid  $(G,\L',N')$ which will be called 
 the \emph{abridgement} $\CR^{a}:=(G,\L',N')$ of $\CR$.  The protorootoid $\CR$ is abridged  as defined in \ref{ss3.3}(k)   if it is equal to $\CR^{a}$. 
 
 It is immediate from the definition that a protorootoid $\CR$ and its 
 abridgement  have equal underlying groupoid-preorders:
 $\mathfrak{P}(\CR)=\mathfrak{P}(\CR^{a})$. 
 The big weak orders  of $\CR$ and $\CR^{a}$ are also equal as posets (though their  corresponding protomeshes $(\lsub{a}{\L},\lsub{a}{\CL})$  and  $(\lsub{a}{\L}',\lsub{a}{\CL})$ differ). 
 The set of atomic (resp., simple) morphisms of  $\CR$
 is equal to (resp., a subset of) the set of atomic (resp., simple) morphisms of $\CR^{a}$. Any property of protorootoids $\CR$ which depends only on the preorder isomorphism type of  $\CR$ holds for $\CR$ if and only if it holds for $\CR^{a}$.  If $\CR$ is faithful, or has any one of the properties in \ref{ss3.3}(a)--(e) or (h)--(n), then  $\CR^{a}$ has that same property (using Lemma \ref{ss3.7} for (i));
 the converse holds for faithfulness,  (a), (c), (d), atomically generated in (h), (j) and (m).
 
  \begin{rem*} The abridgement of a unitary protorootoid need not be unitary.  However, there is an  analogue of abridgement for unitary protorootoids. It attaches to a unitary protorootoid
 $\CR=(G,\L,N)$ a unitary protorootoid $(G,\L'',N'')$ where for $a\in \ob(G)$,  $\L'(a)$ is the subring  of $\L(a)$ generated by $\set{1_{\L(a)}}\cup \set{N(g)\mid g\in \lsub{a}{G}}$, and $N''$ is the evident restriction of $N$.  One may have for example, that  $\CR^{a}=(G,\L',N')$ where $\L'(b)$ is the Boolean  ring of all finite subsets of some set $U(b)$ whereas
 $\L''(b)$ is the Boolean algebra of all subsets of $U(b)$ which are finite or cofinite (that is,  have finite complement) in $U(b)$. 
 \end{rem*}

   \subsection{}\label{ss3.10}  Let $\prootc^{a}$ denote the full subcategory of $\prootc$ consisting of  abridged protorootoids.
   There is an abridgment functor  $\ab\colon \prootc\to \prootc^{a}$ with $\ab(\CR)=\CR^{a}$ for $\CR$ as above, and defined on morphisms as follows. Let $\CT:=(H,\G,M)$ be a 
 protorootoid with abridgement $\CT^{a}=(H,\G',M')$, and 
 let $f=(\a,\nu)\colon \CT\to \CR$ be a morphism in 
 $\prootc$. For any $b\in \ob(H)$ and $h\in \lsub{b}{H}$, one has $\nu_{b}(M_{h})=N_{\a(h)}$. It follows from the definitions   that the homomorphism  $\nu_{b}\colon \G(b)\to \L(\a(b))$ of Boolean rings restricts  to  a homomorphism
 $\nu'_{b}\colon \G'(b)\to \L'(\a(b))$. Clearly,  the homomorphisms $\nu'_{b}$ for $b\in \ob(H)$ are the components of a natural transformation $\nu'\colon \G'\to \L'\a$, and $f':=(\a,\nu')\colon \CT^{a}\to \CR^{a}$ is a morphism in $\prootc$. Setting $\ab(f):=f'$ defines the functor $\ab$ as  required.  
 
 It is easily seen that  $\ab$ is right adjoint to the inclusion functor 
 $\mathfrak{B}\colon\prootc^{a}\to \prootc$. 
 The unit of the adjunction is the identity natural transformation of the identity functor $\Id_{\prootc^{a}}$. 
 The component at $\CR=(G,\L,N)$ of the counit is the 
 protorootoid morphism $(\Id_{G},\mu)\colon \CR^{a}\to \CR$ in which for all $a\in \ob(G)$,  $\mu_{a}$ is the inclusion $\L'(a)\to \L(a)$. In particular, $\prootc^{a}$ is a full, coreflective subcategory of $\prootc$. 
 
\ssect{}  \label{ss3.11} The  following  result describes the relationship between principal and preprincipal protorootoids. \begin{prop*} The protorootoid $\CR$ is preprincipal if and only if its abridgement $\CR^{a}$ is principal. In that case, the atomic generators of $\CR$ coincide with the simple generators of $\CR^{a}$.
 \end{prop*}
 
 \begin{proof} Let $\CR$ be any protorootoid.  Note that if $\CR$ is preprincipal and $\CT:=\CR^{a}$ is principal, then $S_{\CT}=A_{\CT}=A_{\CR}$ by  Lemma \ref{ss3.7} and the comments at the end of \ref{ss3.10}. Those comments also show that (i)--(ii)   below hold: 
 \begin{conds}
 \item If $\CR$ is principal, then $\CR^{a}$ is principal.
 \item $\CR$ is preprincipal if and only if $\CR^{a}$ is preprincipal.
 \item If $\CR$ is principal, it is preprincipal.
 \item If $\CR$ is preprincipal and abridged, it is principal.
 \end{conds}
It is easily seen that (ii)--(iv)  imply ((i) and)  the first assertion of the Lemma, so we need only prove (iii)--(iv).
   
To prove (iii),  assume  that $\CR$ is principal. Then 
 $A:=A_{\CR}=S_{\CR}=:S$ and $l_{A}=l_{S}=l_{N}$. 
 For any 
 $a\in \ob(G)$, $g\in \lsub{a}{G}$ and $s\in \lsub{a}{A}$, either 
 $N(g)\cap N(s)=\eset$ or $N(s)\seq N(g)$  since $s\in S$ implies that $N(s)$ is an atom of $\lsub{a}{\L}$. Hence $\CR$ is  preprincipal, since it is faithful and interval finite.
 
 To prove (iv),  assume that $\CR=(G,\L,N)$ is preprincipal and abridged.
 Let $a\in \ob(G)$ and $s\in \lsub{a}{A}$ where $A:=A_{\CR}$.
 Define \begin{equation}\label{eq3.11.1}
 B:=\mset{u\in \lsub{a}{\L}\mid u\cap N_{s}\in \set{\eset,N_{s}}}.
 \end{equation}
 Using the fact that $\set{\eset,N_{s}}$ is a subring of 
 $\lsub{a}{\L}$, one easily checks that $B$ is a subring of 
 $\lsub{a}{\L}$.
 By the assumption that $\CR$ is preprincipal, $B$ contains
 $\lsub{a}{\CL}$. Since $\CR$ is abridged,  $\lsub{a}{\CL}$ generates  $\lsub{a}{\L}$ as ring, so $B=\lsub{a}{\L}$.
 This implies that $N(s)$ is an atom of $\lsub{a}{\L}$ i.e. $s\in S:=S_{\CR}$. Thus $A\seq S$. Since $\CR$ is preprincipal, it is 
 faithful  and interval finite, hence atomically generated by Lemma \ref{ss3.5}. By Lemma \ref{ss3.7}, it follows that $\CR$ is principal.
        \end{proof}
     
\subsection{}\label{ss3.12} The final characterization of principal protorootoids here is the following.
\begin{lem*} Let $\CR$ be a protorootoid. 
    \begin{num}\item  $\CR$ is principal if and only if  it is preprincipal and saturated.
    \item $\CR$ is preprincipal if and only if it is faithful, interval finite and pseudoprincipal.
    \end{num} 
\end{lem*}
\begin{proof}  For the proof of (a), suppose first that $\CR=(G,\L,N)$ is principal.
It is preprincipal by \ref{ss3.11}(iii). To show $\CR$ is saturated, let $g\in \lsub{a}{G}$.  It will suffice to show that a maximal chain $M$ in $[\eset,N(g)]_{\lsub{a}{\CL}}$ is also a maximal chain in $[\eset,N(g)]_{\lsub{a}{\L}}$. This is trivial if $g=1_{a}$. Assume
  $g\neq 1_{a}$ and write $M$ as 
\begin{equation}\label{eq3.12.1}
\eset \sneq N(x_{1})\sneq N(x_{1}x_{2})\sneq\ldots\sneq N(x_{1}\cdots x_{n}), \quad g=x_{1}\ldots x_{n}
\end{equation} in $[\eset,N(g)]_{\lsub{a}{\CL}}$. It is easy to show  that each $x_{i}\in S:=S_{\CR}$ (for instance, from the correspondence of chains such as $M$ with compatible expressions with value $g$, and the substitution property; see \ref{ss2.6}). Now any  chain in the Boolean interval  
$[\eset,N(g)]_{\lsub{a}{\L}}$ has length at most $\rank(N(g))=l_{N}(g)$, and any 
expression of $g$ as a product of elements of $S$ has length 
at least $l_{S}(g)$, so $l_{S}(g)\leq n\leq l_{N}(g)$. 
Since $\CR$ is principal, $l_{N}(g)=l_{S}(g)=n$. This implies that the above chain is a maximal chain in
 $[\eset,N(g)]_{\lsub{a}{\L}}$, so
$\CR$ is saturated.

Conversely, suppose that $\CR$ is preprincipal and saturated.
Let $S:=S_{\CR}$, $A:=A_{\CR}$. By Lemma 3.5, Lemma \ref{ss3.7} and the definition of preprincipal protorootoids, it  will suffice to show 
that $A\seq S$. Let $a\in \ob(G)$ and $s\in \lsub{a}{S}$.
Then since $N(s)$ is an atom of $\lsub{a}{\CL}$, it follows that $\eset\seq N(s)$ is a maximal chain in $[\eset, N(s)]_{\lsub{a}{\CL}}$. Since $\CR$ is saturated, this is also a maximal chain
in $[\eset, N(s)]_{\lsub{a}{\L}}$ i.e. $N(s)$ is an atom of 
$\lsub{a}{\L}$. This shows $s\in S$, so $A\seq S$ as required to complete the proof of (a).

To prove (b),  set  $A:=A_{\CR}$. Suppose first that $\CR$ is preprincipal. Then it is interval finite and faithful by assumption. Let $g,h\in \lsub{a}{G}$ with $N(h)\neq \eset$. Since $\CR$ is interval finite, there exists some $x\in \lsub{a}{A}$,  such that
$\eset\seq N(x)\seq N(h)$.  By definition of preprincipal rootoid, either
$N(x)\cap N(g)=\eset$ or $N(x)\seq N(g)$, which implies that $\CR$ is pseudoprincipal. Conversely, suppose that $\CR$ is faithful, interval finite and pseudoprincipal. Let $g\in\lsub{a}{G}$ and $s\in \lsub{a}{S}$.
Taking $h=s$ in the defining condition of pseudoprincipal protorootoid, there exists $x\in \lsub{a}{G}$ with $\eset\neq N(x)\seq N(s)$, and either  $N(x)\cap N(g)=\eset$, or $N(x)\seq N(g)$. Since $s$ is an atom and $\CR$ is faithful, it follows that $x=s$, and so $\CR$ is 
preprincipal.
 \end{proof}
\begin{rem*} The previous results show that principal  protorootoids are interval finite, regular, saturated and pseudoprincipal. 
Regular, saturated,  pseudoprincipal rootoids will be studied in subsequent papers as a generalization of principal rootoids.  \end{rem*}
   \subsection{}\label{ss3.13}
     In the case of  cocycle finite, principal and preprincipal  protorootoids,
     compatibility has the following  descriptions in terms of   length functions on $G$.
  \begin{lem*} Let  $\CR=(G,\L,N)$ be a protorootoid  and  $e=\lrsub{a_{0}}{[g_{1},\ldots, g_{n}]}{a_{n}}$ be an expression  in $G$ with value $g$.
   \begin{num}
   \item If $\CR$ is cocycle finite, then $e$ is compatible if and only if 
  $l_{N}(g)=\sum_{i=1}^{n}l_{N}(g_{i})$.
  \item   If $\CR$ is principal and $S:=S_{\CR}$, then $e$ is compatible if and only if 
  $l_{S}(g)=\sum_{i=1}^{n}l_{S}(g_{i})$.
  \item     If $\CR$ is preprincipal and $A:=A_{\CR}$, then $e$ is compatible if and only if 
  $l_{A}(g)=\sum_{i=1}^{n}l_{A}(g_{i})$.
    \end{num}
  \end{lem*}
  \begin{proof} Part (a) is trivial for  $n\leq 1$.
   Suppose now that $n=2$. Then  by Lemma \ref{ss3.2},
   \begin{multline*}l_{N}(g_{1}g_{2})=\rank(N(g_{1}g_{2}))=\rank(N(g_{1})+g_{1}N(g_{2}))\\ =\rank(N(g_{1})+\rank(g_{1}N(g_{2}))-\rank(N(g_{1}) \cap g_{1} N(g_{2}))\\
   =l_{N}(g_{1})+l_{N}(g_{2})-\rank(N(g_{1}) \cap g_{1} N(g_{2}))\end{multline*} since $[\eset, g_{1}N(g_{2})]\cong [\eset,N(g_{2})]$
   implies that  $\rank(g_{1}N(g_{2}))=l_{N}(g_{2})$.
  Thus $l_{N}(g_{1}g_{2})=l_{N}(g_{1})+l_{N}(g_{2})$ if and only if  
  $N(g_{1})\cap g_{1}N(g_{2})=\eset$, which holds if and only if  $g_{1}g_{2}$ is compatible. 
    In general,  (a) follows from the $n=2$ case  by induction using the substitution property (Lemma \ref{ss2.6}).  Part (b) follows from (a) since if $\CR$ is principal, then $l_{S}=l_{N}$ by definition. Part (c) follows from (b)  using  Proposition \ref{ss3.11}.
   \end{proof}
   \subsection{}\label{ss3.14} The following shows that, for principal or  cocycle finite protorootoids, the weak preorder can be equivalently expressed in terms of the appropriate  length functions in a manner similar to the standard definition for  weak order on Coxeter groups (cf. \cite{BjBr}). Similarly, orthogonality of morphisms can be expressed in terms of length functions (see \ref{ss2.7}).
   \begin{cor*} Let $\CR:=(G,\L,N)$ be a protorootoid, $a\in \ob(G)$, $x,y\in \lsub{a}{G}$ and  set $z:=x^{*}y$.\begin{num}
   \item If $\CR$ is cocycle finite, then $x\lsub{a}{\leq   }y$ if and only if  
   $l_{N}(y)=l_{N}(x)+l_{N}(z)$.
        \item   If $\CR$ is principal and $S:=S_{\CR}$ is its set of simple generators,   then $x\lsub{a}{\leq   }y$ if and only if  
   $l_{S}(y)=l_{S}(x)+l_{S}(z)$.  
    \item   If $\CR$ is preprincipal and $A:=A_{\CR}$ is its set of simple generators,   then $x\lsub{a}{\leq   }y$ if and only if  
   $l_{A}(y)=l_{A}(x)+l_{A}(z)$.    \end{num}
       \end{cor*}
       \begin{proof} This follows from Lemmas \ref{ss2.7} and 
       \ref{ss3.13}.\end{proof}
          \section{Rootoids}\label{s4}
  \subsection{Complemented protomeshes} \label{ss4.1} A protomesh $(R,L)$ is said to be \emph{complemented} if $R$ is unital and for each $A\in L$, $A^{\cp}:=1_{R}+A\in L$. \begin{lem*} Let $(R,L)$ be a complemented protomesh. 
\begin{num}\item   Suppose given a   family $(A_{i})_{i\in I}$ of elements of $L$ and $B\in L$ satisfying $A_{i}\cap B=\eset$ for all $i\in I$. If the join $A= \join_{i}A_{i}$ exists  in $L$, then $A\cap B=\eset$.
\item For all $A\in R$,  $(R,A+L)$ is a  complemented protomesh.
\item If $A\in L$, then $A+L$ has a maximum element $1_{R}$.
\end{num}
  \end{lem*}
  \begin{proof}  One has $A_{i}\seq B^{\cp}\in L$ for all $i$, so $A=\join A_{i}\seq B^{\cp}$ and $A\cap B=\eset$. This proves  (a).
  For (b), a typical element of $A+L$ is $A+B$ where $B\in L$.
  One has $B^{\cp}=1_{R}+B\in L$, so $(A+B)^{\cp}=(A+B)+1_{R}=A+(B+1_{R})\in A+L$.  For (c), $1_{R}=A+A^{\cp}\in A+L$.    \end{proof}
   \subsection{The JOP}\label{ss4.2}  The following condition ($*$)  on a    protomesh $(R,L)$, which is suggested by Lemma \ref{ss4.1}(a),   is called  the \emph{join orthogonality property (JOP)}.
   \begin{enumerate}\item[($*$)]  Suppose that  $(A_{i})_{i\in I}$ is a family  of elements of $L$ and $B\in L$ satisfies $A_{i}\cap B=\eset$ for all $i\in I$. If the join $A= \join_{i}A_{i}$ exists  in $L$, then $A\cap B=\eset$.
   \end{enumerate}
   By abuse of terminology, we may say that $L$ satisfies the JOP when ($*$) holds. 

 \subsection{Rootoids}\label{ss4.3} Rootoids are defined as faithful protorootoids such that their weak orders are complete meet semilattices which satisfy the JOP.  In full detail:
 
 \begin{defn*}  A protorootoid $\CR:=(G,\L,N)$ with big weak order $\CL$  is said to be  a \emph{rootoid}, or to be {\em rootoidal}, if it satisfies the following   conditions (i)--(iii):
\begin{conds}\item For all $a\in \ob(G)$, if $g,h\in \lsub{a}{G}$ with $N(g)=N(h)$, then $g=h$.
\item 
For  all $a\in \ob{G}$,  $\lsub{a}{\CL}:=\mset{N(g)\mid g\in \lsub{a}{G}}$ is a complete meet semilattice. 
\item  If $a\in \ob{G}$,   and $A_{i},A\in \lsub{a}{\CL}$ are  such that
$A_{i}\cap A=\eset$ for all $i\in I$ and 
$B:=\join_{i} A_{i}$ exists in $\lsub{a}{\CL}$, then $B\cap A=\eset$.
\end{conds}
\end{defn*}
A protorootoid satisfying (iii) is said to satisfy JOP, even if it does not satisfy (i)--(ii).  
 Rootoids are said to have a certain property of protorootoids if the underlying protorootoid has that property. For example,  a  rootoid  is \emph{complete} (resp., \emph{principal}) if it is  complete (resp., principal) as a protorootoid. Thus, a rootoid is complete if  and only if  each of its weak orders has a maximum element.  If $\CR$ is a complemented  rootoid, then each weak order $\lsub{a}{\CL}$ is a complete ortholattice and in particular, $\CR$ is  complete.
 \begin{rem*} (1) The conditions  for a protorootoid to be a rootoid depend only on its preorder isomorphism type, since (iii) may be  reformulated in terms of orthogonality of  morphisms in $G$.
 
 (2)  In particular, (1) implies that a protorootoid is a rootoid if and only if its abridgement is a rootoid, since abridgement doesn't change the underlying groupoid-preorder.
 \end{rem*}

 \subsection{} \label{ss4.4} Let  $\CR=(G,\L,N)$ 
  denote a  protorootoid with big weak order $\CL$.   Part (c) of the following is a technical consequence of the JOP which plays an important role in   subsequent papers.
\begin{lem*}   Let $x\in \lrsub{a}G{b}$.
\begin{num}\item The map $A\mapsto B:=x\cdot A= N_{x}\dotcup x(A)$
induces an order isomorphism
$ \mset{A\in \lsub{b}{\CL}\mid A\cap   N_{x^{*}}=\eset}\xrightarrow[\th]{\cong}
 \mset{B\in \lsub{a}{\CL}\mid B\sreq  N_{x}}$.
 \item The maps $\th$ and $\th^{-1}$ preserve whatever   meets and   joins exist in their domains (in the natural order induced by the appropriate weak order).
 \item  If $\CR$ is a rootoid, then $\domn(\th)$ (resp., $\cod(\th)$)  is a  complete join-closed meet subsemilattice of 
 $\lsub{b}{\CL}$ (resp., $\lsub{a}{\CL}$). Hence $\th$ (resp.,  $\th^{-1}$) preserves    meets or  joins (of  non-empty subsets  of its domain) which exist in  $\lsub{b}{\CL}$ (resp., in $\lsub{a}{\CL}$).\end{num}
\end{lem*}
\begin{rem*} The point of (c) in relation to (b) is that joins or meets  of non-empty subsets
of $\domn(\th)$ (resp., $\cod(\th)$)     exist or not, and have the same value if they exist,   whether taken in $\domn(\th)$ or in $\lsub{b}{\CL}$  (resp.,  in $\cod(\th)$ or in $\lsub{a}{\CL}$).
\end{rem*}
\begin{proof} Suppose $A\in \lsub{b}{\CL}$ with $A\cap   N_{x^{*}}=\eset$.
Let 
$B:=x\cdot A =N_{x}+ x(A)$.  Since $N(x)\cap x(A)=x(N_{x^{*}}\cap A)=\eset$, $B=N_{x}\dotcup x(A)$ is in the right hand side and $\th\colon A\mapsto x\cdot A=B$ is order-preserving.   Similarly, if 
$B\in \lsub{a}{\CL}$ with $ B\sreq  N_{x}$, then $A:=x^{*}\cdot B=N_{x^{*}}+x^{*}(B)$ where $N_{x^{*}}=x^{*}(N_{x})\seq x^{*}(B)$, so $A$ is in the left hand side and $\th^{-1}\colon B\mapsto x^{*}\cdot B=A$ is order preserving.
This proves (a). The assertions of (b) are true of any order isomorphism
whatsoever. 

For (c), $\domn(\th)$ is  closed under  taking meets of its non-empty subsets  in $\lsub{b}{\CL}$ since it is an order ideal of $\lsub{b}{\CL}$. Also, $\domn(\th)$ is closed under taking those joins which exist  in $\lsub{b}{\CL}$   of its  non-empty subsets  since $\CR$ satisfies JOP. Similarly,
$\cod(\th)$ is  closed  under taking   meets in $\lsub{a}{\CL}$  of non-empty subsets of $\cod(\th)$, since $N_{x}\in \lsub{a}{\CL}$  is a lower bound for any subset of $\cod(\th)$. Finally, 
$\cod(\th)$ is  closed  under  joins which exist   in $\lsub{a}{\CL}$ of non-empty subsets of $\cod(\th)$  since $\cod(\th)$ is an order coideal of $\lsub{a}{\CL}$. 
 \end{proof}
\subsection{}\label{ss4.5}
A  semilattice $X$ is said to be \emph{pseudocomplemented} if  it has a minimum element $0_{X}$ and for any $x\in X$, there exists $x'\in X$ such that
for $y\in X$, $y\meet x=0_{X}$ if and only if $y\leq x'$.  \begin{prop*} Let $\CR=(G,\L,N)$ be a principal rootoid with simple generators  $S:=S_{\CR}$. Let $a\in \ob(G)$.
\begin{num}
\item For all $x\in \lsub{a}{G}$, let $S(x):=\mset{s\in \lsub{a}{S}\mid N_{s}\seq N_{x}}$.  Then for a non-empty  family  $(x_{i})$ in $ \lsub{a}{G}$, one has  $\meet x_{i}=1_{a}$ if and only if  $\bigcap_{i}S(x_{i})= \eset$.
\item 
 If  $x_{i},y\in \lsub{a}{G}$ are  such that
$x_{i}\wedge y=1_{a}$  for all $i\in I\neq\eset$ and 
$x:=\join_{i} x_{i}$ exists in $\lsub{a}{G}$, then $x\wedge y=1_{a}$.
\item If $\CR$ is complete, then its weak orders are pseudocomplemented.
\end{num}\end{prop*} 
\begin{proof} Part (a) follows on noting that
 \begin{equation}\label{eq4.5.1} S\bigl(\meet x_{i}\bigr)=\bigcap S(x_{i})\end{equation} and that, for $x\in  \lsub{a}{G}$,     $S(x)=\eset$ if and only if  $x=1_{a}$. 
For part (b), note first that for $s\in \lsub{a}{S}$ and $x\in \lsub{a}{G}$, one has either $N_{s}\seq N_{x}$  or $N_{s}\cap N_{x}=\eset$,  since $N_{s}$ is an atom of  $\lsub{a}{\L}$. The JOP therefore implies that
\begin{equation} \label{eq4.5.2} S\bigl(\join x_{i}\bigr)=\bigcup_{i}S(x_{i}).\end{equation}
Part (b), which is formally similar to JOP,  follows readily from \eqref{eq4.5.1}, \eqref{eq4.5.2} and (a). To prove  (c), one  checks from (b) that if  $x\in \lsub{a}{G}$, the pseudocomplement of $x$ is  \begin{equation}\label{eq4.5.3}
x':=\join_{\substack{y\in \lsub{a}{G}\\ y\wedge x=1_{a}}} y.\qedhere
\end{equation} 
\end{proof} 
 \subsection{Complete semilattices} \label{ss4.6} 
Subsections \ref{ss4.6}--\ref{ss4.7} describe basic facts about a category of complete semilattices 
which is involved  in the definition in \ref{ss4.8}  of the main  categories of rootoids considered in these papers.

Recall that the category $\preordc$ denotes the category of preordered sets 
with morphisms given by  preorder-preserving maps, and $\posetc $ is 
its full subcategory of posets.  One may view $\preordc$ as a full subcategory of $\catc$, with objects those small categories for which there is at most one morphism between any two objects. Then
     $\posetc$ is the full subcategory of $\preordc$ with objects the preordered sets (as categories)  in which every isomorphism is an identity map.

The  category $\csl_{0}$ is the following  
subcategory of 
$\posetc$. The objects $\G$ of $\csl_{0}$ are  the non-empty complete 
meet semilattices $\G$, viewed as non-empty posets  in which every  non-empty subset 
has a  greatest lower bound. 
A morphism $\th\colon \L\to \G$ in $\csl_{0}$ is a function $\L\to \G$, preserving minimum elements, which also preserves  all meets of non-empty subsets of $\L$ and all joins which exist of  subsets of $\L$. That is,   $\th(0_{\L})=0_{\G}$,
 $\th(\meet X)=\meet\th(X)$ for all $\eset\neq X\seq \L$, $\th(\join Y)=\join(\th Y)$ for any  $Y\seq \L$ for which $\join Y$ exists. Such a morphism is order-preserving i.e. has an underlying morphism in $\posetc$ (to see this,  note that  $x\leq y$ if and only if  $x\meet y=x$, in any poset).
 \begin{rem*} If $\th\colon \L\to \G$ is a morphism in $\csl_{0}$ and the underlying map of sets $\L\to \G$ is bijective, then 
 $\th$ is an  isomorphism in  $\csl_{0}$.
 This follows by first using the preceding criterion for $x\leq y$ to show that $\th$ is an isomorphism in $\posetc$.\end{rem*} 
  \subsection{}\label{ss4.7} Let $\th\colon \L\to \G$ be a morphism in $\csl_{0}$.
 Let \begin{equation}\label{eq4.7.1}
  \G':=\mset{\g\in \G\mid \g\leq \th(\a)\text{ \rm for some $\a\in \L$}}\end{equation} denote the order ideal of $\G$ generated by  the  image of $\th$. Then $\G'$ with the induced order is an object of  $\csl_{0}$ and $\th$ restricts to a morphism $\th'\colon \L\to \G'$ in $\csl_{0}$.
 Note that, viewing $\L$, $\G'$ as categories and $\th'$ as a functor, $\th'$ has a left adjoint $\th^{\bot}\colon \G'\to \L$. That is, $\th^{\bot }$ is a  functor (order preserving map) $\G'\to \L$ satisfying 
 \begin{equation}\label{eq4.7.2}
 \Hom_{\G'}(\g,\th'(\a))\cong \Hom_{\L}(\th^{\bot}(\g),\a)
 \end{equation} for all $\a\in \L$ and $\g\in \G'$.  In fact, $\th^{\bot}$ is uniquely determined by the corresponding map of objects, which
 is given by 
 \begin{equation}\label{eq4.7.3}
 \th^{\bot}(\g):=\meet\mset{\a\in \L\mid \g\leq \th(\a)}=
 \min(\mset{\a\in \L\mid \g\leq \th(\a)})
 \end{equation}  for $\g\in \G'$. One has
 \begin{equation}\label{eq4.7.4}
 \g\leq \th(\a) \iff \th^{\bot }(\g)\leq \a
 \end{equation} for all $\g\in \G'$ and $\a\in \L$. 
 As a left adjoint, $\th^{\perp}$ preserves colimits, such as coproducts. Hence  its underlying map of objects $\G'\to \L$ preserves those joins  which exist in $\G'$, and  is order-preserving.
 
Henceforward, for any morphism $\th$ in $\csl_{0}$, $\th^{\bot}$ will be identified with its corresponding map of objects,
unless otherwise specified i.e. it will regarded as 
a function $\th^{\bot}\colon \G'\to \L$ where $\G':=\domn(\th^{\bot})$.
 Informally, it is often convenient to regard $\th^{\bot}$ as a partially defined, join preserving map $\G\to \L$, defined
only at elements of the order ideal of $\G$ generated by the image of $\th$.
  The following Lemma, which  is a trivial variant of   well known facts about adjoints of composite functors, implies  that $\th\mapsto \th^{\bot}$ gives a contravariant  functor to a  suitable category of complete meet semilattices with partially defined, join preserving maps.
  \begin{lem*} Let $\th_{i}\colon \L_{i-1}\to \L_{i}$ be morphisms in $\csl_{0}$ for $i=1, 2$. Let $\th:=\th_{2}\th_{1}\colon \L_{0}\to \L_{2}$.  Then  $ \domn(\th^{\bot})=\mset{\g\in \domn(\th_{2}^{\bot})\mid \th_{2}^{\bot}(\g)\in \domn(\th_{1}^{\bot})}$
  and  for 
  $\g\in \domn(\th^{\bot})$, one has $(\th_{2}\th_{1})^{\bot} (\g)= \th^{\bot}_{1}(\th^{\bot}_{2}(\g))$.
  \end{lem*}
\begin{proof}  For $\g\in \L_{2}$ and $\a\in \L_{0}$, one has \begin{multline*}
 (\th_{2}\th_{1})^{\bot} (\g)\leq \a \iff\g\leq  \th_{2}(\th_{1}  (\a))
\iff \th_{2 }^{\bot}(\g)\leq \th_{1}(\a)\iff  \th_{1}^{\bot}(\th_{2}^{\bot}(\g))\leq \a
\end{multline*} assuming all terms involved are defined. 
The remaining details are omitted.
\end{proof}

  \subsection{Categories of rootoids}\label{ss4.8} 
The categories $\rtd$ and $\rlec$ of rootoids defined below are those of primary concern. They are  defined  as  subcategories of an auxiliary  category $\rootc$. 
The category $\rtd$ (resp,  $\rlec$) is called the category of \emph{rootoids} (resp., of  \emph{rootoid local embeddings}).
The term ``morphism of rootoids'' (resp., ``local embedding of rootoids'') refer to a morphism of $\rtd$ (resp., $\rlec$) unless otherwise specified.
  
 \begin{defn*}\begin{num}\item  The category $\rootc$   is the following  subcategory of $\prootc$.   Its objects are  rootoids
 $(G,L,N)$. For any object $(G,L,N)$ of $\rootc$ and each $a\in \ob(G)$,  regard   $(\lsub{a}{G}, \lsub{a}{\leq  })$ as an object of $\csl_{0}$.  Morphisms in $\rootc$ are  
  morphisms  $(\alpha,\mu)\colon (G,\L,N)\to(G',\L',N')$ in $\prootc$ such that for each $a\in \ob(G)$, the map $\lsub{a}{\a}\colon \lsub{a}{G}\to \lsub{\a(a)}{G}'  $  induced by $\a$ is a morphism in $\csl_{0}$. 
  \item The category $\rtd$ of rootoids is the following subcategory of $\rootc$. It has the same objects as $\rootc$, and  a  morphism
   $(\alpha,\mu)\colon (G,\L,N)\to(G',\L',N')$ in $\rootc$ is a morphism in  $\rtd$ if and only if  for all $a\in \ob(G)$,
   $g\in \lsub{a}{G}$ and $g'\in \lsub{\a(a)}{G}'  $ with
   $g'$ and $\lsub{a}{\a}(g)$ orthogonal in $G'$ (i.e. $N'(\lsub{a}{\a}(g))\cap N'_{g'}=\eset$ in $\lsub{\a(a)}{\L}'  $) and with
   $g'\in\domn(\lsub{a}{\a}^{\bot})$, one has 
   $g$ and $\lsub{a}{\a}^{\bot}(g')$  orthogonal in $G$ (i.e.
   $N(g)\cap N(\lsub{a}{\a}^{\bot} (g'))=\eset$ in $\lsub{a}{\L}$).
   \item The category $\rlec$ of local embeddings of rootoids is the subcategory of $\rtd$ with all objects, and morphisms  $(\alpha,\mu)\colon (G,\L,N)\to(G',\L',N')$ in $\rtd$ such that for each $a\in \ob(G)$, the map $\lsub{a}{\a}\colon \lsub{a}{G}\to \lsub{\a(a)}{G}'  $  in $\csl_{0}$ induced by $\a$ is injective and its image is a join-closed meet subsemilattice of  $\lsub{\a(a)}{G}'  $.  \end{num}
      \end{defn*}
  It is straightforward to check that $\rootc$, $\rtd$ and $\rlec$ are categories;     the fact that composites of morphisms   in $\rtd$  are again morphisms in $\rtd$ is a consequence of Lemma \ref{ss4.7} and the definitions.
  
  The additional condition required in (b) for a morphism $f=(\a,\mu)$ of 
  $\rootc$ to be a morphism in $\rtd$ will be called the \emph{adjunction orthogonality property} or \emph{AOP} for short. Thus, a morphism in $\rtd$ is a morphism in $\rootc$ satisfying AOP.
  Notice that  whether a morphism $f=(\a,\mu)$ in $\prootc$ is a morphism in $\rootc$, $\rtd$ or $\rlec$ depends only on the underlying morphism $\a=\mathfrak{P}(f)$ of groupoid-preorders.

 \subsection{Coverings} \label{ss4.9}
 The following Lemma lists  useful properties of covering morphisms. 
 \begin{lem*} Let $f=(\alpha,\nu)\colon \CR'\to \CR$ be a covering morphism of protorootoids. Write $\CR=(G,\L,N)$ and  $\CR' =(H,\L',N')$.  Let $A:={A}_{\CR}$, $A':={A}_{\CR'}$,  
 $S:={A}_{\CR}$, $S':={A}_{\CR'}$ denote the 
 sets of atomic  and simple morphisms of $\CR$ and $\CR'$.
 
\begin{num}\item The weak order of $\CR'$ at an object $a$ of $H$ identifies naturally with the weak order of $\CR$ at the object $\a(a)$ of $G$.
\item For $a\in\ob(G)$, one has  $\lsub{a}{A}'  =\mset{s\in \lsub{a}{H}\mid  \a(s)\in \lsub{\a(a)}{A}}$ and  
$\lsub{a}{S}'  =\mset{s\in \lsub{a}{H}\mid  \a(s)\in \lsub{\a(a)}{S}}$.
If $A$ (resp., $S$) generates $G$, then $A'$ (resp., $S'$) generates $H$; the converses hold if $f$ is a covering quotient morphism.
 \item If $\CR$ is a (faithful, complete, interval finite, cocycle finite, preprincipal, principal, pseudoprincipal, regular or rootoidal)  protorootoid, then so is $\CR'$; the converses hold if  $f$ is a covering quotient morphism.
\item If $\CR$ is a rootoid then $f$ is a morphism in $\rtd$ and in $\rlec$.\end{num} \end{lem*} 
\begin{rem*} Note that it follows from (d)  that 
an isomorphism in $\prootc$ between two rootoids
is an isomorphism in $\rtd$ and  $\rlec$.\end{rem*} 
\begin{proof} Let $\CL'$ denote the big weak order of $\CR'$. The map $\nu_{a}$ gives an isomorphism  $\lsub{a}{\L}'\xrightarrow{\cong}\lsub{\a(a)}{\L}$ of Boolean rings such that  for $h\in \lsub{a}{H}$, $\nu_{a}(N'(h))=N(\a(h))$. Since the map
$h\mapsto \a(h)\colon \lsub{a}{H}\to \lsub{\a(a)}{G}$ is bijective by definition of covering morphism, 
 the definitions imply, slightly more strongly than (a),  that $\nu_{a}$ induces an isomorphism of protomeshes $(\lsub{a}{\L}'  ,\lsub{a}{\CL}'  )\cong (\lsub{\a(a)}{\L},\lsub{\a(a)}{\CL})$.  
 
  We prove the part of (b) concerning  simple morphisms.   Note first that
 (a)  and its proof  imply that
 \begin{equation}\label{eqa4.9.1} l_{N'}(h)=l_{N}(\a( h)) \text{ \rm for all $h\in \mor(H)$}.\end{equation} Hence  \begin{equation}\label{eqa4.9.2} S'=\mset{s'\in \mor(H)\mid \a(s')\in S}.\end{equation}
 Using the definition of covering morphism, it follows  that if $S$ generates $G$, then $S'$ generates $H$, and conversely if $\ob(H)\to \ob(G)$ is surjective. Assume that $S$ generates $G$.
  A simple argument using \eqref{eqa4.9.2} and the definition of covering morphism 
  shows that \begin{equation}\label{eqa4.9.3} l_{S'}(h)=l_{S}(\a(h)) \text{ \rm  for all $h\in \mor(H)$}.\end{equation} 
   The desired 
  conclusions in (b) involving simple morphisms follows readily from these facts and the definitions. The proof of the parts of  (b) involving atomic morphisms are similar using 
  \begin{equation}\label{eqa4.9.4} A'=\mset{s'\in \mor(H)\mid \a(s')\in A}\end{equation}
and, when $A$ generates $G$,   \begin{equation}\label{eq4.9.5} l_{A'}(h)=l_{A}(\a(h)) \text{ \rm  for all $h\in \mor(H)$.}\end{equation} 
Part (c) follows easily from (a)--(b) and the above formulae,   by the definitions.

Now for (d). There are  inclusion functors \begin{equation}\label{eqa4.9.6}
   \xymatrix{{\rlec}\ar[r]&{\rtd}\ar[r]&{\rootc}\ar[r]&{\fprc}\ar[r] &{\prootc.}   }
   \end{equation} 
    It will be  shown that $f$ is a morphism  in each of these categories, working from right to left. By assumption, 
    $f$ is a morphism is $\prootc$
       Next,  $f$ is an morphism is $\fprc$, since $\CR'$ is faithful by (c) and  $\fprc$ is a full subcategory of $\prootc$.
       Also by (c), $\CR'$ is a rootoid since $\CR$ is a rootoid.
       For any $a\in \ob(H)$, the map $\lsub{a}{\a}\colon \lsub{a}{H}\to \lsub{\a(a)}{G}$ is an order isomorphism by (a). Since 
       $ \lsub{\a(a)}{G}$ is in $\csl_{0}$, it follows that 
       $ \lsub{a}{\a}$ is an isomorphism in $\csl_{0}$, and hence
       $f$ is a morphism in $\rootc$; further, if $f$ is a morphism
       in $\rtd$, it is a morphism in $\rlec$. To see that $f$ is in $\rtd$, we must verify AOP.
        Let $a,b\in \ob(H)$,
   $h\in \lrsub{a}{H}{b}$ and $g\in \lsub{\a(a)}{G}$ with
   $g$ and $\lsub{a}{\a}(h)$ orthogonal in $G$
   i.e.  $(\a_{a}(h))^{-1}\lsub{\a(b)}{\leq  }(\a_{a}(h))^{-1}g$. Now    $\lsub{b}{\a}$ and      $\lsub{a}{\a}$ are  order isomorphisms. Applying $\lsub{b}{\a}^{-1}$ to the last equation gives 
   $h^{-1}\lsub{b}{\leq  }h^{-1}g'$ where 
   $g'=(\lsub{a}{\a})^{-1}(g)=\lsub{a}{\a}^{\perp}(g)$.
    Hence $h$ and $g'$ are orthogonal in $G$ as 
   required.  This completes the proof of (d).
       \end{proof}

\section{Root systems and set  protorootoids}\label{sa5}
\subsection{Signed sets}\label{ssa5.1}  The \emph{sign group} $\set{\pm}$ is the group of order two with elements 
$\set{+,-}$
and identity element $+$. Sometimes it is  identified  with the group $\set{\pm 1}$ of units of the ring
 $\Int$ of integers.

An action of a  group $H$ on a set $\Theta$ is said to be \emph{free} if the 
stabilizers in $H$ of all elements  of  $\Theta$ are trivial. An \emph{indefinitely signed 
set} is defined to be  a set $\Theta$ together with a free left  action of the sign group. 
A \emph{definitely signed set} is a pair $(\Th,\Th_{+})$ of an indefinitely signed set $\Theta$
together with a specified set $\Theta_{+}\seq \Theta$ of $\set{\pm}$-orbit 
representatives on $\Theta$.
Then $\Theta=\Theta_{+}\dotcup \Theta_{-}$ where $\Theta_{-}:=-\Theta_{+}=\mset{-s\mid s\in \Theta_{+}}$.

Let $\setc $ (resp.,  $\setc _{\set{\pm}}$, $\setc _{\pm}$) denote the categories of sets (resp.,
indefinitely signed sets, resp., definitely signed sets) with functions (resp.,
$\set{\pm }$-equivariant functions,  $\set{\pm }$-equivariant functions) as morphisms. 

Note especially that it is  not required that  a 
morphism $\Theta\to \Theta'$ in $\setc _{\pm}$  be \emph{positivity preserving} i.e. that it induce  a  map 
$\Theta_{+}\to \Theta'_{+}$ (and therefore map  $\Theta_{-}$ to $\Theta'_{-}$). The subcategory of $\setc _{\pm}$
consisting of all its objects but with only positivity preserving morphisms will be denoted $\setc _{+,-}$. There is an  
equivalence of categories  $\setc _{+,-}\xrightarrow{\cong}\setc $ given on objects by $\Theta\mapsto \Theta_{+}$.

 \subsection{Signed groupoid-sets}\label{ssa5.2}
  The category $\grpdset$ has as its  objects the pairs $R=(G,\Phi)$ where $G$ is a  groupoid and $\Phi\colon G\to \setc_{\pm}$ is a 
functor.  Its objects are called \emph{signed groupoid-sets}. 
A morphism $(G,\Phi)\to (H,\Psi)$ is defined to be  a pair
$(\a,\nu)$ where $\a\colon G\to H$ is a functor and $\nu\colon \Psi\a\to \Phi$ is a natural transformation of functors $G\to \setc_{\pm}$ such that for each $a\in \ob(G)$, the component 
$\nu_{a}\colon \Psi(\a(a))\to \Phi({a})$ is \emph{positivity preserving} i.e. each component is a morphism in the subcategory $\setc_{+,-}$ of $\setc_{\pm}$. For another morphism $(\b,\mu)\colon (H,\Psi)\to (K,\L)$, the composite $(\b,\mu)(\a,\nu)\colon (G,\Phi)\to (K,\L)$ is defined by  $(\b,\mu)(\a,\nu)=(\b\a,\nu (\mu\a))$ where the component of $\nu (\mu\a)$ at $a\in \ob(G)$ is  $(\nu (\mu\a))_{a}=\nu_{a}\mu_{\a(a)}$.

For fixed $G$, the subcategory $G\text{\rm -}\setc_{\pm}$ of 
$\grpdset$ containing   objects $(G,\Phi)$ and morphisms $(\Id_{G},\nu)$ is called the category of signed $G$-sets.  

The functor $\Phi$ is called the \emph{root system} of a signed groupoid-set  $(G,\Phi)$.  It may sometimes  be regarded  as a family of definitely signed sets $(\lsub{a}{\Phi})_{a\in \ob(G)}$ (or, if $G$ is a group,   as the unique signed set in that family) with
action maps $\lrsub{a}{G}{b}\times \lsub{b}{\Phi}\to \lsub{a}{\Phi}$ satisfying suitable conditions (associativity, inverse and unit axioms) similar to those for $H$-sets for a group $H$. 

 An element   of  $\lsub{a}{\Phi}$ (resp.,
$\lrsub{a}{\Phi}{+}$,   $\lrsub{a}{\Phi}{-}$)
 for some $a\in \ob(G)$, is called   a  \emph{root} (resp.,  \emph{positive root}, \emph{negative root}) of $R$ or $\Phi$.
  
\begin{rem*} (1) Note that  
morphisms of signed sets induced by the action of groupoid 
elements  in a signed groupoid-set are \emph{not} positivity preserving in general.

(2) 
There is also a (different)  category $\grpdset'$  with  signed groupoid-sets as objects  in which a morphism $(G,\Phi)\to (H,\Psi)$ is a pair $(\a,\nu)$ where $\a\colon G\to H$ is a functor and $\nu\colon \Phi\to \Psi\a$ is a natural transformation with positivity preserving components. Morphisms in this latter category induce  commutative diagrams of  action maps
\begin{equation}\label{eqa5.2.1}
\xymatrix{
{\lrsub{a}{G}{b}\times \lsub{b}\Phi}\ar[r]\ar[d]&\lsub{a}{\Phi}\ar[d]\\
{\lsub{\a(a)}{H}_{\a(b)}\times \lsub{\a(b)}\Psi}\ar[r]&\lsub{\a(a)}{\Psi}
}
\end{equation}  However, it is $\grpdset$     in which the usual product (see \cite{Bour})
 \begin{equation} \label{eqa5.2.2}(W_{1},\Phi_{1})\times (W_{2},\Phi_{2})=(W_{1}\times W_{2},\Phi_{1}\coprod \Phi_{2})\end{equation} of root systems of Coxeter groups may be interpreted as a categorical product, and 
which is related to the category $\prootc$;  $\grpdset'$ is similarly related to $\prootcp$.  
 \end{rem*}

  \subsection{}\label{ssa5.3} The  category of set protorootoids defined below provides a convenient bridge between the categories $\prootc$ of  
 protorootoids and $\grpdset$  of signed groupoid-sets.
   \begin{defn*}\begin{num}\item   A set protorootoid  is  defined to be  a  triple  $(G,\L,N)$ such that $G$ is a groupoid, $\L\colon G\to \setc $ is a representation of $G$ in $\setc $ and $N$ is a $G$-cocycle for $\wp _{G}(\L)$.
 \item A set protorootoid $(G,\L,N)$ is called a \emph{set rootoid} if 
 $(G,\wp _{G}(\L),N)$ is a rootoid.
\item  The category $\setproot$  has  set protorootoids
 $(G,\L,N)$  as objects.  A morphism $(G,\L,N)\to (H,\G,M)$
 in $\setproot$ 
 is a pair $(\a,\nu)$ consisting of a groupoid homomorphism $\a\colon G\to H $ and a natural transformation $\nu\colon \G\a\to \L$   such that   for any $g\in \lsub{a}{G}$, $M_{\a(g)}=\nu_{a}^{-1}(N_g):=\mset{p\in \lsub{\a(a)}{\G}\mid \nu_{a}(p)\in N_{g}}$.
 Composition of morphisms is given by $(\b,\mu)(\a,\nu)=(\b\a,\nu(\mu\a))$.\end{num}\end{defn*}
  \subsection{} \label{ssa5.4} There is   a  functor 
$\mathfrak{I}\colon \setproot\to \prootc$ as follows.
 Directly from the definitions,   if  $\CR=(G,\L,n)$ is a set protorootoid, then
   $\mathfrak{I}(\CR):=(G,\wp_{G}(\L),N)$ is a protorootoid. Further, if 
   $(\a,\nu)\colon (G,\L,N)\to (H,\G,M)$ is a morphism of set 
   protorootoids, then \begin{equation*} (\a,\wp_{G}(\nu))\colon (G,\wp_{G}(\L),N)\to (H,\wp_{H}(\G),M)\end{equation*} is a morphism in 
   $\prootc$. To see this, note that in \ref{ssa5.3}(c), $\nu_{a}^{-1}(N_{g})=(\wp_{G}(\nu))_{a}(N_{g})$ and that 
   $\wp_{G}(\G\a)=\wp \G\a \iota_{G}=\wp \G\iota_{H}\a=\wp_{H}(\G)\a$.  This defines $\mathfrak{I}$ on objects and morphisms. 
   Using $\wp_{G}(\nu\a)=\wp_{G}(\nu)\a$,  one checks $\mathfrak{I}$ is a functor as claimed. Note that $\mathfrak{I}$ is faithful since $\wp$ is faithful. 
   
    \begin{prop*}  Let $\CR$ be a principal protorootoid. Then there is a set protorootoid $\CT$ such that $\ab(\mathfrak{I}(\CT))\cong \ab(\CR)$.\end{prop*}
     \begin{proof}   Write  $\CR=(G,\L,N)$ and denote its simple 
     generators as  $S$ and abridgement $\CR^{a}$ as $(G,\L',N')$. The 
     cocycle condition implies  that for $a\in \ob(G)$, $\L'(a)$ is 
     generated as Boolean ring  by the elements $g(N_{s})$ for 
     $g\in \lrsub{a}{G}{b}$ and $s\in \lsub{b}{S}$.  Let 
     $\lsub{a}{\Phi}$ denote the set of all these elements.  
     Note that elements of $\lsub{a}{\Phi}$ are atoms of $\L(a)$, 
     hence of $\L'(a)$. Thus, $\lsub{a}{\Phi}$ is a set of orthogonal 
     idempotents generating  $\lsub{a}{\L}'$, which is therefore 
     isomorphic to the Boolean ring of finite subsets of $\lsub{a}{\Phi}$. 
     Note that there is a natural representation of $G$  on the sets 
     $\lsub{a}{\Phi}$, corresponding to a functor $\Phi\colon G\to \setc$ 
        with $\Phi(a)=\lsub{a}{\Phi}$. Clearly, $\L'$ is equivalent to the 
        subrepresentation $\wp'_{G}(\Phi)$  of $\wp_{G}(\Phi)$ such that 
         $\lsub{a}{(\wp'_{G}(\Phi))} $ is the set of all finite subsets of 
 $\lsub{a}{(\wp_{G}(\Phi))} $. Let 
 $M\colon G\to \dotcup_{a\in \ob(G)}\lsub{a}{(\wp_{G}(\Phi))} $ be 
 the function defined by $M(g):=\mset{x\in\lsub{a}{\Phi}\mid x\seq N'(g)}\in \lsub{a}{(\wp'_{G}(\Phi))} $.  Using  Lemma \ref{ss3.2}, $M$ is a $G$-cocycle for $\wp_{G}(\Phi)$. It is easy to see from this that $\CT:=(G,\Phi,M)$ is a set protorootoid with the  required property.
 \end{proof} 

   \begin{rem*} One version of Stone's theorem implies that any Boolean ring has a (canonical) realization as a ring of sets. It is easy to see from it that any protorootoid is  preorder isomorphic to a protorootoid in the image of $\mathfrak{I}$. 
  A more precise statement of this fact, and a generalization of the above proposition, will be proved in a 
  subsequent paper.\end{rem*}
 
  \subsection{}    \label{ssa5.5} The  following result  is a special case
  (involving only fibers $\set{\pm}$) of  facts about  the analogues of bundles  in the category of groupoid representations in the category of sets  (cf. \cite{DyRig} for related results for groups).   \begin{prop*} The categories $\setproot$ and $\grpdset$ are equivalent.
    \end{prop*}
\begin{proof} We first construct a functor $\mathfrak{K}\colon 
\setproot \to\grpdset$. Let $(G,\L,N)$ be a set protorootoid. 
Define a functor $\Phi\colon G\to\setc_{\pm}$ as follows.
For $a\in \ob(G)$, let $\Phi(a)=\lsub{a}{\Phi}:=\lsub{a}{\L}\times \set{\pm}$.  Regard it as a  definitely signed set
with $\set{\pm }$ action, which we write as $(\epsilon,\a)\mapsto \e\a$, by multiplication on the second factor
and  with $\lrsub{a}{\gp \Phi}{\e}:=\lsub{a}{\gp \L}\times\set{\e}$ for $\e\in \set{\pm}$
 and $a\in \ob(G)$. 
For $a\in \ob(G)$, $g\in \lrsub{a}{G}{b}$, $x\in\lsub{b}{\L}$  and $\e\in \set{\pm}$, set
\begin{equation} \label{eqa5.5.1}
\Phi(g)(x,\e):=\bigl((\L(g))(x),\e\e'\bigr), \quad \e'=\e'_{g,x}:=\begin{cases}-,&x\in N(g^{*})\\
                                   +,&x\not\in N(g^{*}).\end{cases}\end{equation}
One can check this defines a signed groupoid-set  $(G,\Phi)$, and we set $\mathfrak{K}(G,\L,N)=(G,\Phi)$.
Next, suppose given a morphism $(\a,\nu)\colon (G,\L,N)\to (H,\G,M)$. Write $\mathfrak{K}(H,\G,M)=(H,\Psi)$.
Define a natural transformation  $\nu'\colon \Psi\a\to \Phi$ which 
has component $\nu'_{a}\colon \Psi\a(a)\to \Phi(a)$ at $a$ equal to the map \begin{equation*} \nu_{a}\times \Id_{\set{\pm}}\colon \lsub{\a(a)}\G\times\set{\pm }\to \lsub{a}{\L}\times \set{\pm}.\end{equation*} Using  \eqref{eqa5.5.1}, it can be checked  that $(\a,\nu')\colon (G,\Phi)\to (H,\Psi)$ is a morphism  in $\grpdset$ and that setting $\mathfrak{K}(\a,\nu)=(\a,\nu')$ defines a functor $\mathfrak{K}$ as required.

A functor $\mathfrak{L}\colon\grpdset  \to \setproot $ defining an inverse equivalence may be constructed as follows.  Suppose that  $(G,\Phi)$ is a signed 
groupoid-set.
For $a\in \ob(G)$, let $\lsub{a}{\L}:=(\lsub{a}\Phi)/\set{\pm 1}$ be the $\set{\pm}$-orbit space. Since morphisms of signed $G$-sets are $\set{\pm}$-equivariant, 
there is a natural  functor $\L\colon G\to \setc$ induced by $\Phi$  with  $\L(a)=\lsub{a}{\L}$ for all $a\in \ob(G)$.
Let $\pi_{a}\colon \lsub{a}\Phi\to\lsub{a}{\L} $ be the orbit map
$\a\mapsto \set{\pm \a}$.
For $g\in \lrsub{a}{G}{b}$, define
\begin{equation}\label{eqa5.5.2}
\lsub{a}{N}'  (g):=\lrsub{a}{\Phi}{+} + \Phi(g)(\lrsub{b}{\Phi}{+})\in \wp(\lsub{a}{\Phi}).
\end{equation}  
This defines  a $G$-cocycle (in fact, a coboundary) $N'$ for  $\wp_{G}(\Phi)$. Note that
\begin{equation}\label{eqa5.5.3}
\lsub{a}{N}'  (g)=\bigl(\lrsub{a}{\Phi}{+} \cap \Phi(g)(\lrsub{b}{\Phi}{-} )
\bigr)\dotcup \bigl(\lrsub{a}{\Phi}{-} \cap \Phi(g)\bigl(\lrsub{b}
{\Phi}{+})\bigr)
\end{equation}  
and in particular 
 $N'(g)=\mset{-\a\mid \a\in N'({g})}$. Let
\begin{equation}\label{eqa5.5.4}
\lsub{a}{N}(g):=\pi_{a}(N_{g}')=\pi_{a}\bigl(\lrsub{a}{\Phi}{+} \cap
 \Phi(g)(\lrsub{b}{\Phi}{-} )\bigr)\in \wp(\lsub{a}{\L}).
\end{equation}  It follows immediately that this defines a 
  $G$-cocycle $N$ for $\wp_{G}(\L)$, so $(G,\L,N)$ is a set 
  protorootoid. Set $\mathfrak{L}(G,\Phi)=(G,\L,N)$. 
The map $\mathfrak{L}$ so defined on objects extends naturally to a 
functor $\mathfrak{L}$ with the desired properties.
  \end{proof} 
\begin{rem*} (1) In a signed groupoid-set $(G,\Phi)$, replacing  each  set $\lrsub{a}{\Phi}{+} $ of $\set{\pm}$ orbit
 representatives by another has the effect of  replacing the cocycle in 
 the corresponding set-protorootoid by another in the same 
 cohomology class. 
In general, there is no close   relation between the corresponding weak 
orders; if  the groupoid is connected and  simply connected, this is clear from Lemma \ref{ss1.14}.

 (2) In a similar manner as one defines 
$\grpdset$, one can define a  category 
 $\grpdc\text{\bf -}\balgc_{\pm}$ of groupoid representations in suitably defined category of signed Boolean algebras, in which  morphisms are   natural transformations with positivity-preserving components. For example, for a signed set $S$,
 one has a signed Boolean algebra $\wp(S)$ with decomposition $\wp(S)\cong \wp(S_{+})\times \wp(S_{-})$ into positive and negative Boolean subalgebras corresponding to $S=S_{+}\coprod S_{-}$. The equivalence in the  above  Proposition may be viewed as a restriction of an equivalence between   $\grpdc\text{\bf -}\balgc_{\pm}$  and $\prootc_{1}$.  
 
(3)  One may also define similarly   a subcategory  
 $\catc\text{\bf -}\abcatc_{\pm}$  of the category of  functors from small categories to a category of  small signed ab-categories.  The equivalence mentioned  in (2)  may be viewed as a restriction   of an equivalence of 
 $\catc\text{\bf -}\abcatc_{\pm}$ with a   category of triples  $(G,\L,N)$ where $G$ is a small category,  $\L$ is a  functor from $G$  to the category of small ab-categories, and $N$  is an idempotent-valued $1$-cocycle   for a naturally associated non-abelian cohomology theory.   Detailed statements and proof of the  facts in  (2)--(3) lie outside the scope of  these papers. 
 \end{rem*}  

  \subsection{Terminology for signed groupoid-sets and set protorootoids}\label{ssa5.6}
It is convenient to transfer  terminology defined for  protorootoids  to set protorootoids and signed groupoid-sets in the following ways. Unless otherwise specified, 
a set protorootoid $T=(G,\L,N)$ will be said to have  a property  defined for protorootoids if the corresponding protorootoid $\mathfrak{I}(T)$ has that same property.
Similarly, a signed groupoid-set $R=(G,\Phi)$   is, unless otherwise specified, said to have such a property if the corresponding   protorootoid   $\mathfrak{I}\mathfrak{L}(R)$   has that property. In particular, this 
convention can be applied to define 
simply connected (connected,  complemented, complete, simply or atomically generated, cocycle or interval finite, principal, preprincipal, pseudoprincipal, complete, saturated, rootoidal etc) set protorootoids and signed groupoid-sets, and faithful set protorootoids. Rootoidal set protorootoids are called set rootoids.  To avoid confusion with the standard notion of a faithful $G$-set (e.g. for $G$ a group), a signed groupoid-set $(G,\Phi)$ is called strongly faithful if the corresponding  protorootoid is faithful.

Similarly, the simple or atomic morphisms (in $\mor(G)$) of  $T$ (resp., $R$) are defined as the simple or atomic  morphisms of  $\mathfrak{I}(T)$ (resp., $\mathfrak{I}\mathfrak{L}(R)$), etc. This makes sense since $\mathfrak{I}$ and $\mathfrak{L}$ preserve the underlying groupoid.

 \subsection{}\label{ssa5.7} For convenience, this subsection explicitly spells out  the definition of rootoidal signed groupoid-sets, following the  conventions in \ref{ssa5.6}, and fixes additional terminology and notation  concerning them.

Let $R=(G,\Phi)$  be a
signed groupoid-set and $\CR=(G,\L,N):=\mathfrak{L}(R)$ denote the corresponding set protorootoid.  Use notation as in the proof of Proposition  \ref{ssa5.5}. For any morphism
 $g\in \lrsub{a}{G}{b}$, define 
 \begin{equation}    \label{eqa5.7.1}\Phi_{g}:= \lrsub{a}{\Phi}{+} \cap g(\lrsub{a}{\Phi}{-} ).
 \end{equation} Note that $\pi_{a}$ induces a bijection $\Phi_{g}\cong N_
 {g}$ of sets. The  $G$-cocycle (in fact, coboundary) $N'$ for $\wp_{G}(\Phi)$ from which $N$ was defined in the proof of \ref{ssa5.5} is, in this notation, $g\mapsto N'(g)=  \Phi_{g}\dotcup -\Phi_{g}$. 
 
     The weak order $\set{N(g)\mid g\in \lsub{a}{G}}$ at $a$ of $\CR$  identifies (as poset) with the inclusion 
    ordered set $\lsub{a}{L}:=\mset{\Phi_{g}\mid g\in \lsub{a}{G}}$ of subsets of $\lrsub{a}{\Phi}{+} $ via the order 
    isomorphism $N_{g}\mapsto \Phi_{g}$ for $g\in \lsub{a}{G}$;
    this  poset $\lsub{a}{L}$ will be  called the weak order of $R$ at $a$.

    It follows directly from the above that $R$ is rootoidal (i.e. by definition, $\CR$ is a set rootoid)  if and only if  the following conditions (i)--(iii) hold:
    \begin{conds}\item $R$ is strongly faithful i.e. if $g\in \lsub{a}{G}$ with $\Phi_{g}=\eset$ then $g=1_{a}$.
    \item For each $a\in \ob(G)$, $\lsub{a}{L}:=\mset{\Phi_{g}\mid g\in \lsub{a}{G}}$ is a complete meet semilattice in the order induced by inclusion of subsets of $\lrsub{a}{\Phi}{+} $.
    \item  Given $a\in \ob(G)$,   a non-empty family $(A_{i})_{i\in I}$ in $\lsub{a}{L}$ and $B\in \lsub{a}{L}$ such that $A_{i}\cap B=\eset$ for all $i\in I$, if $A:=\join_{i\in I} A_{i}$ exists in $\lsub{a}{L}$, then $A\cap B=\eset$.
       \end{conds}
       \subsection{} \label{ssa5.8} Finally, for  convenience of reference,  we give the following diagram of functors, which provides considerable latitude in formulation of many  results  about   rootoids and protorootoids.        
      
        \begin{equation*}
       \xymatrix{
       {\grpdset}\ar@<1ex>[r]_-{\cong}^-{\mathfrak{L}}&
       {\setproot }\ar@<1ex>[r]^-{\mathfrak{I}'}
       \ar@<1ex>[l]^-{\mathfrak{K}}
        \ar@/^2.5pc/[rr]_-{\mathfrak{I}}&
       {{\prootc}_{1}}\ar@<1ex>[r]^-{\mathfrak{I}''}
        \ar@<1ex>[l]^-{\mathfrak{J}'}
       &
       {\prootc}\ar@<1ex>[r]^-{\mathfrak{A}}
       \ar@/^2.5pc/[rr]_-{\mathfrak{P}}
       \ar@<1ex>[l]^-{\mathfrak{J}''}
        \ar@/^2.5pc/[ll]
	_-{\mathfrak{J}}&
       {\prootc^{a}}\ar@<1ex>[r]^-{\mathfrak{P}'}\ar@<1ex>[l]^-{\mathfrak{B}}&
       {\gpdpreord}\ar@<1ex>[l]^-{\mathfrak{Q}'} \ar@/^2.5pc/[ll]
	_-{\mathfrak{Q}}   }    
       \end{equation*}

      In this, the functors $\mathfrak{L}$, $\mathfrak{I}$, $\mathfrak{A}$, $\mathfrak{P}$, $\mathfrak{B}$ and $\mathfrak{K}$
have been previously defined. The functor $\mathfrak{I}''$ is the evident inclusion functor, and $\mathfrak{I}$  obviously factors via 
$\prootc_{1}$ to give a (unique) functor  $\mathfrak{I}'$ such that
$\mathfrak{I}=\mathfrak{I}''\mathfrak{I}'$.
The forgetful functor 
$\mathfrak{P}$ also obviously factors through abridgement $\ab$  to define the functor  $\mathfrak{P}'$ such that  $\mathfrak{P}=\mathfrak{P}'\mathfrak{A}$.  Since  $\mathfrak{A}\mathfrak{B}=\Id$, one has $\mathfrak{P}':=\mathfrak{P}\mathfrak{B}$.  The remaining   functors  $\mathfrak{J}'$, $\mathfrak{J}''$,  $\mathfrak{J}=\mathfrak{J}'\mathfrak{J}''$, $\mathfrak{Q}'$ and   
    $\mathfrak{Q}=\mathfrak{B}\mathfrak{Q}'$
    (and additional functors  involving $\gpdpreord_{P}$ which are relevant   to  the remarks in \ref{ss2.16})  will be described in  subsequent papers.           
   The functors    are defined  so that  each  upper  functor  (with arrow to the right)  is right adjoint to the symmetrically  corresponding lower functor (with arrow to the left). 
\section{Examples of rootoids}\label{sa6}
  \subsection{Coxeter systems}\label{ssa6.1}
 Let $S$ be a set and $M$ be a $S$-indexed Coxeter matrix. This means that
$M=(m_{r,s})_{r,s\in S}$ where $m_{s,r}=m_{r,s}\in \Nat_{\geq 2}\cup\set{\infty}$   for $r\neq s$ in $S$ and $m_{r,r}=1$ for all $r\in S$. The associated Coxeter group $W$ is the group with presentation \begin{equation}
\label{eqa6.1.1} W=\mpair{S\mid (rs)^{m_{r,s}}=1, r,s\in S, m_{r,s}\neq \infty}.\end{equation} It is known that the natural map $S\rightarrow W$ is an inclusion; we always identify $S$ with a subset of $W$ via this map. It is also known that the order of $rs$ in $W$ is $m_{r,s}$ for any $r,s\in W$. The pair $(W,S)$ is called a Coxeter system (with Coxeter matrix $M$). 

\subsection{}\label{ssa6.2} It is convenient to collect for reference some of the  many equivalent characterizations of Coxeter systems. In formulating these,  the following  general framework will be used. 
\begin{conds}\item[($*$)]  $(W,S)$ is a pair consisting of a group $W$ and a set $S\seq W$ of involutions (elements of order exactly two) generating $W$ (thus, $1_{W}\not\in S$). \end{conds}
Define the length function $l=l_{S}\colon W\rightarrow \Nat$ of the pair $(W,S)$ and the subset $T:=\mset{wsw^{-1}\mid w\in W,s\in S}$ of $W$.
 
\subsection{Exchange condition}\label{ssa6.3}  Let   $(W,S)$ be a pair  satisfying \ref{ssa6.2}($*$). Then $(W,S)$  is said to satisfy the \emph{exchange condition} if for all $s_{1},\ldots, s_{n}\in S$ and $s_{0}\in S$ with $l(s_{0}s_{1}\cdots s_{n})\leq
 l(s_{1}\cdots s_{n})=n$, there is some $i$ with $s_{0}s_{1}\cdots s_{n}=s_{1}\ldots \hat s_{i}\cdots s_{n}$, where the term $\hat s_{i}$ is omitted from the product.
The \emph{strong exchange condition}  is  the same except requiring $s_{0}\in T$ instead of $s_{0}\in S$.

These conditions are often stated  in  equivalent versions with  the hypothesis that $l(s_{1}\cdots s_{n})=n$ omitted.
There are many other variants, including the following two: EC is the condition that 
for all $w\in W$, $r,s\in S$, with $l(wr)>l(w)$ but $l(swr)\leq l(sw)$, one has  $sw=wr$. Further, SEC is the condition 
that if $w\in W$, $r\in S$,  $t\in T$, $l(tw)\geq l(w)$ but $l(twr)\leq l(wr)$ then $tw=wr$.
The following  is well known. 
\begin{prop*}  If $(W,S)$ satisfies $\text{\rm \ref{ssa6.2}($*$)}$, the following conditions are equivalent:
\begin{conds}\item $(W,S)$  is a Coxeter system.
\item $(W,S)$ satisfies the exchange condition.
\item $(W,S)$ satisfies  the strong exchange condition.
\item $(W,S)$ satisfies EC.
\item $(W,S)$ satisfies SEC.\end{conds} \end{prop*}
\begin{proof} It is easy to see that (v)  implies (iv), (iii) and (ii).
The equivalence of (i) and (ii) is 
 in  \cite{Bour}. Finally, (i) implies (v) by a routine computation with the reflection cocycle of $(W,S)$ given in Remark \ref{ssa6.4}. \end{proof}

 \subsection{Reflection cocycle}\label{ssa6.4}  Let $(W,S)$ satisfy \ref{ssa6.2}($*$). Define the $W$-set
\begin{equation*}T=\mset{wsw^{-1}\mid w\in W,s\in S},\end{equation*} with (left) $W$-action by conjugation $(w,t)\mapsto wtw^{-1}$. Let $\L\colon W\to \setc$ denote the functor corresponding  to the $W$-set $T$ (regarding $W$ as one-object groupoid).  This gives a functor $\wp_{W}(\L)\colon W\to \bringc$, which affords the conjugacy  representation of  $W$ on $\wp(T)$. From \cite{DyRef}, one has the following.

 \begin{prop*} The pair 
$(W,S)$ is a Coxeter system if and only if there is a cocycle $N\colon W\rightarrow \wp_{W} (\L)$ such that $N(r)=\set{r}$ for all $r\in S$. In that case,  for all $ w\in W$, one has  $ N(w)=\mset{t\in T\mid l(tw)<l(w)}$, $l(tw)\equiv  l(w)+1 \pmod 2$ for all $t\in T$ and $l_{S}(w)=\vert N(w)\vert$.
\end{prop*}

If $(W,S)$ is a Coxeter system, $T$ is called the set of
 \emph{reflections}   and $N$ is called the \emph{reflection cocycle} of $(W,S)$. 
One then has   a set protorootoid $\CC'_{(W,S)}:=(W,\L,N)$ and its corresponding protorootoid $\CC_{(W,S)}:=\mathfrak{I}(\CC'_{(W,S)})=(W,\wp_{W}(\L),N)$.
\begin{rem*} To illustrate the usefulness of the Proposition for 
computations, we use it to  show that  a Coxeter system $(W,S)$ satisfies SEC. Suppose $l(tw)\geq l(w)$ but $l(twr)\leq l(wr)$ where $t\in T$, $r\in S$. Then
$t\not\in N(w)$ but \begin{equation*}
t\in N(wr)=N(w)+wN(r)w^{-1}=N(w)+\set{wrw^{-1}},
\end{equation*} so $t=wrw^{-1}$ and $tw=wr$.   This argument also shows that  
$l(wr)=l(w)+1$.
 \end{rem*}
\subsection{Root system}\label{ssa6.5}
 Using the generalities in Section \ref{sa5}, another  equivalent reformulation of the last  characterization of Coxeter systems is  as follows.\begin{prop*}
 The pair $(W,S)$ satisfying $\text{\rm \ref{ssa6.2}($*$)}$ is a Coxeter system if and only if there is an action of $W$ on the set $T\times \set{\pm }$ such that for $s\in S$, $t\in T$, $\epsilon\in \set{\pm }$
one has  \begin{equation}\label{eqa6.5.1} s(t,\epsilon)=(sts,\nu\epsilon), \qquad  \nu=\begin{cases}+,&\text{if $s=t$}\\
-,&\text{if $s\neq t$.}\end{cases}\end{equation}\end{prop*}

  It is known from \cite{Bour} (and follows also  from the Proposition, Proposition \ref{ssa6.4}  and Section  \ref{sa5}) that if  this action exists, it  is given explicitly by 
 \begin{equation} \label{eqa6.5.2}w(t,\epsilon)=(wtw^{-1},\nu\epsilon),\qquad  \nu=\begin{cases}+,&\text{if $l(wt)>l(w)$}\\
-,&\text{if $l(wt)<l(w)$.}\end{cases}\end{equation}

For any Coxeter system $(W,S)$, this gives a  a  $(W\times\set{\pm })$-set $\Phi:=T\times\set{\pm }$, called the \emph{abstract root system} of $(W,S)$, with $\set{\pm}$ action by multiplication on the right factor  and  positive roots $\Phi_{+}:=T\times\set{+}$.  For $\a=(t,\e)\in \Phi$, define the corresponding reflection  $s_{\a}:=t\in T$.  One then has \begin{equation}\label{eqa6.5.3}
N(w)=\mset{s_{\a}\mid \a\in \Phi_{w}}, \qquad \Phi_{w}:=\Phi_{+}\cap w(-\Phi_{+}).
\end{equation}

Viewing $W$ as one-object groupoid and  $\Phi$ as a functor $\Phi\colon W\to \setc_{\pm}$, define  $C_{(W,S)}:=(W,\Phi)$ regarded as object of $\grpdset$ and call it  the \emph{standard signed groupoid-set} of $(W,S)$. Comparing the above  with Section  \ref{sa5} shows that 
$\mathfrak{L}(\CC_{(W,S)})\cong C_{(W,S)}$.

 The set  
 $\Pi:=\mset{\a\in \Phi_{+}\mid s_{\a}\in S}$ is called the set of 
 \emph{simple roots} of
  $\Phi$.
   For any $J\seq S$, the subgroup $W_{J}:=\mpair{J}$  is called the \emph{standard parabolic subgroup} of $W$ generated by $J$  and $\Pi_{J}:=\mset{\a\in \Pi\mid s_{\a}\in J}$ is called the set of simple roots of $W_{J}$. 
  
\subsection{}\label{ssa6.6} This subsection contains informal remarks about   other, more standard,  notions of root systems of Coxeter groups. It is well-known that  Coxeter systems $(W,S)$   admit  
geometric realizations as  reflection groups associated to  root systems 
 $\Psi\seq V$ in  real vector spaces $V$.
Such  root systems will be loosely referred to  as \emph{linearly realized} root systems. Consider those  
 in   \cite[Ch 5]{Hum}, for example.  Note  that $\Psi$ is \emph{reduced} in the sense that the only scalar multiples of a root $\a$ which are  roots are $\pm \a$.
 Regard $\Psi$ 
as a  $W$-set  with commuting free action of the sign group $\set{\pm }$ by 
multiplication by $\pm 1$, and the standard  system of  positive roots  $\Psi_{+}$ as orbit representatives. 
 For each 
 $\a\in \Psi$, there is an associated 
  reflection $r_{\a}\in T\seq W$.
 There is an 
   isomorphism of $W\times\set{\pm  }$-sets  $\th\colon \Psi\mapsto \Phi=T\times\set{\pm }$ given by 
 $(\alpha,\e)\mapsto (r_{\a},\e)$ for $\a\in \Psi_{+}$,
  $\e\in \set{\pm }$; this holds since  the $W$-action on $T\times \set{\pm}$ is determined by \eqref{eqa6.5.1}  and there is an analogous formula for the $W$-action on $\Phi$ (which is easily checked using the well-known  fact that the only positive root  of $\Psi$ made negative by a simple reflection is its corresponding simple root). 
 The isomorphism maps $\Psi_{+}$ bijectively to $\Phi_{+}$ and 
 satisfies $r_{\th(\a)}=s_{\a}$ for $\a\in \Psi$.
 
  One may view $\Psi$ as a functor and   $(W,\Psi)$ as object of  $\grpdset$; as such, it is clearly   isomorphic to $C_{(W,S)}=(W,\Phi)$. 
 Similar remarks can be made for many other classes of  linearly realized root systems of Coxeter groups which are reduced in the above sense and  \emph{real} in that they do not contain imaginary roots  as  in the root systems of Kac-Moody Lie algebras.   In the case of a finite Weyl 
  group $W$, for example, one obtains a  signed 
  groupoid-set $(W,\Psi)$  isomorphic  to $(W,\Phi)$   from a reduced (crystallographic) root   system $\Psi$ of $(W,S)$ and a positive system $\Psi_{+}$, in   the sense of \cite{Bour}. 
   More generally still, the real compressions  (as informally described in the introduction, and defined in  subsequent papers) of linearly realized   root systems of Coxeter groups  in the literature  are   isomorphic to $(W,\Phi)$, whether they are real and  reduced or not.
   
 Many of the most important applications of Coxeter groups involve natural occurrences of  linearly realized root systems (e.g.   in the theory of semisimple 
complex Lie algebras) and existence of such root systems in general provides  powerful techniques   for 
the deeper  study  of Coxeter groups and related structures.   One expects that root systems 
in the abstract sense of these papers will not  have linear realizations in real vector spaces in a similar sense in general
(in the case of complete, principal, rootoidal signed groupoids-sets attached to non-realizable simplicial oriented geometries, for example). Some  results concerning   preservation of realizability   under   the main  constructions of these papers  will be  given in  subsequent papers.

 \subsection{Standard rootoid of a  Coxeter system}\label{ssa6.7}
   The standard protorootoids attached to  Coxeter systems are the motivating examples of   principal rootoids.
\begin{thm*} Let $(W,S)$ be a Coxeter system. \begin{num} \item The triple $\CC=\CC_{(W,S)}:=(W,\wp_{W} (\L),N)$ is    a   principal  rootoid  with $S$ as its simple generators.  It is complete if and only if  $W$ is finite.
\item  $\CC'=\CC'_{(W,S)}:=(W,\L,N)$ is a (principal) set rootoid. 
\item The   (principal, rootoidal) signed groupoid-set $\mathfrak{K}(\CC'_{(W,S)})$  is isomorphic to  the standard signed groupoid-set  $C_{(W,S)}:=(W,\Phi)$ of $(W,S)$.\end{num} \end{thm*}
\begin{proof}  
 It has already been noted  that   
 $\CC'=(W,\L,N)$ is a  set protorootoid, that  $\CC:=\mathfrak{I}(\CC')=(W,\wp (\L),N)$ is its corresponding   protorootoid and that $C_{(W,S)}\cong \mathfrak{K}(\CC')$ is the corresponding signed groupoid-set. So only (a) requires proof.
By Proposition \ref{ssa6.4}, \begin{equation} \label{eqa6.7.1}l(w)=\vert N(w)\vert,   \text{ for $w\in W$}\end{equation} where $l:=l_{S}$.
Therefore, $\vert N(w)\vert=0$ if and only if  $l(w)=0$ if and only if  $w=1_{W}$, since $S$ generates $W$. 
This implies that $\CC$ is faithful i.e. satisfies  \ref{ss4.3}(i). 
Note that the elements of the Boolean ring $\wp(T)$  of finite rank are the finite subsets, and their rank is their cardinality.
Therefore, \eqref{eqa6.7.1} implies that $S$ is the set of simple morphisms of $W$ (so $\CC$ is simply generated) and that $l_{S}=l_{N}$. This shows that $\CC$ is a faithful, principal protorootoid.
To show that $\CC$ is a   rootoid, it remains to prove that $\CC$ satisfies \ref{ss4.3}(ii)--(iii).

 The   weak right  
order $\leq$ of $\CC$ at the unique object of $W$ (as groupoid)    is, by definition, the partial order $\leq $ of  $W$ defined by
$x\leq y$ if and only if  $N(x)\seq N(y)$; this is  consistent with the usual definition of right  weak order of $W$ by \cite[Proposition 3.1.3]{BjBr} or Corollary \ref{ss3.14}.
It is well known that $(W,\leq)$ is a complete meet semilattice (see 
 \cite[Theorem 3.2.1]{BjBr}) i.e. \ref{ss4.3}(ii) holds.
The JOP (property \ref{ss4.3}(iii))  is not standard, but it (and also \ref{ss4.3}(ii)) is proved  in \cite{DyWeak}.  Hence $\CC$ is a rootoid.
  By
 definition, $\CC$ is complete if and only if  $W$ has a maximal element in weak right 
 order.
It is well known  that such a maximal  element exists if and only if  $W$ 
 is finite (in which case, it is the longest element of $W$ with respect to $l$).
 This completes the proof.\end{proof}
\begin{rem*} Given a  Coxeter groupoid $G$ with a (linearly realized, 
real) root system $\Phi$  (as defined in  \cite{HY} and \cite{CH}), one can   show 
similarly  that, regarding $\Phi$ just as signed $G$-set in the natural 
way, $(G,\Phi)$ is a principal, rootoidal  signed groupoid-set. 
(In fact, \cite{HV}  proves that the weak orders of the subclass of 
(finite) Weyl groupoids have complete ortholattices as their weak orders.) Also,  $(G,\Phi)$ is  complete if and only if   each component of $G$ is finite. One proof  uses results from \cite{HY} and \cite{CH}   along with  extensions to Coxeter 
groupoids of some of the results  in \cite{DyWeak} (note that no  extensions of the arguments in \cite{DyWeak} involving Bruhat order of $W$ are known). Another proof of this fact can be given along the lines of \ref{ssa6.10}. A  proof of a  more general fact will be given  in  subsequent papers.
\end{rem*}
\subsection{Reflection subgroups}\label{ssa6.8}
It is worth observing that inclusions of reflection subgroups of Coxeter groups give rise to morphisms in $\prootcp$, instead of  in  the category $\prootc$ with which   this series of papers is more concerned.  To see this, 
let $(W,S)$ be a Coxeter system, $\CR=(W,\L,N):=\CC'_{(W,S)}$ and $W'$ be a reflection subgroup of
$W$ i.e. a subgroup $W'$ of $W$ such that  $W'=\mpair{T'}$ where 
$T':=T\cap W'$. Let $\L'$ be the natural conjugacy representation of 
$W'$ on $T'$. Let  $i\colon W'\to W$  and $j\colon T'\to T$ denote 
the inclusion maps. There is  a natural transformation
$\nu\colon \L'\to \L i$ with $j$  as its unique component.
Then $\wp_{W'}(\nu)\colon \wp_{W}(\L)i\to \wp_{W'}(\L')$ is a natural transformation  with unique component $\wp(j)$ where  $\wp(j)(A):=j^{-1}(A)=A\cap T'$ for all $A\seq T$.  Let $N'\colon W\to \wp(T')$ be the composite $N':=\wp(j) N i$ i.e. $N'(w):=N(w)\cap T'$ for $w\in W'$.
The definitions  imply that $N'$ is a $W'$-cocycle for $\wp_{W'}(\L')$ i.e.  $N'\in Z^{1}(W',\wp_{W'}(\L'))$. 
Hence $\CR':=(W',\L',N')$  is a set protorootoid. It is easily checked that $(i,\wp_{W'}(\nu))
\colon (W',\wp_{W'}(\L'),N')\to (W,\wp_{W}(\L),N)$ is a morphism in $\prootcp$. 

The above-mentioned facts are all purely formal consequences of the 
definitions. According to \cite{DyRef}, there is a subset $S'$ of $W'$ such that  $(W',S')$ is a Coxeter 
system and $\CR'=\CC_{(W',S')}$ (necessarily, $S'=\mset{s\in W'\mid \vert N'(s)\vert=1}$).
\begin{rem*} From the above formula $N'(w):=N(w)\cap T'$ for $w\in W'$, it follows that the identity map on $W'$ is an order preserving map from $W'$, ordered by the restriction to $W'$ of weak order on $W$, to $W'$ in its weak order. The map is not in general an order isomorphism.  \end{rem*}

   \subsection{Semilocal criterion for rootoids}
   \label{ssa6.9}  The following result  is closely related to a  criterion in \cite{BjEZ} for a finite poset to be a lattice, and  
 its proof is very similar  in arrangement to  the argument from \cite[Section 4]{DyWeak} used to establish the JOP  in the proof of Theorem \ref{ssa6.7}. 
 \begin{prop*} Assume that  $\CR=(G,\L,N)$ is  an interval finite, faithful protorootoid. Let  $A:=A_{\CR}$ denote its set of atoms. Then
 $\CR$ is a rootoid if and only if  it  satisfies the following condition: whenever $a\in \ob(G)$, $r,s\in \lsub{a}{A}$ and $g\in \lsub{a}{G}$  are such that $ N_{r}\cap  N_{g}= N_{s}\cap  N_{g}=\eset$ and $ N_{r}, N_{s}$ have an upper bound in $ \lsub{a}{\CL}$, then the join  $ N_{r}\vee  N_{s}$ exists in $ \lsub{a}{\CL}$ and satisfies $( N_{r}\vee  N_{s})\cap  N_{g}=\eset$.\end{prop*}
\begin{rem*} The condition in the  proposition  will be called the semilocal criterion (SLC).  In this, the term ``semilocal'' is intended to suggest that the condition is not entirely local, in the sense of being  expressed purely in terms of the generators, but involves also general elements of the groupoid. No corresponding general local criterion is known. \end{rem*}

\begin{proof} Make the assumptions of the proposition. Then $A$ generates $G$, by Lemma \ref{ss3.5}. Validity of SLC   is clearly  necessary for $\CR$ to be a rootoid. Conversely, suppose that the condition holds.
The following statement will be proved: \begin{conds} \item[($*$)] If $a\in \ob(G)$ and  $x,y,g\in \lsub{a}{G}$  are such that $ N_{x}\cap  N_{g}= N_{y}\cap  N_{g}=\eset$ and $ N_{x}, N_{y}$ have an upper bound $ N_{u}$ in $ \lsub{a}{\CL}$, then the join  $ N_{x}\vee  N_{y}$ exists in $ \lsub{a}{\CL}$ and satisfies $( N_{x}\vee  N_{y})\cap  N_{g}=\eset$.\end{conds}
 
 Let us first check that  ($*$)   would imply that $\CR$ is a rootoid.
 Note that (taking $g=1_{a}$), ($*$) implies that if elements  $x$, $y$ of $\lsub{a}{\leq  }$ have an upper bound, they have a least upper bound.
 This in turn implies that $\lsub{a}{\leq  }$ is a complete meet semilattice. For given a non-empty subset $X$ of $\lsub{a}{G}$, with say $x\in X$, $\meet X=\join Y$ 
 where $Y$ is the set of lower bounds of $X$; the join of $Y$ exists from the above  since $Y$ is bounded above by $x$ (and hence  in particular is finite by interval finiteness of $\CR$). This shows that ($*$) implies that $\CR$ satisfies \ref{ss4.3}(ii).  Note also that if a subset $X$ of $\lsub{a}{G}$ is bounded above, then, 
  writing $X=\set{x_{1},\ldots, x_{n}}$, one has  $ \vee_i{x_{i}}=x_{1}\vee (x_{2}\vee(x_{3}\vee \ldots (x_{n-1}\vee x_{n})))$. From this,  it follows inductively  that ($*$) implies that $\CR$ satisfies JOP, and hence is a rootoid.

  Now we prove ($*$) by induction on the cardinality $n(u):=\vert [1_{a},u]\vert$ of the interval $[1_{a},u]$ in $\lsub{a}{\leq  }$.
The statement is  trivial if 
  $u=1_{a}$ or $x=1_{a}$ or $y=1_{a}$. Assume inductively that the statement holds if $n(u)<m$. Suppose that $n(u)=m$, $x\neq 1_{a}$ and $y\neq 1_{a}$. Choose $r,s\in \lsub{a}{A}$ so $1_{a}<r\leq x$ and $1_{a}< s\leq y$. Since $u$ is an upper bound of $r$, $s$ in $\lsub{a}{G}$, $z:= r\vee s$ exists  and satisfies $z\leq u$, $ N_{z}\cap  N_{g}=\eset$.
  
  Since $s\leq z$, we have $ N_{z}= N_{s}\dotcup s( N_{s^{*}z})$.
  Since $ N_{s}\cap  N_{g}\seq  N_{y}\cap  N_{g}=\eset$, we have $ N_{s^{*}g}= N_{s^{*}}\dotcup s^{*}( N_{g})$. It follows that $ N_{s^{*}z}\cap  N_{s^{*}g}=\eset$.
  Similarly,  $ N_{s^{*}y}\cap  N_{s^{*}g}=\eset$. Let $s\in \lrsub{a}{A}{b}$. Now we use Lemma \ref{ss4.4}.
  In weak order on $\lsub{b}{G}$, 
  $s^{*}u$ is an upper bound of  $s^{*}y$ and $s^{*}z$. Also, $[1_{b},s^{*}u]=s^{*}[s,u]\sneq s^{*}[1,u]$ so $n(s^{*}u)<n(u)$. By induction, there is $w\in \lsub{a}{G}$ such that
  $s^{*}z\vee s^{*}y=s^{*}w$ and $ N_{s^{*}w}\cap  N_{s^{*}g}=\eset$. In particular, $ N_{s^{*}w}\cap  N_{s^{*}}=\eset$, so $ N_{w}= N_{s}\dotcup s( N_{s^{*}w})$.
  It follows by Lemma \ref{ss4.4} that $z\vee y=w$ in $ \lsub{a}{G}$ (since the join $z\vee y$ is the same whether calculated  in $\lsub{a}{G}$ or in the  principal order coideal of $\lsub{a}{G}$ generated by $s$). Also, it is easily seen from above that $ N_{w}\cap  N_{g}=\eset$.   
 
  A similar argument to that in the last paragraph,   with $s$ replaced by $r$,
  $z$ replaced by $w$, and $y$ replaced by $x$, shows that $v:=w\vee x$ exists in $\lsub{a}{G}$  and satisfies $ N_{v}\cap  N_{g}=\eset$.
  The proof is completed by noting that  \begin{equation*} v=x\vee w=x\vee(z\vee y)=x\vee(r\vee s)\vee y=(x\vee r)\vee(s\vee y)=x\vee y.\qedhere\end{equation*}
 \end{proof}
   \subsection{Another proof of Theorem \ref{ssa6.7}} 
   \label{ssa6.10}
The  SLC can be used to give an alternative   proof that $\CC_{(W,S)}$ is a rootoid, with the advantage of being  self-contained except for well-known properties of shortest coset representatives of (rank two) standard parabolic subgroups.  
The argument at the beginning of the proof of Theorem \ref{ssa6.7} implies that $\CC$ is faithful and principal.  It is interval finite by Lemma  \ref{ss3.6}, and $A=S$ by Lemma \ref{ss3.7}. As in the proof of Theorem \ref{ssa6.7},
it is  possible to use the description of weak right order of $\CC$ in terms of $l$ given by Corollary \ref{ss3.14}.

To show that $\CC$ is a rootoid, it will suffice to verify the SLC. In terms of $l$, SLC requires that if $r,s\in S$ and $g\in W$ with $l(r^{-1}g)>l(g)$ and  $l(s^{-1}g)>l(g)$ and $r,s$ have 
 have an upper bound $x$ in weak right order  on $W$, then the join $y=r\vee s$ exists in weak right order on $W$ and satisfies $l(y^{-1}g)=l(y)+l(g)$.
Now  $l(r^{-1}x)<l(x)$ and $l(s^{-1}x)<l(x)$. Let $x'\in W_{J}x$, where $J:=\set{r,s}$, be the shortest coset representative in $W_{J}x$ i.e.  $x'\in W_{J}x$ and $l(yx')=l(y)+l(x')$ for all $y\in W_{J}$. Write 
$x=yx'$ where $y\in W_{J}$. Then $l(s^{-1}y)<l(y)$ and $l(r^{-1}y)<l(y)$, which implies that $W_{J}$ is finite and  $y$ is the longest element of $W_{J}$. It is easy to check that  $y=r\vee s$. The assumptions imply that $g$ is the shortest coset representative in  $W_{J}g$.
Hence $l(y^{-1}g)=l(y)+l(g)$ as required. This implies that $\CC$ is a  principal rootoid, with simple generators $S$. The proof of the criterion in Theorem \ref{ssa6.7} for completeness of $\CC$ is as in the earlier proof.

 \subsection{Rootoids from simplicial  hyperplane arrangements}
\label{ssa6.11}
Let $V$ be a real Euclidean space with inner product $\mpair{-,-}$. An arrangement  $\CH$ of hyperplanes in $V$ is  a  set of affine hyperplanes in $V$. (For background on hyperplane arrangements needed here, a convenient source is  
\cite{BjEZ}, which we shall  follow  as regards terminology and notation since the main result of this subsection is essentially a direct translation into protorootoid terminology of a main result of that paper)   The arrangement  $\CH$ is said to be central if  all the hyperplanes in $\CH$ are linear hyperplanes  (i.e contain $0$) and to be   essential if their normals span $V$. 

Fix a finite, essential, central hyperplane arrangement $\CH$ in $V$, with $V\neq 0$ to avoid trivialities. 
The connected components of $V\setminus \cup_{H\in \CH}H$ (in the 
metric topology induced by the inner product $\mpair{-,-}$) are called 
the chambers of $\CH$.  One says that  two chambers $D, E$ are 
separated by the hyperplane  $H\in \CH$ if they are in different 
connected components of $V\sm H$. Two chambers are said to be adjacent if they are separated by a unique hyperplane $H\in \CH$.

There are finitely many chambers. Consider the (unique up to isomorphism) connected, simply connected  groupoid $G$ with the chambers of $\CH$ as its objects.
For two chambers $D,E$, let $f_{E,D}\colon D\rightarrow E$ be the unique morphism in $G$ from $D$ to $E$.

Define an object   $R=C_{\CH}:=(G,\Phi)$ of $\grpdc\text{\rm-} \setc_{\pm}$ as follows. Let  $U$
 be the set of all unit normal vectors to hyperplanes $H\in \CH$, regarded as indefinitely signed set 
with action of the sign group $\set{\pm }$ by multiplication.
 Next,   define a definitely signed set 
  $\lsub{D}{ \Phi}$ for each chamber $D$, with underlying indefinitely signed set 
  $U$ and with  \begin{equation} \label{eqa6.10.1}\lrsub{D} {\Phi}{+}:=\mset{u\in \lsub{D} {\Phi}\mid \mpair{u,D}\subseteq \real_{>0}}.\end{equation} The functor $\Phi\colon G\to \setc_{\pm}$ is defined by
  setting $\Phi(D):=\lsub{D}{\Phi}$ as definitely signed set  and requiring that the composite of  $\Phi$ with the forgetful functor $\setc_{\pm}\to \setc_{\set{\pm}}$ be equal to the constant functor $G\to \setc_{\set{\pm}}$ with  value $U$. This completes the definition of $R=C_{\CH}$.
 Define the associated set protorootoid $\CC_{\CH}':=\mathfrak{L}(R)$ and protorootoid $\CC_{\CH}:=\mathfrak{I}(\mathfrak{L}(R))$

\begin{thm*} The signed groupoid-set $R=C_{\CH}$ of a finite, central, essential hyperplane arrangement $\CH$ in the real Euclidean space $V$ is  rootoidal if and only if  $\CH$ is a simplicial arrangement i.e. if and only if  each chamber is an open simplicial cone in $V$. In that case, $R$  is  complete and  principal and the simple generators of the underlying groupoid are the morphisms between adjacent chambers.
\end{thm*}
\begin{proof} For $E,D\in \Phi$, $\Phi_{f_{E,D}}$ is the set of all
unit normals in $\lrsub{E}{\Phi}{+} $  to hyperplanes  in $\CH$ which separate
$E$ and $D$. Thus, the cardinality $\vert \Phi_{f_{E,D}}\vert $ is the number of hyperplanes in $\CH$ separating $D$ and $D$, which is called the \emph{combinatorial distance} between $E$ and $D$. The combinatorial distance is $0$ if and only if  $E=D$, so $R$ is strongly   faithful  i.e $R$ satisfies
 \ref{ssa5.7}(i).
Now  $D\lsub{E}{\leq  } D'$ in the right weak order $\lsub{E}{\leq  }$ on $\lsub{E}{G}$ if and only if  every hyperplane in $\CH$ separating $E$ and $D$ also separates $E$ and $D'$. Thus, the right weak order   $\lsub{E}{\leq  }$ on $\lsub{E}{G}$ is just
  the poset of regions (chambers) of $\CH$ ordered  by combinatorial distance from  the base chamber $E$, as  studied in \cite{BjEZ}.
  
   Let $S$  be the set of all simple morphisms  in $\mor( G) $ i.e. the set of all $f_{E,D}$ such that $\vert \Phi_{f_{E,D}}\vert=1 $.     Restating basic facts from \cite{BjEZ} in the terminology here,  $S$ is a set of groupoid generators of $G$ and 
  \begin{equation}\label{eq:6.11.2}  \vert \Phi_{g}\vert =l_{S}(g), \text{ \rm for $g\in \wh G$}.
  \end{equation} Hence $R$ is a principal signed groupoid-set in general.
  
  For regions $E,D$ of $\CH$, note that   $-D$ is a region and 
  \begin{equation}\label{eq:6.11.3} \Phi_{f_{E,-D}}=\lrsub{E}{\Phi}{+} \sm\Phi_{f_{E,D}}.\end{equation} 
 To prove that  $R$ satisfies the JOP  (i.e. \ref{ss4.3}(iii)),
  it is necessary to show that if $E$, $A_{i}$ and $A$ are chambers such that
  $\Phi_{f_{E,A_{i}}}\cap \Phi_{f_{E,A}}=\eset$ for all $i$ and 
  $f_{E,B}:=\join_{i}f_{E,A_{i}}$ exists in $(\lsub{E}{G},\lsub{E}{\leq})$ then
  $\Phi_{f_{E,B}}\cap \Phi_{f_{E,A}}=\eset$. But under those assumptions,  
  $\Phi_{f_{E,A_{i}}}\seq \Phi_{f_{E,-A}}$ for all $i$ so by definition of join,
  $\Phi_{f_{E,B}}\seq  \Phi_{f_{E,-A}}$ and therefore $\Phi_{f_{E,B}}\cap \Phi_{f_{E,A}}=\eset$ (cf. Lemma  \ref{ss4.1}).
  
  From the above, it now  follows that $R$ is rootoidal  if and only if it satisfies \ref{ss4.3}(ii) or equivalently 
   if and only if  for every  chamber, the poset of 
   regions of $\CH$ oriented from that chamber is a complete  semilattice. The 
  weak order $ \lsub{E}{\leq  } $ at $E$ is finite, and  clearly  has a maximum   element $-E$, so it is a complete meet semilattice if and only if it is a lattice.
     By 
   \cite{BjEZ}, $ \lsub{E}{\leq  } $ is a lattice  for all $E$  if and only if 
  $\CH$ is simplicial.  The argument shows  that if $R$ is  rootoidal,   it  is complete and principal.     \end{proof}
    \begin{rem*}  Let $W$ of a finite Coxeter group $W$ acting (with no pointwise fixed subspace of positive dimension)   as reflection group on a  Euclidean space. Consider  the reflection arrangement $\CH$, which consists of the reflecting hyperplanes of $W$.         This arrangement  is well known to be simplicial.
    Fix a fundamental chamber $C$  of the arrangement; the set  $S$ of reflections  of  $W$ in walls of $C$ makes $(W,S)$ a Coxeter system.
 The simply connected groupoid underlying  $R_{\CH}$ is  canonically isomorphic to the universal  covering groupoid of  $W$   (using the natural bijection between chambers and elements of $W$ afforded by the choice of $C$).     One can check that this groupoid  isomorphism underlies an isomorphism of the  rootoid $\CC_{\CH}$  attached to $\CH$ with  the universal covering rootoid   of  $\CC_{(W,S)}$.\end{rem*}

  \subsection{Other examples of rootoids} \label{ssa6.12}
 To partly offset any misleading impression due to the fact that  the preceding examples were all principal rootoids,   this  subsection describes some non-principal rootoids of a quite different character. \begin{exmp*}    
    (1) Let $G$ denote the additive group $\real$  with standard partial order.  Let $P$ denote the cone of non-negative elements 
    $P:=\mset{x\in G\mid x\geq 0}$.  Denote addition in $G$ by $+$ and addition 
    (symmetric difference)  in the Boolean ring $\wp(G)$ by $\dotplus$ to avoid 
    confusion. Attached to $(G,P)$ there is a protorootoid $\CR:=(G,\L,N)$    as follows.  The 
      functor  $\L\colon G\to\bringc$ gives the translation action of $G$ on subsets of $G$; formally,  $\L$ sends the unique object of $G$ to $\wp(G)$ 
  and \begin{equation*} \L(\l)(A):=A+\l=\mset{a+\l\mid a\in A}\end{equation*} for any morphism  $\l\in G$ and  any $A\in \wp(G)$.  The 
  cocycle $N$ is the coboundary
  $N\in B^{1}(G,\L)$ defined  by $N(\l)=(\l+P)\dotplus  P$.
  Note that either $N(\l)$ is empty, or it is a closed-open interval in $\real$ with $0$ as either the (closed)  left  endpoint or  (open) right endpoint.
   It  is straightforward to check that  $\CR$ is a rootoid, which is not interval finite.   The unique weak right order is  the partial order $\preceq$ of $\real$ such that $\lambda\preceq \mu$ if $\lambda\mu\geq 0$ and $\vert \lambda\vert\leq \vert \mu \vert$,  where $\vert \nu\vert$ denotes the absolute value of $\nu\in \real$.

  (2)  Let $V$ be a real Euclidean space.
 Let $G=\text{\rm O}(V)$  denote the corresponding real orthogonal group.
 A vector space total ordering of $V$ is given by the lexicographic ordering on coordinate vectors with respect to a chosen ordered orthonormal basis  of $V$. The natural action of $G$ on the set of  vector space total orderings $\leq$ of $V$ is simply transitive. Fix such an ordering $\leq$. Let $\mathbb{S}:=\mset{v\in V\mid \mpair{v,v}=1}$ denote the unit sphere in $V$ and   $\mathbb{S}_{+}:=\mset{v\in  \mathbb{S}\mid v>0}$. 
 This makes $\mathbb{S}$ into  a definitely signed set.
The group  $G$ and sign group $\set{\pm 1}$ have natural commuting actions  on    $\mathbb{S}$.   Let $\Phi\colon G\to \setc_{\pm}$ be the representation of $G$
 taking the value $\mathbb{S}$  at the unique object of $G$, and  with the natural  action of $G$ on the unit sphere as underlying representation of $G$ in $\setc$.
 This defines  a signed groupoid-set 
  $R=C_{\text{\rm O}(V)}:=(G,\Phi)$. Let $\CC_{\text{\rm O}(V)}:=\mathfrak{L}(R)$ and $\CC_{\text{\rm O}(V)}':=\mathfrak{I}(\mathfrak{L}(R))$ 
   denote the associated set protorootoid and protorootoid, respectively. 
    It is shown in \cite{Pilk} that $C_{\text{\rm O}(V)}$ is complete and rootoidal.  
  
    For example, take $V=\real^{2}$ with dot product. Then   
 $\mathbb{S}=\mset{(x,y)\in \real^{2}\mid x^{2}+y^{2}=1}$ and  one can take (for instance) 
 \begin{equation*} \mathbb{S}_{+}:=\mset{(x,y)\in \mathbb{S}\mid y>0 \text{ or ($y=0$ and $x>0$)}}.\end{equation*} 
 Define a bijection $f\colon [0,\pi)\to\mathbb{S}_{+}$ by
 $f(t)=(\cos(t),\sin(t))$. The  sets
 $\Phi_{g}$ for $g\in \text{\rm O}(V)$ are $\eset$, $\mathbb{S}_{+}$,
 $\set{(1,0)}$, $ \mathbb{S}_{+}\sm\set{(1,0)}$
 and  the (open, closed or open-closed) arcs $f(I)$ where  $I$ is a subinterval of $[0,\pi)$ of the form
 $[0,t)$, $[0,t]$, $(t,\pi)$ or  $[t,\pi)$ for some $t\in(0,\pi)$. 
  It is easy  to check from this that   
$C_{\text{\rm O}(V)}$  is a  complete,  regular, saturated, pseudoprincipal, rootoidal signed groupoid-set.   Clearly $C_{\text{\rm O}(V)}$ is not interval finite. The   orthogonal  reflection which  fixes the $y$-axis pointwise is the unique simple morphism. 
 
 \end{exmp*}

   \section{Complete rootoids}\label{sa7}

 \ssect{} \label{ssa7.1}
 The following proposition gives analogues of standard properties of longest elements of finite Coxeter systems.
 Note that the abridgement of any faithful, complete protorootoid satisfies the hypotheses.  \begin{prop*} Suppose that $\CR=(G,\L,N)$ is a  faithful, complete, abridged protorootoid. Let $a\in \ob(G)$. \begin{num}
 \item  The weak order    $\lsub{a}{\leq  }$ has a maximum element $\o(a)$. It satisfies  $\o(a)=\join \lsub{a}{G}$ in $\lsub{a}{G}$  and $N(\o(a))=\join \lsub{a}{\CL}$ in $\lsub{a}{\CL}$.  
  \item  $e_{a}:=N(\o(a))$ is an identity element of  (the Boolean algebra) $\lsub{a}{\L}$. 
 \item   $\CR$ is a  unitary protorootoid.
  \item Let $a':=\domn(\o(a))\in \ob(G)$, so  that $\omega(a)\in \lrsub{a}{G}{a'}$. Then $\o(a)^{*}=\o(a')$.
 \item The map $h\mapsto \o(a)h\colon \lsub{a'}{G}\to \lsub{a}{G}$ is an order anti-isomorphism in the weak right orders.
  \item For $g\in \lrsub{b}{G}{a}$, $N(g\o(a))=N(g)^{\cp}:=e_{b}+N(g)$ in 
  $\lsub{b}{\L}$.
 \item $\lsub{a}{\CL}$ is a complete ortholattice   with orthocomplement  given by restriction of  complement 
 $A\mapsto A^{\cp}$ in the Boolean algebra $\lsub{a}{\L}$.
 \item $\CR$ is a complemented, complete rootoid.
 \end{num}
\end{prop*}
\begin{proof} Note that $\lsub{a}{\CL}$ has a maximum element 
by definition,  because $\CR$ is complete. Since $\CR$ is faithful, 
there is an order isomorphism 
$x\mapsto N(x)\colon \lsub{a}{\CL}\cong\lsub{a}{G}$, and the rest 
of (a) follows trivially. From (a), for all $g\in \lsub{a}{G}$ one has 
 $N(g)\seq N(\o(a))$  so $N(g)\cap N(\o(a))=N(g)$. Since $\CR$ is 
 abridged, $\lsub{a}{\L}$ is generated as ring by $\lsub{a}{\CL}$ 
 and (b) follows.
 
 For any $g\in \lrsub{a}{G}{b}$, $\L(g)\colon \L(b)\to \L(a)$ is a  
 homomorphism of Boolean rings. Hence it is order preserving.
 Since it is bijective,  with inverse $\L(g^{*})$, it must preserve 
 maximum elements i.e. it is a homomorphism of Boolean algebras. This 
 proves (c).
 
 Next, observe that by (b) and (c), 
  \begin{equation*} N(\o(a)\o(a'))=N(\o(a))+\o(a)N(\o(a'))=e_{a}+\o(a)(e_{a'})
  =e_{a}+e_{a}=0=N(1_{a}).\end{equation*}
 Since $\CR$ is faithful, $\o(a)\o(a')=1_{a}$ and (d) follows.

For any $h\in \lsub{a'}{G}$, one has $N(\o(a)h)=N(\o(a))+\o(a)(N
(h))=\bigl(\o(a)N(h)\bigr)^{\cp}$. This immediately implies that the 
map in (e) is weak right order reversing. Then (e) follows 
 since, by symmetry,  an inverse order anti-isomorphism is given 
 by $k\mapsto \o(a')k\colon \lsub{a}{G}\to \lsub{a'}{G}'  $.
 
 For $g$ as in (f), the cocycle condition gives 
 \begin{equation*}
 N(g\o(a))=N(g)+gN(\o(a))=N(g)+g(e_{a})=N(g)+e_{b}=N(g)^{\cp}
 \end{equation*} which proves (f). For $A\in \lsub{b}{\CL}$, write 
 $A=N(g)$ where $g\in \lrsub{b}{G}{a}$. Then by (f), $A^{\cp}=N
 (g\o(a))\in \lsub{b}{\CL}$. This easily implies (g).  For (h), note  that $(\lsub{a}{\L},\lsub{a}{\CL})$ is a complemented protomesh, by (g). Therefore, JOP follows by Lemma \ref{ss4.1}(a) and $\CR$ is a rootoid.  Also by  (g),  $\CR$ is complete and complemented, proving (h). 
 \end{proof}
\ssect{} \label{ssa7.2} Proposition \ref{ssa7.1}  applies  to the abridgement $\CR^{a}$ of any  complete rootoid $\CR$.
The next  corollary describes, for such $\CR$, the analogue of the  automorphism of a Coxeter system defined  by conjugation by the longest element.

\begin{cor*} Let $\CR=(G,\L,N)$ be a complete  rootoid.
\begin{num}\item There is a unique functor
$d\colon G\to G$ such that there is  natural isomorphism $\o\colon d\to \Id_{G}$ with component  $\o_{a}=\o(a)$ at each $a\in \ob(G)$.  \item Let $D:=(d,\L\o^{*})$. Then $D\in \Aut_{\rlec}(\CR)$ and $D^{2}=\Id_{\CR}$.

\end{num}\end{cor*}\begin{proof}  For $a\in \ob(G)$, let $d(a):=a'=\domn(\o(a))$ where $\o(a),a'$ are as defined in Proposition \ref{ssa7.1} for $\CR^{a}$. For $f\in \lrsub{a}{G}{b}$, define $d(f)\in \lrsub{d(a)}{G}{d(b)}$ by $d(f):=\o(a)^{*}f\o(b)$.  It is easy to check that this defines the
unique  functor as required in (a). By Remark \ref{ss4.9}, it is sufficient to check (b) with $\rlec$ replaced by $\prootc$.
 This  is done by  a straightforward calculation using  the definitions and  Proposition \ref{ssa7.1}. \end{proof}

\ssect{}\label{ssa7.3} A protorootoid $\CR=(G,\L,N)$ is said to be \emph{finite} if $\mor(G)$ is finite (and hence  $\ob(G)$ is finite) and $\lsub{a}{\L}$ is a finite Boolean ring for all $a\in \ob(G)$. \begin{cor*} Let $\CR=(G,\L,N)$ be a rootoid.
\begin{num} \item  If $\CR$ is interval finite, complete and connected, then $G$ is finite.
\item  If $\CR$ is cocycle finite, complete, abridged and connected, then it is finite.
\item  If $\CR$ is principal, complete, abridged and connected, then it is finite.
\end{num}
\end{cor*}
\begin{proof} If   $\ob(G)= \eset$, this is trivial, so assume that 
$\ob(G)\neq  \eset$. Let $a\in \ob(G)$ and assumptions be as in (a). Then $\lsub{a}{G}=[1_{a},\o(a)]_{\lsub{a}{\CL}}$ is finite. Since $G$ is connected, 
for any morphism $f\colon b\to c$ in $G$, one may write $f=g^{*}h$ for some $g,h\in \lsub{a}{G}$ so $\mor(G)$ is finite, proving (a). For (b), note that $\CR$ cocycle finite implies that $\CR$ is interval finite, so $G$ is finite by (a).  For $a\in \ob(G)$, 
$\lsub{a}{\L}=[0,e_{a}]_{\lsub{a}{\L}}=[0,N(\o(a))]_{\lsub{a}{\L}}$
which is finite  since $\CR$ is cocycle finite. This proves (b).
Finally, (c) follows from (b) since if $\CR$ is  principal, it is   cocycle finite.
\end{proof}
 

\end{document}